\documentclass[11pt,reqno]{amsart}
\usepackage[english]{babel} \usepackage[latin1]{inputenc} \usepackage[T1]{fontenc}
\usepackage{amsmath,amssymb,amsfonts,amsthm,amscd}
\allowdisplaybreaks %long equations
\usepackage{fontenc,indentfirst,delarray}
\usepackage{mathrsfs}

\usepackage{etex}\usepackage{tikz}\usetikzlibrary{matrix,arrows, positioning}
\usepackage{pictexwd, dcpic} %
\usepackage{enumitem}%cross-reference enumerated items with coherent names
\usepackage{cancel, color}%, bbold
\usepackage{wasysym, verbatim, stmaryrd, arydshln, esint}
\usepackage[normalem]{ulem}

\makeatletter
\newcommand{\labitem}[2]{%
\def\@itemlabel{\textbf{#1}}
\item
\def\@currentlabel{#1}\label{#2}}
\makeatother

\numberwithin{equation}{section}

\def\sideremark#1{\ifvmode\leavevmode\fi\vadjust{\vbox to0pt{\vss% the remark
 \hbox to 0pt{\hskip\hsize\hskip1em%                          will appear only
 \vbox{\hsize3cm\tiny\raggedright\pretolerance10000%          on the side
 \noindent #1\hfill}\hss}\vbox to8pt{\vfil}\vss}}}%

\newcommand{\edz}[1]{\sideremark{#1}}

\mathchardef\za="710B  %\alpha
\newcommand{\bza}{\boldsymbol{\alpha}}  %bold alpha
\mathchardef\zb="710C  %\beta
\mathchardef\zg="710D  %\gamma
\mathchardef\zd="710E  %\delta
\mathchardef\zve="710F %\epsilon
\mathchardef\zz="7110  %\zeta
\mathchardef\zh="7111  %\eta
\mathchardef\zvy="7112 %\theta
\mathchardef\zi="7113  %\iota
\mathchardef\zk="7114  %\kappa
\mathchardef\zl="7115  %\lambda
\mathchardef\zm="7116  %\mu
\mathchardef\zn="7117  %\nu
\mathchardef\zx="7118  %\xi
\mathchardef\zp="7119  %\pi
\mathchardef\zr="711A  %\rho
\mathchardef\zs="711B  %\sigma
\mathchardef\zt="711C  %\tau
\mathchardef\zu="711D  %\upsilon
\mathchardef\zvf="711E %\phi
\mathchardef\zq="711F  %\chi
\mathchardef\zc="7120  %\psi
\mathchardef\zw="7121  %\omega
\mathchardef\ze="7122  %\varepsilon
\mathchardef\zy="7123  %\vartheta
\mathchardef\zvw="7124  %\varomega
\mathchardef\zvr="7125 %\varrho
\mathchardef\zvs="7126 %\varsigma
\mathchardef\zf="7127  %\varphi
\mathchardef\zG="7000  %\Gamma
\mathchardef\zD="7001  %\Delta
\mathchardef\zY="7002  %\Theta
\mathchardef\zL="7003  %\Lambda
\mathchardef\zX="7004  %\Xi
\mathchardef\zP="7005  %\Pi
\mathchardef\zS="7006  %\Sigma
\mathchardef\zU="7007  %\Upsilon
\mathchardef\zF="7008  %\Phi
\mathchardef\zW="700A  %\Omega
\newcommand{\zvk}{\varkappa}  %\varkappa

\newcommand{\cA}{\mathcal{A}}
\newcommand{\cuA}{\underline{\mathcal{A}}}
\newcommand{\cB}{\mathcal{B}}
\newcommand{\cuB}{\underline{\mathcal{B}}}

\newcommand{\cD}{\mathcal{D}}

\newcommand{\cI}{\mathcal{I}}
\newcommand{\cJ}{\mathcal{J}}

\newcommand{\cL}{\mathcal{L}}

\newcommand{\cX}{\mathcal{X}}
\newcommand{\cY}{\mathcal{Y}}

\newcommand{\f}{\mathfrak}

\newtheorem{thm}{Theorem}[section]

\newtheorem{thmintro}{Theorem}
\newtheorem{prop}[thm]{Proposition}
\newtheorem{lem}[thm]{Lemma}
\newtheorem{cor}[thm]{Corollary}
\newtheorem{defi}[thm]{Definition}
\newtheorem{exa}[thm]{Example}

\newtheorem{rema}[thm]{Remark}

\newcommand{\be}{\begin{equation}}
\newcommand{\ee}{\end{equation}}
\newcommand{\bea}{\begin{eqnarray}}
\newcommand{\eea}{\end{eqnarray}}
\newcommand{\beas}{\begin{eqnarray*}}
\newcommand{\eeas}{\end{eqnarray*}}
\newcommand{\bemap}{\be \begin{array}{lccc}}
\newcommand{\eemap}{\end{array}\ee}
\newcommand{\bemaps}{\[ \begin{array}{lccc}}
\newcommand{\eemaps}{\end{array}\]}

\newcommand{\N}{\mathbb{N}}
\newcommand{\Z}{\mathbb{Z}}
\newcommand{\R}{\mathbb{R}}
\newcommand{\K}{\mathbb{K}}
\newcommand{\C}{\mathbb{C}}

\newcommand{\qH}{\mathbb{H}}
\DeclareMathOperator{\qi}{i}
\DeclareMathOperator{\qj}{j}
\DeclareMathOperator{\qk}{k}
\newcommand{\la}{\langle}
\newcommand{\ra}{\rangle}
\newcommand{\lp}{\left(}
\newcommand{\rp}{\right)}

\newcommand{\raa}{\rightarrow}

\newcommand{\op}[1]{\!\!\mathop{\rm ~#1}\nolimits}

\DeclareMathOperator{\id}{id}
\DeclareMathOperator{\Id}{Id}

\DeclareMathOperator{\Obj}{Ob}

\newcommand{\I}{\mathbb{I}}%id matrix

\newcommand{\pari}{\bar}
\newcommand{\degr}{\widetilde}
\newcommand{\evp}{_{\hspace{-0.1cm}\phantom{.}_{\overline{0\!}}}} % even part
\newcommand{\odp}{_{\hspace{-0.1cm}\phantom{.}_{\overline{1\!}}}} % odd part

\newcommand{\ks}[2]{(-1)^{\la \;\ensuremath{#1},\; \ensuremath{#2}\, \ra}}
\DeclareMathOperator{\Hom}{Hom}

\DeclareMathOperator{\Aut}{Aut}
\DeclareMathOperator{\iHom}{\mathcal{H}\!\mathit{om}}
\DeclareMathOperator{\iEnd}{\mathcal{E}\!\!\;\textit{nd}}

\DeclareMathOperator{\Mat}{\mathsf{M}}
\DeclareMathOperator{\GL}{\mathsf{GL}}
\DeclareMathOperator{\hGL}{\mathsf{hGL}}
\DeclareMathOperator{\str}{str}

\DeclareMathOperator{\Tr}{\zG tr}
\DeclareMathOperator{\Ddet}{D\!\det}
\DeclareMathOperator{\Gdet}{\zG\!\!\det}

\DeclareMathOperator{\ber}{Ber}
\DeclareMathOperator{\Gber}{\zG\!\ber}

\newcommand{\zld}[2]{\zl\big(\, \degr{\ensuremath{#1}}, \degr{\ensuremath{#2}}\, )}%%
\newcommand{\zlsup}{\zl^{\op{super}}}
\newcommand{\zls}{\zl^{\zvs}}
\newcommand{\fS}{\mathfrak{S}}

\newcommand*{\low}[1]{\raisebox{-.25ex}{\ensuremath{{#1}}}}

\newcommand*{\EnsQuot}[2]%
{\ensuremath{%
    #1/\!\raisebox{-.65ex}{\ensuremath{{#2}}}}} %quotient
\newcommand{\vect}{\mathbf} %vectors (bold)
\newcommand{\ts}{\times}

\newcommand{\opl}{\oplus}

\newcommand{\ots}{\otimes}
\newcommand{\bots}{\low{\boxtimes}\,}

\newcommand{\w}{\wedge}

\newcommand{\Top}{\rule{0pt}{3ex}}
\newcommand{\und}[1]{\underline{\ensuremath{#1}\!}\,}

\newcommand{\cat}{\mathtt}
\newcommand{\CC}{\mathtt{C}}
\newcommand{\DD}{\mathtt{D}}
\newcommand{\gAlg}{\zG\mbox{-}\mathtt{Alg}}
\newcommand{\glAlg}{(\zG, \zl)\mbox{-}\mathtt{Alg}}
\newcommand{\glsAlg}{(\zG, \zls)\mbox{-}\mathtt{Alg}}
\newcommand{\sAlg}{\mathtt{SAlg}}

\newcommand{\sMod}{\mathtt{SMod}}
\newcommand{\Moda}{\zG\mbox{-}\mathtt{Mod}_{\cA}}
\newcommand{\Modabar}{\zG\mbox{-}\mathtt{Mod}_{\und{\cA}}}
\newcommand{\glLie}{(\zG, \zl)\mbox{-}\mathtt{LieAlg}}
\newcommand{\opp}{^{\op{op}}}
\newcommand{\Ivs}{I_{\zvs}}

\begin{document}

\title{Determinants over graded-commutative algebras, \\ a categorical viewpoint}
\date{}

\author{Tiffany Covolo}
\address[Tiffany Covolo]{ Universit\'e du Luxembourg, UR en math\'ematiques,
6, rue Richard Coudenhove-Kalergi, L-1359 Luxembourg,
Grand-Duch\'e de Luxembourg}
\email{tiffany.covolo@uni.lu, covolotiffany@gmail.com}

\author{Jean-Philippe Michel}
\address[Jean-Philippe Michel]{ Universit\'e de Li\`ege, Grande Traverse, 12, Sart-Tilman, 4000 Li\`ege, Belgique}
\email{jean-philippe.michel@ulg.ac.be}

\thanks{ TC thanks the
Luxembourgian NRF for support via AFR grant 2010-1 786207.
JPM thanks the Interuniversity Attraction
Poles Program initiated by the Belgian Science Policy Office for funding.
}

\vspace{2mm}
\noindent
\medskip
\noindent

	\begin{abstract}
We generalize linear superalgebra to higher gradings and commutation factors, given by arbitrary abelian groups and bicharacters.
Our central tool is an extension, to monoidal categories of modules, of the Nekludova-Scheunert faithful functor
between the categories of graded-commutative and supercommutative algebras.
As a result we generalize (super-)trace, determinant and Berezinian to graded matrices over graded-commutative algebras.
For instance, on homogeneous quaternionic matrices, we obtain a lift of the Dieudonn\'e determinant to the skew-field of quaternions.
    \end{abstract}

\maketitle \thispagestyle{empty}
%\tableofcontents
%\newpage
%%%%%%%%%%%%%%%%%%%%%%%%%%%%%%%%%%%%%%%%%%%%%%%%%%%%
%%%%%%%%%%%%%%%%%%%%%%%%%%%%%%%%%%%%%%%%%%%%%%%%%%%%
\section*{Introduction} \label{Sec0}
%%%%%%%%%%%%%%%%%%%%%%%%%%%%%%%%%%%%%%%%%%%%%%%%%%%%
%%%%%%%%%%%%%%%%%%%%%%%%%%%%%%%%%%%%%%%%%%%%%%%%%%%%

Superalgebra theory, which relies on $\mathbb{Z}_2$-grading and Koszul sign rule, admits many applications and turns out to be a non-trivial generalization of the non-graded case, regarding e.g. linear superalgebra, with the notion of Berezinian, or the classification of simple Lie superalgebras. Its success, both in mathematics and physics, prompted from the outset mathematicians to look for generalizations.
Mirroring the superalgebras, were thence introduced first $(\Z_2)^n$-graded analogues \cite{RW1,RW2}, originally called  \emph{color algebras}, and then more general $\zG$-graded versions, for an arbitrary abelian group $\zG$.
These latter were introduced independently by Scheunert \cite{Sch} in the Lie algebra case, and by Nekludova in the commutative algebra case (see \cite{SoS}). The ``color'' character of the notions considered, is encoded in a pair $(\zG,\zl)$ consisting of a \emph{grading group} $\zG$ and a \emph{commutation factor} (or \emph{bicharacter}), i.e., a biadditive skew-symmetric map $\zl: \zG\ts\zG \to \K^{\ts}$, valued in the multiplicative group of the base field $\K$.
Then, a $(\zG,\zl)$-commutative algebra over $\K$ is an associative unital $\K$-algebra, which is  $\zG$-graded, i.e. $\cA= \oplus_{\zg\in\zG}\cA^{\zg}$ and $\cA^\za\cdot\cA^\zb\subset\cA^{\za+\zb}$ for all $\za,\zb\in\zG$, and which satisfies the commutation rule
$$
ab=\zl(\za,\zb)\, ba \;,
$$
for any homogeneous elements $a\in \cA^{\za}$, $b\in \cA^{\zb}$.

The attention to higher gradings is not just a mere question of generalization. Besides the original motives of a possible application to particle physics along the lines of the coupling of superalgebra and SUSY (see \cite{RW1,Sch}),
%idea that reappeared recently  \cite{YJ},\cite{KDH},
higher gradings appear naturally in different branches of mathematics. %which makes clear that ``super is not enough''.
In geometry, they play a role in the theory of higher vector bundles \cite{GRo09}.
For instance, the algebra of differential forms over a supermanifold happens to be a $((\Z_2)^2,\zl)$-commutative algebra.
%(see intro \cite{COP}?)
Moreover, many classical non-commutative algebras, such as the algebra of quaternions \cite{KNa84,Lyc, AM, AM1,MO} or the algebra of square matrices over $\C$ \cite{Bah}, can be regarded as $(\zG,\zl)$-commutative algebras for appropriate choices of grading group $\zG$ and commutation factor $\zl$. Particularly interesting examples are the Clifford algebras \cite{AM1}.  Indeed, they are the only simple $(\zG,\zl)$-commutative algebras, with $\zG$ finitely generated and $\zl: \zG \ts \zG \to \{\pm 1\}$  \cite{OM}. Moreover, the quaternion algebra is one of them, $\mathbb{H}\simeq \op{Cl}(0,2)$.
%Note that a differential calculus has been developed over graded commutative algebras, in the spirit of non-commutative geometry \cite{Lyc,GMW}.
\\

We are interested in graded linear algebra over a graded commutative algebra.
While a basic topic, this is the starting point for many developments. Thus, a well-suited notion of tensor product lies at the heart
of a putative quaternionic algebraic geometry \cite{Joy98}, while Moore determinant plays a central role in quaternionic analysis \cite{Al03}.
In the general graded-commutative setting, the color Lie algebra of traceless elements, with respect to a graded trace, is basic  in the derivation based differential calculus \cite{GMW}.
Our aim is precisely to generalize trace, determinant and Berezinian to matrices with entries in a graded commutative algebra.
For quaternionic matrices, this is an historical and tough problem, pursued by eminent mathematicians including Cayley, Moore and Dieudonn\'e. We mention the Dieudonn\'e determinant, which is the unique group morphism $\mathrm{Ddet}:\GL(n,\mathbb{H})\to\mathbb{H}^\ts/[\mathbb{H}^\ts,\mathbb{H}^\ts]\simeq\mathbb{R}^\ts_+$, satisfying some normalization condition.
%Beyond applications in "color" mathematics, this may lead to new observations concerning classical problems. Notably, the graded %point of view gives new perspectives to  the historical and though problem of defining the trace and the determinant of matrices

%first works: \cite{COP} and \cite{Cov}
%Pushed by physics motivation,
The notion of graded trace, for matrices with entries in a graded commutative algebra $\cA$, was introduced in \cite{Sch83}. The introduction of graded determinant and Berezinian  was done slightly after in \cite{KNa84}.
These objects have been rediscovered in \cite{COP}, where they received a unique characterization in terms of their properties.
In particular, the graded determinant is constructed by means of quasi-determinants and UDL decomposition of matrices. Right after,  a cohomological interpretation of the graded Berezinian has been given in \cite{Cov}, in a close spirit to the construction in \cite{KNa84}. The main motivations of the works \cite{COP,Cov} was to lay the ground for a geometry based on quaternions, or more generally on Clifford algebras, and thus restrict to $((\Z_2)^n,\zl)$-graded commutative algebras.
Note that, so far, the graded determinant is only defined for matrices of degree zero, which is a strong restriction as we will see.

%what we do here... categorical approach \`a la NS
In the present paper, we follow another approach to the problem of defining trace, determinant and Berezinian in the graded setting.
We use their formulation as natural transformations and the Nekludova-Scheunert functor between the categories of graded-commutative and supercommutative algebras.
Its generalization to monoidal categories of graded and supermodules allows us to pull-back the wanted natural transformations
to the graded setting. The morphisms properties of trace, determinant and Berezinian are then preserved if restricted to
categorical endomorphisms $f\in\mathrm{End}_\cA(M)$, of a free graded $\cA$-module $M=\oplus_{\zg\in\zG}M^\zg$. These are
 $\cA$-linear maps $f:M\rightarrow M$, which are homogeneous, $f(M^\zg)\subset M^{\zg+\tilde{f}}$ for all $\zg\in\zG$, and of
 degree $\tilde{f}=0$. If $M$ is finitely generated, the space of all $\cA$-linear maps, without assumption on degree, turns out to be equal to the space of internal morphisms $\iEnd_{\cA}(M)$. We prove that graded determinant and Berezinian admit a proper extension to all homogeneous endomorphisms and that the graded trace can be extended to any $\cA$-linear maps, while keeping their defining properties. Note that, in the case of quaternionic matrices, this yields a lifting of the Dieudonn\'e determinant from the quotient space $\mathbb{H}^\ts/[\mathbb{H}^\ts,\mathbb{H}^\ts]\simeq \R^{\ts}_{+}$ to homogeneous invertible quaternions.
Besides, contrary to the previous works \cite{KNa84,COP,Cov}, we take full advantage of the well-developed
theory of graded associative rings and their graded modules (see e.g. \cite{NOy04}). In particular, we clarify the notion of rank of a free graded $\cA$-module and establish isomorphisms between various matrix algebras. This shows that the graded determinant of degree zero matrices, over quaternion or Clifford algebras, boils down to the determinant of a real matrix, after a change of basis.

\begin{comment}
A key observation is that the commutation factor induces a splitting of the grading group in two parts $\zG=\zG\evp\cup\zG\odp$, called \emph{even} and \emph{odd}, and then provides an underlying $\Z_{2}$-grading $\zf:\,\zG\to\Z_{2}$.
This is the starting point of the Nekludova-Scheunert Theorem, which holds in the commutative \cite{SoS} and Lie setting \cite{Sch}. We need the commutative case which can be stated as follows. If the grading group is finitely generated, there is an equivalence between the categories of $(\zG,\zl)$-commutative algebras and super commutative algebras. Roughly speaking, the Nekludova-Scheunert functor twists the $\zl$-commutative product so that it becomes super-commutative with respect to the underlying $\Z_{2}$-grading.
\end{comment}

\medskip
We detail now the content of the paper and state the main theorems. The three first sections are of introductory nature, with one new result (Theorem \ref{Thm:A}). Section $4$ is dealing with the graded trace and constitutes a warm-up for the main subject of the paper, that is graded determinant and Berezinian. The grading group is always considered abelian and finitely generated.

In \textbf{Section 1}, we recall the basic notions of $(\zG,\zl)$-commutative algebra. In particular, the commutation factor induces a splitting of the grading group in two parts $\zG=\zG\evp\cup\zG\odp$, called \emph{even} and \emph{odd}, and then provides an underlying $\Z_{2}$-grading $\zf:\,\zG\to\Z_{2}$. This is the starting point of Nekludova-Scheunert equivalence \cite{Sch,SoS}. It
is given by a family of invertible functors $\Ivs:\,\glAlg\to\glsAlg$, where  the modified commutation factor $\zls$ factors through the underlying parity, i.e., such that $\zls(\za,\zb) = (-1)^{\zf(\za)\zf(\zb)}$ for all $\za,\zb\in\zG$. Hence the \emph{NS-functors} $I_\zvs$ are valued in supercommutative algebras, with extra  $\zG$-grading.
These functors are parameterized by peculiar biadditive maps $\zvs:\zG\times\zG\rightarrow \K^\ts$, called \emph{NS-multipliers}. For later reference, the set of such maps is denoted by $\fS(\zl)$.

In \textbf{Section 2}, we introduce the closed symmetric monoidal category $\Moda$ of graded modules over a $(\zG,\zl)$-commutative algebra $\cA$, following partly \cite{NOy04}.
We prove our first result: for every $(\zG,\zl)$-commutative algebra $\cA$ and every $\zvs\in\fS(\zl)$,
there exists a functor
\bemaps
 \widehat{I}_{\zvs}:& \Moda &\raa&\Modabar,
\eemaps
which can be completed into a closed monoidal functor. Here, $\cuA=I_{\zvs(\cA)}$ is a supercommutative algebra.
This yields in particular the following result, which is the starting point
of later developments.
\begin{thmintro}\label{Thm:A}%A
Let $M$ a graded module over a $(\zG,\zl)$-commutative algebra $\cA$. Then, for any map $\zvs\in\f{S}(\zl)$, the map $\zh_{\zvs}$ defined in \eqref{eta} is an $\cA^0$-module isomorphism
\bemaps
\zh_{\zvs}: &\iEnd_{\cA}(M)& \raa &\iEnd_{\cuA}\lp \widehat{I}_{\zvs}(M)\rp
\eemaps
such that
\be\label{etacomp}
 \zh_{\zvs}(f\circ g) = \zvs\lp\degr{f},\degr{g}\rp^{-1}\,  \zh_{\zvs}(f)\circ \zh_{\zvs}(g)\;,
\ee
  for any pair of homogeneous endomorphisms  $f,g\in \iEnd_{\cA}(M)$.
\end{thmintro}

In \textbf{Section 3}, we report on free graded modules and graded matrices, using \cite{NOy04} and \cite{COP} as references.
In particular, the $\cA$-module of $m\times n$ matrices %over a $(\zG,\zl)$-commutative algebra $\cA$
receives a $\zG$-grading,
induced by degrees $\boldsymbol{\zm}\in\zG^m$
and $\boldsymbol{\zn}\in\zG^n$ associated
respectively to rows and columns of the matrices.
This module is denoted by $\Mat(\boldsymbol{\zm}\times\boldsymbol{\zn};\cA)$ or simply $\Mat(\boldsymbol{\zn};\cA)$ if
$\boldsymbol{\mu}=\boldsymbol{\nu}$. If a free graded $\cA$-module  $M$ admits a basis of homogeneous elements $(e_i)$ of degrees $(\zn_{i})=\boldsymbol{\zn}$, then we have an isomorphism of $\zG$-algebras
$$ \iEnd_{\cA}(M) \simeq \Mat(\boldsymbol{\zn};\cA)\;.$$
The statement of the previous section then rewrites in matrix form, the map $\eta_\zvs$ becoming then an $\cA^0$-modules isomorphism
$$
J_{\zvs}: \,  \Mat(\boldsymbol{\zn};\cA) \raa  \Mat(\boldsymbol{\zn};\cuA).
$$
To complete our understanding of matrix algebras, we determine conditions on degrees $\boldsymbol{\zn}, \boldsymbol{\zm}\in \zG^{n}$ such that $\Mat(\boldsymbol{\zn};\cA) \simeq \Mat(\boldsymbol{\zm};\cA)$.
This is equivalent to find bases of degrees $\boldsymbol{\zn}$ and $\boldsymbol{\zm}$ in the same free graded $\cA$-module.
If $\cA^{0}$ is a local ring, we prove that two such bases exist if and only if $\boldsymbol{\zn}\in (\zG_{\cA^\ts})^n+ \boldsymbol{\zm}$ (up to permutation of components).  Here, $\zG_{\cA^\ts}$ denotes the set of degrees $\zg\in\zG$ such that $\cA^\zg$ contains at least one invertible element. This result is closely related to a classical Theorem of Dade \cite{Dad80} (cf.\ Proposition \ref{DadeProp}).

\begin{comment}
If $\cA$ is strongly graded, i.e.
$\cA^\za\cdot\cA^\zb=\cA^{\za+\zb}$ for all $\za,\zb\in\zG$, then a Theorem of Dade shows that two free graded $\cA$-modules
are isomorphic if and only they have same total rank \cite{Dad80}. This means that a free graded $\cA$-module of
total rank $n$ admits bases of arbitrary degree $\boldsymbol{\zn}\in\zG^n$.
\end{comment}

%From now on, $\cA$ denotes a $(\zG,\zl)$-commutative algebra, with $\zG$ assumed finitely generated, and $M$ denotes a free graded $\cA$-module with a homogeneous basis of degree $\boldsymbol{\zn}\in\zG^n$.

In \textbf{Section 4}, we study the \emph{graded trace}. By definition, this is a degree-preserving $\cA$-linear map  $\Tr: \iEnd_{\cA}(M) \to \cA$, which is also a $(\zG,\zl)$-Lie algebra morphism. The latter property means that
$$
\Tr(f\circ g)-\zl(\tilde{f},\tilde{g})\Tr(g\circ f)=0
$$
for any pair of homogeneous endomorphisms $f,g$ of respective degrees $\tilde{f},\tilde{g}\in \zG$. Using the map $\eta_\zvs$ of Theorem \ref{Thm:A} and the supertrace $\str:\iEnd_{\cA}(M)\rightarrow \cuA$, we construct a graded trace and show it is
essentially unique.
\begin{thmintro}\label{Thm:C}%B
Let $\cA$ be a $(\zG,\zl)$-commutative algebra, $M$ be a free graded $\cA$-module and $\zvs\in\fS(\zl)$. Up to multiplication by a
 scalar in $\cA^0$, there exists a unique graded trace $\Tr: \iEnd_{\cA}(M) \to \cA$. One is given by the  map $\str \circ \zh_{\zvs}$,
 which does not depend on $\zvs\in\fS(\zl)$.
%We call $\Tr$ the \emph{graded trace} on endomorphisms of $M$.
\end{thmintro}
In matrix form, the graded trace reads as $\Tr=\str\circ J_\zvs$, and its evaluation on a homogeneous matrix $X=(X^{i}_{\,\;j})_{i,j}\in \Mat^x(\boldsymbol{\zn};\cA)$ gives
$$  \Tr(X) = \sum_i  \zl(\zn_i, x+\zn_i) X^i_{\;\,i} \;. $$
As a result our graded trace on matrices coincides with the one introduced in \cite{Sch83} and is a generalization of those studied in \cite{KNa84,COP,GMW}.

As for the graded determinant and Berezinian, the situation is more involved. For the sake of clarity, we first treat the case of purely-even algebras. These are $(\zG,\zl)$-commutative algebras $\cA$ with no odd elements, i.e., for which $\zG=\zG\evp$. In this case, each NS-functor $\Ivs$ sends $\cA$ to a classical commutative algebra $\cuA:=\Ivs(\cA)$, and we make use of the classical determinant $\det_{\cuA}$ over $\cuA$.

In \textbf{Section 5}, we restrict to the subalgebra of $0$-degree matrices, denoted by $\Mat^0(\boldsymbol{\nu};\cA)$, and to the subgroup of invertible matrices, denoted by $\GL^0(\boldsymbol{\nu};\cA)$.
As already mentioned, the arrow function of the NS-functor $\widehat{I}_{\zvs}$ allows then to pull-back the determinant to the graded side, while keeping its multiplicativity property. Following \cite{McD}, we generalize the multiplicative characterization of the classical determinant to the graded setting.% determinant.
            \begin{thmintro}\label{thm:Gdet0}
Let $\boldsymbol{\zn}\in\zG^n$, $n\in\N$. There exists a unique family of maps
$$  \Gdet^0_{\cA}: \, \GL^0(\boldsymbol{\zn}; \cA) \raa (\cA^0)^{\ts}\,, $$
parameterized by objects $\cA$ in $\glAlg$,
such that
\begin{enumerate}
    \item[\textsc{ai}.] it defines a natural transformation
    \[\begin{tikzpicture}
    \matrix(m)[matrix of math nodes, column sep=8em, row sep=4em]
{
\glAlg & \mathtt{Grp} \\
};
\draw[->, shorten >=5pt,shorten <=5pt](m-1-1) to[bend left=40] node[label=above:$\scriptstyle\GL^0(\boldsymbol{\zn};   -  ) $] (U) {} (m-1-2);
\draw[->,shorten >=5pt,shorten <=5pt] (m-1-1) to[bend right=40,name=D] node[label=below:$\scriptstyle \lp( - )^0\rp^{\ts}$] (V) {} (m-1-2);
\draw[double,double equal sign distance,-implies,shorten >=5pt,shorten <=5pt]
  (U) -- node[label=right:$\Gdet^0$] {} (V);
    \end{tikzpicture}\]
    \item[\textsc{aii}.]  for any invertible $a\in\cA^0$,
    %and any index $1\leq k\leq r$, denoting $\I_{k,a}$ the diagonal matrix with $a$ as $k$-th diagonal element and $1$ at all the other diagonal positions, we have for any matrix $X\in \GL^0(\boldsymbol{\zn};\cA)$, $$ \Gdet^0(\I_{k,a})= a \;.$$
    $$ \Gdet^0_{\cA}\lp \begin{array}{cccc} 1&&&\\&\ddots&&\\&&1&\\&&&a \end{array}\rp = a \;.$$
\end{enumerate}
% Moreover, $\Gdet^0_{\cA}(X)$ is  a polynomial linear in the rows and the columns of $X$// or // $\Gdet^0$ admits a polynomial expression analogue to the on of the classical determinant.
The natural transformation $\Gdet^0$ is called the \emph{graded determinant} and satisfies
\be\label{def:gdet0}
 \Gdet^{0}_{\cA}(X)=\det\nolimits_{\cuA}\lp J_{\zvs}(X)\rp \,,
\ee
for all $X\in \GL^0(\boldsymbol{\zn};\cA)$ and all $\zvs\in\f{S}(\zl)$.
%Notably, this attest that $\Gdet^0$ does not depend on the map $\zvs\in \f{S}(\zl)$ chosen for its construction.
            \end{thmintro}
The axiom \textsc{ai}. implies that $\Gdet^0_{\cA}$ is a group morphism. Equation \eqref{def:gdet0} allows to extend $\Gdet^0$ to all matrices in $\Mat^0(\boldsymbol{\nu};\cA)$, and then $\Gdet^0(X)$ is invertible if and only if $X$ is invertible. Our graded determinant coincides with the ones defined in \cite{KNa84,COP}, and, as an advantage of our construction, we find an explicit formula for $\Gdet^0(X)$. This is the same polynomial expression in the entries of the matrix $X$ that the classical determinant would be, but here the order of the terms in each monomial  is very important: only with respect to a specific order
%(related to the cycle decomposition of the permutation)
does the formula retain such a nice form. This situation is analogous to that of Moore's determinant, defined for Hermitian quaternionic matrices  (see e.g.\ \cite{Asl}).
In addition, if $\cA$ admits invertible homogeneous elements of any degree, we provide a transition matrix $P$, such that
$PXP^{-1}$ is a matrix with entries in the commutative algebra $\cA^0$ and $\Gdet^0_\cA(X)=\det_{\cA^0}(PXP^{-1})$, for all $X\in\Mat^0(\boldsymbol{\nu};\cA)$. This recovers a result in \cite{KNa84}.

In \textbf{Section 6}, we investigate the natural extension of the graded determinant $\Gdet^0$ to arbitrary graded matrices, by defining maps in the same way
\bemaps
\Gdet_{\zvs}= \det \circ J_{\zvs}:& \Mat(\boldsymbol{\zn};\cA) &\raa& \cA\;,
\eemaps
where $\zvs\in\fS(\zl)$. These functions share determinant-like properties and can be characterized as follows.

\begin{thmintro}\label{Thm:E}
Let $\f{s}:\,\zG\ts\zG\to \K^{\ts}$. If $\f{s}\in\fS(\zl)$, then $\Gdet_{\f{s}}$ satisfies properties \ref{GdetGdet0}\!--\ref{Axiom3} below.
Conversely, if there exists, for all $\boldsymbol{\zn}\in \bigcup_{n\in\N^*}\zG^{n} $,  a natural transformation
\be \label{Nat-Det}
\begin{tikzpicture}
    \matrix(m)[matrix of math nodes, column sep=8em, row sep=4em]
{
\glAlg & \cat{Set}\\
};
\draw[->, shorten >=5pt,shorten <=5pt](m-1-1) to[bend left=40] node[label=above:$\scriptstyle\Mat(\boldsymbol{\zn};   -  ) $] (U) {} (m-1-2);
\draw[->,shorten >=5pt,shorten <=5pt] (m-1-1) to[bend right=40,name=D] node[label=below:$\scriptstyle {\rm forget} $] (V) {} (m-1-2);
\draw[double,double equal sign distance,-implies,shorten >=5pt,shorten <=5pt]
  (U) -- node[label=right:$\zD_{\f{s}}$] {} (V);
    \end{tikzpicture}
\ee
satisfying properties \ref{GdetGdet0}\!--\ref{Axiom3}, then $\f{s}\in\fS(\zl)$ and $\zD_{\f{s}}=\Gdet_{\f{s}}$. The properties are
\begin{enumerate}[label=\textsc{\roman *.}]
    \item \label{GdetGdet0} extension of $\Gdet^0$, i.e.,  $\zD_{\f{s}}(X)=\Gdet^{0}(X)$  for all $X\in\GL^{0}(\boldsymbol{\zn};\cA)$;
    \item \label{Axiom1b}
weak multiplicativity, i.e.,
$
\zD_{\f{s}}(X Y)=\zD_{\f{s}}(X) \cdot \zD_{\f{s}}(Y)
$
for all $X,Y\in\Mat(\boldsymbol{\zn};\cA)$ such that either $X$ or $Y$ is homogeneous of degree $0$;
    \item\label{Axiom2}
 additivity in the rows, i.e.,
 $\zD_{\f{s}}(Z)= \zD_{\f{s}}(X)+\zD_{\f{s}}(Y)$ for all graded matrices
 $$
X=  \left[ \begin{array}{c} \mathtt{x}^1 \\ \vdots \\  \mathtt{x}^n \end{array}\right]\;,\;
Y=\left[ \begin{array}{c} \mathtt{y}^1 \\ \vdots \\  \mathtt{y}^n \end{array}\right]
\;
,
\;
Z=         \left[ \begin{array}{c}
         \mathtt{z}^1 \\  \vdots \\ \mathtt{z}^n \end{array}\right]\;\;\in\Mat(\boldsymbol{\zn};\cA)
$$
whose rows satisfy, for a fixed index  $1\leq k\leq n$,
 \bea \label{XYZ}
  \mathtt{z}^i=  \mathtt{x}^i =  \mathtt{y}^i\,,\mbox{ if } i\neq k
 & \mbox{and} &
   \mathtt{z}^k=  \mathtt{x}^k+ \mathtt{y}^k
\;; \eea
	\item \label{Axiom3} heredity for diagonal matrices, i.e.,
for all $c\in\cA^{\degr{c}}$, %we have that
\be\label{GdetDiag}
\zD_{\f{s}}
\lp \begin{array}{c|c}
D & \\
\hline
&\zl(\degr{c},\zn_{n})\, c
\end{array}\rp
=
\f{s}\Big(\degr{c}\,,\, \op{deg}\lp\zD_{\f{s}} \lp D \rp\rp \Big)\, c\cdot \zD_{\f{s}} \lp D \rp\,,
\ee
\begin{comment}
\be\label{GdetDiag}
\Gdet_{\zvs}
\lp \begin{array}{ccccc|c}
&&&&&\\
&&D^{\bza}& & & \\
&&&&&\\
\hline
&&&&&\zl(\degr{c},\zn_{n})\, c
\end{array}\rp
=
\zvs\Big(\degr{c}\,,\, \sum_{i=1}^{n-1} \za_i \Big)\, c\cdot \Gdet_{\zvs} \lp D^{\bza}\rp\,,
\ee
\end{comment}
where $D\in\Mat(\boldsymbol{\zn'};\cA)$, with $\boldsymbol{\zn'}=(\zn_{1},\ldots , \zn_{n-1})$, is a diagonal matrix with homogeneous entries and $\op{deg}\lp\zD_{\f{s}} \lp D \rp\rp$ is the degree\footnote{By induction on the rank, $\zD_{\f{s}} \lp D \rp$ is homogeneous.}
of $\zD_{\f{s}} \lp D \rp$.
 \end{enumerate}
\end{thmintro}
Despite the above characterization, the determinant-like functions $\Gdet_\zvs$ are not proper determinants: they are not multiplicative and do not characterize invertibility of matrices in general. They depend both on the choices of NS-multiplier $\zvs\in\fS(\zl)$
and of grading $\boldsymbol{\nu}$ on matrices. We illustrate these drawbacks on $2\times 2$ quaternionic matrices.
Note that any multiplicative determinant satisfies properties \ref{GdetGdet0}-\ref{Axiom1b},
and property \ref{Axiom3} is rather natural and satisfied, e.g., by Dieudonn\'e determinant.
Hence, Theorem \ref{Thm:E} puts severe restriction on the existence of multiplicative and multiadditive determinants.
On quaternionic matrices, as is well-known, no such determinant exists \cite{Dys,Asl}.
Nevertheless, $\Gdet_\zvs$ define proper determinants  once restricted to homogeneous  matrices.
In particular,  on homogeneous quaternionic matrices, $\Gdet_\zvs$ provide lifts of the Dieudonn\'e determinant to $\mathbb{H}$.
Note that restriction to homogeneous elements is also needed for the recent extension of Dieudonn\'e determinant to graded division algebras \cite{HWa12}.

Finally, in \textbf{Section 7}, we deal with the general case of a $(\zG,\zl)$-commutative algebra $\cA$ with odd elements. In this case, each NS-functor $\Ivs$ sends $\cA$ to a supercommutative algebra $\cuA:=\Ivs(\cA)$, and we use the classical Berezinian $\ber$ over $\cuA$ to define the \emph{graded Berezinian} as $\Gber_{\zvs}:=\ber\circ J_{\zvs}$, with $\zvs\in\fS(\zl)$. We get the following results.
\begin{thmintro}\label{Thm:F}%E
Let $r\evp\,,r\odp\in\N$ and $\boldsymbol{\zn}=(\boldsymbol{\zn}\evp\,, \boldsymbol{\zn}\odp)\in \zG\evp^{r\evp} \ts \zG^{r\odp}\odp$.
\begin{enumerate}
    \item The maps
$\Gber_{\zvs}:\, \GL^0(\boldsymbol{\zn}; \cA) \to  (\cA)^{\ts}$
do not depend on the choice of $\zvs\in\f{S}(\zl)$ and define a group morphism, denoted by $\Gber^0$.
\item The maps $\Gber_\zvs$ are well-defined on matrices $X\in \GL^x(\boldsymbol{\zn};\cA)$, $x\in\zG\evp$, and given by
    \begin{equation}\label{Gbers}
\Gber_{\zvs}(X)=\zvs(x,x)^{-r\odp\,(r\evp\,-r\odp\,)}
\, \Gdet_{\zvs}(\cX_{00}- \cX_{01}\cX_{11}^{-1} \cX_{10} )\cdot \Gdet_{\zvs}(\cX_{11})^{-1}\;,
\end{equation}
where $\lp \begin{array}{cc} \cX_{00} & \cX_{01}\\ \cX_{10}& \cX_{11} \end{array} \rp =X$ is the block-decomposition of $X$ according to parity.
\end{enumerate}
\end{thmintro}
Note that $\zvs(0,0)=1$. The graded Berezinian $\Gber^0$ coincides with the ones introduced formerly, and independently, in \cite{KNa84} and \cite{COP,Cov}. For a characterization of $\Gber^0$, we refer to \cite{COP}.
The formula $\Gber_\zvs=\ber\circ J_\zvs$, defining the graded Berezinian, cannot be extended further to inhomogeneous matrices.
Indeed, the map $J_{\zvs}$  does not preserves invertibility of inhomogeneous matrices in general and the Berezinian is only defined over even invertible matrices.

For the sake of self-consistency of the paper, we present some basic notions of category theory
in an appendix.

\medskip
\textbf{Acknowledgements}.
We are grateful to Dimitry Leites who made us aware of Nekludova's work, which initiated this work, and for his comments.
We are pleased to thank Valentin Ovsienko  for enlightening discussions
and valuable comments.
We also thank Norbert Poncin for his support.

\medskip

\medskip
\textbf{Notation}.
In the whole paper, $\K$ is a field of characteristic zero (e.g. $\R$ or $\C$) and $\K^{\ts}$ denotes its group of invertible elements.
The grading group is a finitely generated abelian group, denoted by $(\zG$,+), with neutral element $0$.

\medskip

%--------------------------------------------------------------------------------------------
%--------------------------------------------------------------------------------------------
%--------------------------------------------------------------------------------------------
%--------------------------------------------------------------------------------------------

%%%%%%%%%%%%%%%%%%%%%%%%%%%%%%%%%%%%%%%%%%%%%%%%%%%%
%%%%%%%%%%%%%%%%%%%%%%%%%%%%%%%%%%%%%%%%%%%%%%%%%%%%
\section{Categories of Graded Commutative Algebras} \label{Sec1}
%%%%%%%%%%%%%%%%%%%%%%%%%%%%%%%%%%%%%%%%%%%%%%%%%%%%
%%%%%%%%%%%%%%%%%%%%%%%%%%%%%%%%%%%%%%%%%%%%%%%%%%%%

In this section, we review the basic notions of graded-commutative algebra and present the Nekludova-Scheunert Theorem, following \cite{Sch,SoS}. For general notions on graded rings and algebras, we refer to \cite{Bourbaki, Dad80, NOy04}.

%%%%%%%%%%%%%%%
\subsection{Definitions and examples}\label{DefAlg}
%%%%%%%%%%%%%%%
A \emph{$\zG$-algebra} is an  associative unital algebra $\cA$ over $\K$, with the structure of a  $\zG$-graded vector space $\cA=\opl_{\za \in \zG} \cA^{\za}$, in which the multiplication respects the grading,
\be\label{zG-algebra}
\cA^{\za}\cdot \cA^{\zb} \subset \cA^{\za+\zb}\;, \quad \forall \za,\zb\in\zG\, .
\ee
If $\cA^{\za}\cdot \cA^{\zb} = \cA^{\za+\zb}$ for all $\za,\zb\in\zG$, we say that $\cA$ is \emph{strongly graded}.
The elements in $\cA^\zg$ are called \emph{homogeneous elements} of degree $\zg$.
The $\zG$-algebra $\cA$ is called a \emph{crossed product} if there exists invertible elements of each degree $\zg\in\zG$ (see Remark \ref{crossproduct} for an alternative, more usual, definition). It is a \emph{graded division algebra} if all non-zero homogeneous elements are invertible.
Clearly, if a $\zG$-algebra is a crossed product, then it is strongly graded.
The converse holds if $\cA^0$ is a  local ring \cite{Dad80}.
\emph{Morphisms of $\zG$-algebras} are morphisms of algebras $f:\, \cA \to \cB$ that preserve the degree, $f(\cA^{\zg}) \subset \cB^{\zg}$, for all $\zg\in \zG$.
The kernel of such morphism is an \emph{homogeneous ideal} of $\cA$, i.e., an ideal $I$ such that
$$ I=\bigoplus_{\zg\in\zG} \lp I\cap\cA ^{\zg} \rp \;.$$
The quotient of a $\zG$-algebra by a homogeneous ideal is again a $\zG$-algebra.
The class of $\zG$-algebras (over a fixed field $\K$) and corresponding morphisms form the category $\gAlg$.

Let $\zl$ be a map $\zl: \zG \ts \zG \to \K^{\ts}$. A \emph{$(\zG,\zl)$-commutative algebra} is a $\zG$-algebra in which the multiplication is $\zl$-commutative, namely
\be \label{grcomm} a\cdot b =\zld{a}{b} \, b\cdot a\;, \ee
for all homogeneous elements $a,b\in \cA$ of degrees $\degr{a},\degr{b}\in \zG$.
\emph{Morphisms of $(\zG,\zl)$-commutative algebras} are morphisms of $\zG$-algebras.
The class of $(\zG,\zl)$-commutative algebras (over a fixed field $\K$) and corresponding morphisms form a full subcategory $\glAlg$ of $\gAlg$.
\begin{exa}\label{ex:crossprod}{\rm
If $\cA^0$ is a unital associative and commutative algebra, then the group algebra $\cA^0[\zG]$ is  a $(\zG,\mathbf{1})$-commutative algebra, with $\mathbf{1}:\zG\times\zG\rightarrow\mathbb{K}$ the constant map equal to $1\in\mathbb{K}$. This is a crossed product in general and a graded division algebra if and only if $\cA^0$ is a field.}
\end{exa}

\begin{exa}{\rm \label{exasuper}
\emph{Supercommutative algebras} are $(\Z_{2}, \zlsup)$-commutative algebras, with
$$
\zlsup(x,y)=(-1)^{xy}\;,
$$
for all $ x,y\in \Z_{2}$. The category $(\Z_{2}, \zlsup)$-$\cat{Alg}$ is denoted for short by $\cat{SAlg}$.
}\end{exa}

\begin{exa}{\rm \label{exaWeil}
Let $n\in\N$. The local commutative algebra
$$\EnsQuot{\R[\ze_1,\ldots,\ze_n]}{(\ze_1^2,\ldots,\ze_n^2)}\;,$$
which generalizes dual numbers, receives a $(\Z_2)^n$-grading by setting $\tilde{\ze}_1=(1,0,\ldots,0),\; \ldots\;,\tilde{\ze}_n=(0,\ldots,0,1)$. This is a $((\Z_2)^n,\zl)$-commutative algebra with commutation factor
\be\label{lambda} \zl(x,y):=\ks{x}{y}\; ,\ee
where $\la-,-\ra$ denotes the standard scalar product of binary $n$-vectors. This algebra is  not strongly graded if $n\geq 1$.
 }\end{exa}

\begin{exa}{\rm \label{Clifford1}
Let $p,q\in\N$ and $n=p+q$. The real {\it Clifford algebra} $\op{Cl}(p,q)$ is the real algebra with $n$ generators $(e_i)_{i=1,\ldots,n}$, satisfying the relations $e_ie_j+e_je_i=\pm 2\delta_{ij}$, with plus for $i\leq p$ and minus otherwise. Setting $\tilde{e}_1=(1,0,\ldots,0,1),\; \ldots\;,\tilde{e}_n=(0,\ldots,0,1,1)$, it turns into a $\left((\Z_2)^{n+1},\zl\right)$-commutative algebra with commutation factor of the form \eqref{lambda}.
 For such a grading, the real Clifford algebras are  graded division algebras.
 }\end{exa}

\begin{exa}{\rm %{\cvd NEW}
The algebra of $n\ts n$ matrices $\Mat(n;\C)$ % over an algebraically closed field $\F$ (of characteristic zero)
 turns into a graded division algebra with respect to the grading group $\zG=\Z_n\ts\Z_n$ and is graded commutative for a commutation factor $\zl: \zG\ts\zG \to \mathbb{U}_n$, with $\mathbb{U}_n$ the group of $n$-th roots of unity (see \cite{Bah}).
%{\cmg more details?}
}\end{exa}

%%%%%%%%%%%%%%%
\subsection{Basic properties}\label{BasicProp}
%%%%%%%%%%%%%%%

Let $\cA$ be a $(\zG,\zl)$-commutative algebra. Clearly, the unit of $\cA$, denoted by $1_{\cA}$ or simply $1$, is necessarily of degree $0$. The commutation relation \eqref{grcomm}, together with the associativity of $\cA$, implies that the \emph{commutation factor}  $\zl$ on $\zG$ satisfies the three following conditions  (see \cite{Sch}):
\begin{enumerate}[label=(C\arabic *), ref=C\arabic *]
    \item \label{cf1} $\zl(x,y)\zl(y,x)=1\;,$
    \item \label{cf2} $\zl(x+y,z)=\zl(x,z)\zl(y,z)\;,$
    \item \label{cf3} $\zl(z,x+y)=\zl(z,x)\zl(z,y)\;,$
\end{enumerate}
for all $x,y,z\in \zG$. Conversely, for any commutation factor $\zl$, there exists a $(\zG,\zl)$-commutative algebra.
%; for a construction of free $(\zG,\zl)$-commutative algebras, see Section \ref{freeAlg}.
An easy consequence of the conditions \eqref{cf2} and \eqref{cf3} is that
  $$ \zl(0,x)=\zl(x,0)=1\;,$$
hence $\cA^0 \subseteq Z(\cA)$, the center of $\cA\;$.
As for the condition \eqref{cf1}, it implies $\lp\zl(x,x)\rp^2=1$, for all $x\in \zG$.  Therefore the commutation factor induces a splitting of the grading group into an ``even'' and an ``odd'' part,
\bea \label{G0G1}
 &\zG=\zG\evp \cup \zG\odp & \\
&\mbox{ with } \;\;\;\zG\evp:=\{x\in \zG\,:\, \zl(x,x)=1 \} \;\; \mbox{ and } \;\; \zG\odp:=\{x\in \zG\,:\, \zl(x,x)=-1 \}\;.& \nonumber
\eea
This is equivalently encoded into a morphism of additive groups
\bemaps
\zvf_{\zl}:& \zG & \to & \Z_{2} \;,
\eemaps
with kernel $\zG\evp\;$. The map $\zvf_{\zl}$ provides a $\Z_2$-grading on every $(\zG,\zl)$-commutative algebra $\cA$, also called \emph{parity} and denoted by $\pari{\zg}:=\zvf_{\zl}(\zg) \in \Z_2\;$. Homogeneous elements of $\cA$ are named \emph{even} or \emph{odd} depending on the parity of their degree. Notice that odd elements of $\cA$ are nilpotent.
Hence, if $\cA$ is strongly graded, we have $\zG\odp=\emptyset$.
The group morphism $\zvf_{\zl}$ induces a \emph{regrading functor}
\bemap \label{regradFun}
\zF_{\zl}:& \glAlg & \to & \Z_{2}\mbox{-}\cat{Alg}\;,
\eemap
equal to the identity on arrows and such that $\zF_{\zl}(\cA)$ is the algebra $\cA$ with $\Z_2$-grading
$$
\begin{array}{ccc}
 \lp \zF_{\zl}(\cA)\rp\evp=\bigoplus_{\zg\in \zG\evp}\cA^{\zg} & \mbox{ and } & \lp \zF_{\zl}(\cA)\rp\odp=\bigoplus_{\zg\in \zG\odp}\cA^{\zg} \;.
\end{array}
$$
The $\Z_2$-graded algebra $\zF_{\zl}(\cA)$ is supercommutative if and only if $\zl$ can be factorized through the parity, as stated below.
\begin{prop}\cite{SoS}\label{regrading}
Let $\zl$ a commutation factor over $\zG$. Then, the two following statements are equivalent:
\begin{enumerate}[label=(\roman *)]
	\item
for all $x,y\in \zG$,  $\zl(x,y)=\zlsup(\zvf_{\zl}(x), \zvf_{\zl}(y))\;;$
	\item
the functor $\zF_{\zl}$ takes values in %a subcategory of
$\cat{SAlg}$.
\end{enumerate}
\end{prop}

%------------------------
\subsection{Change of commutation factor}
%------------------------

Let $\zvs: \zG\ts\zG \to \K^{\ts}$ be a map. One can twist the multiplication of a $(\zG,\zl)$-commutative algebra  $\cA=(A,\cdot)$ by the map $\zvs$ in the following way:
\be \label{star-cdot} a\star b:=\zvs(\degr{a},\degr{b})\, a\cdot b \;.\ee
The resulting graded algebra $\cuA:=(A, \star)$ is an associative deformation of $\cA$ if and only if $\zvs$ satisfies the cocycle condition
$$ \zvs(x,y+z)\zvs(y,z)=\zvs(x,y)\zvs(x+y,z)\;,$$
for all $x,y,z\in\zG$. If in addition $\zvs(0,0)=1$, then $\cuA^{0}=\cA^{0}$ as algebras and
$\zvs$ is called \emph{multiplier} (see \cite{Sch}).
Note that biadditive maps are multipliers.

If $\zvs$ is a multiplier, the deformed algebra $\cuA$ is a $(\zG, \zls)$-commutative algebra,
 the commutation factor $\zls:\zG\ts\zG \to \K^{\ts}$ being given by
$$ \zls(a,b):=\zl(a,b)\zvs(a,b)(\zvs(b,a))^{-1}\;.$$
Several multipliers can lead to the same change of commutation factor. More precisely,
\begin{lem}\cite{Sch}\label{lem:biadditif}
Let $\zvs$, $\zvs'$ be two multipliers on $\zG$.  Then, $\zl^{\zvs} = \zl^{\zvs'}$ if and only if there exists a biadditive symmetric map  $\mathfrak{b}:\zG\ts\zG\to \K^{\ts}$  such that $\zvs'(x,y)= \zvs(x,y) \, \mathfrak{b}(x,y)$ for all $x,y \in \zG$.
\end{lem}

For a classification of multipliers and commutation factors, we refer to \cite{GMW}.

%%%%%%%%%%%%%%%
\subsection{The Nekludova-Scheunert Theorem}\label{SNthmAlg}
%%%%%%%%%%%%%%%

Let $\zvs$ be a multiplier on $\zG$. In general, the $(\zG,\zl)$-commutative algebra $\cA=(A,\cdot)$ is not isomorphic to  the $(\zG,\zl^{\zvs})$-commutative algebra $\cuA:=(A, \star)$  obtained by change of commutation factor.
The link between these two types of algebra is encompassed by a functor, defined below.
\begin{prop}\cite{SoS}\label{Izvs}
Let $\zvs: \zG \ts \zG \to \K^{\ts}$ be a multiplier.
There exists an isomorphism of categories
\begin{eqnarray}\nonumber
\Ivs: \glAlg& \to & \glsAlg\\  \label{Eq:Izvs}
\cA=(A,\cdot) & \mapsto &\cuA=(A, \star)
\end{eqnarray}
%where the functor $\Ivs$ is defined on objects by
%$\Ivs(\cA):=\cuA$ and
equal to the identity on morphisms.
\end{prop}
The proof is obvious, the inverse functor is given by $I_{\zd}$ with $\zd(x,y):=(\zvs(x,y))^{-1}\,$.

\medskip
According to Scheunert \cite{Sch}, for any commutation factor $\zl$, there exists a multiplier $\zvs$ such that $\zls$ factorizes through the parity. This uses the fact that $\zG$ is a finitely generated abelian group. By Lemma \ref{lem:biadditif}, the set of such \emph{NS-multipliers},
\[\label{fS(zl)}
 \fS(\zl):=\{\zvs: \zG \ts \zG \to \K^{\ts}\,  \vert \, \zls = \zlsup \circ (\zvf_{\zl}\ts \zvf_{\zl}) \} \;,
 \]
is parameterized by symmetric biadditive maps from $\zG \ts \zG$ to $\K^{\ts}$. It turns out that all $\zvs\in\f{S}(\zl)$ are biadditive \cite{Sch}. This result, combined with Propositions \ref{regrading} and \ref{Izvs}, yields the following Theorem, due to Nekludova  \cite[p. 280]{SoS}. Scheunert has proved  a similar theorem, in the Lie algebra setting \cite{Sch}.

\begin{thm}[The Nekludova-Scheunert Theorem]\label{thm:NSalg}
Let $\zl$ be a commutation factor on $\zG$.
There exists a biadditive map $\zvs$ such that the functor $\zF_{\zls}$ (defined in \eqref{regradFun}) takes values in $\cat{SAlg}$, and the composite
$$
\begin{array}{rcccl}
 \glAlg &\xrightarrow[\Ivs]{\;\sim\;} & \glsAlg & \xrightarrow[\zF_{\zls}]{} &\cat{SAlg}
\end{array}
$$
is  a faithful functor.
\end{thm}

If $\zvs$ is a NS-multiplier, the functors $\Ivs$ satisfy automatically Theorem \ref{thm:NSalg} and we call them \emph{NS-functors}.
%If  $\zvs\in\f{S}(\zl)$, the functors $\Ivs$ satisfy Theorem \ref{thm:NSalg}, by definition of $\f{S}(\zl)$. We call them the \emph{Nekludova-Scheunert functors}, or \emph{NS-functors} for short.

The $\zG$-algebras which are crossed product are characterized in \cite{NOy04}. This specifies as follows for $(\zG,\zl)$-commutative algebras.
\begin{cor}
An algebra $\cA$ is $(\zG,\zl)$-commutative and  a crossed product if and only if $\cA\simeq I_{\zvs}^{-1}(\cA^0[\zG])$ for some $\zvs\in\fS(\zl)$.
\end{cor}
\begin{proof}
Let $\cA$ be a $(\zG,\zl)$-commutative algebra. Since odd elements are nilpotent,  $\cA$ is a crossed product implies $\zG=\zG\evp$.
For such $\zG$, the NS-functor $I_{\zvs}$, with  $\zvs\in\f{S}(\zl)$, takes values in the category $(\zG, \mathbf{1})$-\texttt{Alg} of commutative $\zG$-graded algebras.
Moreover, it clearly sends crossed products to crossed products.
Thence, from the crossed product $\cA^{0}[\zG]$, we get a $(\zG,\zl)$-commutative algebra $I_{\zvs}(\cA^{0}[\zG])$ which is also a crossed product.
Conversely, if $\cA$ is $(\zG,\zl)$-commutative and  a crossed product, it is easy to prove that  $I_{\zvs}^{-1}(\cA)$ is isomorphic to  $\cA^{0}[\zG]$, and the result follows.
\end{proof}
\begin{rema}\label{crossproduct}
By definition, the algebra $I_{\zvs}^{-1}(\cA^0[\zG])$ is the $\cA^{0}$-module generated by the group elements $(e_{\za})_{\za\in\zG}\,$, with product law
$e_{\za}e_{\zb}= \zvs(\za,\zb)e_{\za+\zb}$ for all $\za,\zb\in \zG$.
This algebra
 is usually denoted by $\cA^{0} \rtimes_{\zvs} \zG$ and called \emph{crossed product of $\cA$ by $\zG$ relatively to $\zvs$}. %\cite{GMW}.
\end{rema}

We provide now some examples of NS-functors.
\begin{exa} \label{ex:BL}{\rm
The algebra of differential forms $\zW$ over a smooth supermanifold is a  $\Z_2\times\Z_2$-algebra, where the $\Z_2\times\Z_2$-degree of a homogeneous differential form $\za$ is  provided by
$$
\degr{\za} =(\bar{\za}, |\za|)\;,
$$
with $\bar{\za}$ the ``super'' degree and  $|\za|$ the cohomological degree modulo $2$. There exist two conventions for the commutation relation of homogeneous differential forms \cite{DM99},
\begin{eqnarray*}
&\za\wedge \zb =(-1)^{\bar{\za}\bar{\zb}+|\za||\zb|} \zb\wedge \za\;,  & \text{Deligne sign rule},\\
&\za\wedge \zb =(-1)^{(\bar{\za}+|\za|)(\bar{\zb}+|\zb|)} \zb\wedge \za\;,  & \text{Bernstein-Leites sign rule}.
\end{eqnarray*}
They correspond to two distinct commutation factors on $\Z_2\times\Z_2$, which are related by a NS-functor $\Ivs$ (see \cite[p. 64]{DM99}), with NS-multiplier given by $ \zvs(\degr{\za},\degr{\zb})=(-1)^{\pari{\za}|\zb|}$. Moreover, the morphism
$$
\zvf: \, \Z_2\times\Z_2 \ni \degr{\za} \; \mapsto \;\bar{\za}+|\za| \in \Z_{2}\;,
$$
induces a parity on $\zW$, which  turns $\zW$ into a supercommutative algebra for the Bernstein-Leites sign rule.
}\end{exa}

\begin{exa}\label{quaternion}{\rm
Consider the algebra of quaternions $\qH\simeq\op{Cl}(0,2)$, with multiplication law
 \be \label{quaternions}\begin{tikzpicture}[
cell/.style={rectangle,draw=black},
space/.style={
    minimum height=1.5em,
    matrix of math nodes,
    row sep=-\pgflinewidth,
    column sep=-\pgflinewidth,
    column 1/.style={font=\ttfamily}},
text depth=0.5ex,
text height=2ex,
nodes in empty cells]
\matrix (m1) [space,
                    column 1/.style={minimum width=3em,nodes={cell,minimum width=3.5em}},
                    column 2/.style={minimum width=3em,nodes={cell,minimum width=3.5em}},
                    column 3/.style={minimum width=3em,nodes={cell,minimum width=3.5em}},
                    column 4/.style={minimum width=3em,nodes={cell,minimum width=3.5em}}]
{
\cdot & \qi  & \qj & \qk \\
\qi   & -1  & \qk & -\qj \\
\qj   & -\qk  & -1 & \qi \\
\qk   & +\qj  & -\qi & -1\\
};
\end{tikzpicture}
\ee
Setting $\degr{\qi}, \degr{\qj}, \degr{\qk}\in(\Z_2)^3$ as follows,
$$
\degr{\qi}:=(0,1,1)\,, \quad \degr{\qj}:=(1,0,1)\,, \quad \degr{\qk}:=(1,1,0)\;,
$$
the algebra $\qH$ turns into a real $((\Z_2)^3, \zl)$-commutative algebra with commutation factor $\zl(x ,y ):=(-1)^{\la x ,y \ra}=(-1)^{x _1y _1+x _2y _2+x _3y _3}$, where $x=(x_1,x_2,x_3)$ and $y=(y_1,y_2,y_3)\in (\Z_2)^3$.

The NS-multiplier $\zvs(x ,y ):=(-1)^{x _1(y _2+y _3)+x _2y _3}$ yields the following twisted multiplication:
\[\begin{tikzpicture}[
cell/.style={rectangle,draw=black},
space/.style={
    minimum height=1.5em,
    matrix of math nodes,
    row sep=-\pgflinewidth,
    column sep=-\pgflinewidth,
    column 1/.style={font=\ttfamily}},
text depth=0.5ex,
text height=2ex,
nodes in empty cells]
\matrix (m1) [space,
                    column 1/.style={minimum width=3em,nodes={cell,minimum width=3.5em}},
                    column 2/.style={minimum width=3em,nodes={cell,minimum width=3.5em}},
                    column 3/.style={minimum width=3em,nodes={cell,minimum width=3.5em}},
                    column 4/.style={minimum width=3em,nodes={cell,minimum width=3.5em}}]
{
\star  & \qi  & \qj & \qk \\
\qi   & +1  & -\qk & -\qj \\
\qj   & -\qk  & +1 & -\qi \\
\qk   & -\qj  & -\qi & +1\\
};
\end{tikzpicture}\]
Another choice of NS-multiplier, e.g., $\zvs(x ,y ):=(-1)^{x _1y _3+ x _2(y _1+y _2+y _3)}$, leads to a different product:
\[\begin{tikzpicture}[
cell/.style={rectangle,draw=black},
space/.style={
    minimum height=1.5em,
    matrix of math nodes,
    row sep=-\pgflinewidth,
    column sep=-\pgflinewidth,
    column 1/.style={font=\ttfamily}},
text depth=0.5ex,
text height=2ex,
nodes in empty cells]
\matrix (m1) [space,
                    column 1/.style={minimum width=3em,nodes={cell,minimum width=3.5em}},
                    column 2/.style={minimum width=3em,nodes={cell,minimum width=3.5em}},
                    column 3/.style={minimum width=3em,nodes={cell,minimum width=3.5em}},
                    column 4/.style={minimum width=3em,nodes={cell,minimum width=3.5em}}]
{
\star  & \qi  & \qj & \qk \\
\qi   & -1  & \qk & -\qj \\
\qj   & \qk  &+1  & \qi \\
\qk   &  -\qj & \qi & -1\\
};
\end{tikzpicture}\]
Since the two above tables are symmetric, both products ``$\star\,$'' are classically commutative.
}\end{exa}

\begin{exa}\label{Ex:Clifford} {\rm
According to Example \ref{Clifford1}, the Clifford algebra $\op{Cl}(p,q)$ is a $\left((\Z_2)^{n+1},\zl \right)$-commutative algebra ($n=p+q$), with $\zl(-,-) =\ks{-}{-}$.
Choosing  as multiplier the biadditive map
$$\zvs(x,y):=(-1)^{\sum_{i<j}x_iy_j},$$
we obtain the new commutation factor $\zls=\mathbf{1}$, and the algebra $\Ivs\lp\op{Cl}(p,q)\rp$ is then commutative.
}
\end{exa}

%%%%%%%%%%%%%%%%%%%%%%%%%%%%%%%%%%%%%%%%%%%%%%%%%%%%
%%%%%%%%%%%%%%%%%%%%%%%%%%%%%%%%%%%%%%%%%%%%%%%%%%%%
\section{Categories of Graded Modules} \label{Sec2}
%%%%%%%%%%%%%%%%%%%%%%%%%%%%%%%%%%%%%%%%%%%%%%%%%%%%
%%%%%%%%%%%%%%%%%%%%%%%%%%%%%%%%%%%%%%%%%%%%%%%%%%%%

From now on, $\cA$ denotes a $(\zG,\zl)$-commutative algebra.%, with $\zG$ a finitely generated abelian group.
We prove an analogue of Nekludova-Scheunert Theorem for the monoidal category of graded
$\cA$-modules. For graded modules over $\zG$-algebras we refer to \cite{NOy04}, as for basic notions of category theory, we refer to the appendix based on \cite{McL,Hov,Kelly}.

%%%%%%%%%%%%%%%
\subsection{Definitions}\label{DefMod}
%%%%%%%%%%%%%%%

%-Module over a grd-comm algebra\\
A \emph{left} (resp. \emph{right}) \emph{graded $\cA$-module} $M$ is a $\zG$-graded vector space  over $\K$, $M=\oplus_{\zb \in \zG}M^{\zb}$,  endowed with a compatible $\cA$-module structure, $\cA^{\za}M^{\zb} \subset M^{\za+\zb}$ (resp. $M^{\zb}
\cA^{\za} \subset M^{\za+\zb}$), for all $\za,\zb\in\zG$.
As $\cA$ is a $(\zG,\zl)$-commutative algebra, a left graded $\cA$-module structure on $M$ also defines a right graded $\cA$-module structure on $M$, e.g.,  by setting
\be \label{rlmod}  m\cdot a = \zl(\degr{m},\degr{a})\, a\cdot m\;,\ee
for any  $a\in\cA$, $m\in M\,$.
A graded $\cA$-module is a graded module over $\cA$ with compatible left and right structures, in the sense of \eqref{rlmod}.

%Morphism (0-degree)
A \emph{morphism of graded $\cA$-modules} is a map $\ell: M\to N$ which is $\cA$-linear and of degree $0\,$,
\beas \ell(m+m'a)=\ell(m)+\ell(m')a & \mbox{ and }& \degr{\ell(m)} =\degr{m},\eeas
for all homogeneous $m,m'\in M$ and $a\in \cA$.
Notice that ``$\cA$-linear'', here and thereafter, means right $\cA$-linear, as it is of common use in superalgebra theory.
We denote the set of such morphisms by $\Hom_{\cA}(M,N)$.
%category
Graded $\cA$-modules and corresponding morphisms form a %(locally small)
category denoted $\Moda$.

\begin{exa}{\rm
If $\cA$ is a supercommutative algebra (see Example \ref{exasuper}), the graded $\cA$-modules are usually called supermodules (see, e.g., \cite{DM99}).
Accordingly, the category $(\Z_{2})$-$\cat{Mod}_{\cA}$ is denoted by $\sMod_{\cA}\,$.}
\end{exa}

%%%%%%%%%%%%%%%
\subsection{Closed Symmetric Monoidal Structure}\label{otsMod}
%%%%%%%%%%%%%%%

The tensor product of graded bimodules over a $\zG$-algebra  and the internal $\iHom$ functor (see \cite{NOy04}) particularizes in $\Moda$ as follows.

\subsubsection{Tensor product}

Let $M,N$ be two graded $\cA$-modules. Their tensor product is defined as the quotient $\Z$-module
$$ M\ots_{\cA}N:= \EnsQuot{(M \ots_{\Z}N)}{I} $$
where $I=\op{span}\{ma\ots n -m\ots an \; \vert \; m\in M, n\in N, a\in \cA\}$. The $\zG$-grading
\be (M\ots_{\cA}N)^{\zg}:=\bigoplus\nolimits_{\za+\zb =\zg}\left\{\sum m\ots n \;\Big\vert \; m\in M^{\za}, n\in N^{\zb}\right\} \label{grmodtens}\ee
together with the following $\cA$-module structures,
\beas
a(m\ots n) :=(am)\ots n
&\quad \mbox{ and }\quad&
(m\ots n)\,a := m\ots (na)\;,
\eeas
turns $M\ots_{\cA}N$ into a graded $\cA$-module.
The tensor product $\ots_{\cA}$ endows the category $\Moda$ with a monoidal structure.
The isomorphism
\bemaps
 \zb_{M,N}:&  M\ots_{\cA} N  & \raa &   N \ots_{\cA} M   \\
 & m\ots n & \mapsto & \zl(\degr{m},\degr{n})\, n \ots m
 \eemaps
defines a braiding, which satisfies
$\zb_{N,M} \circ \zb_{M,N}=\id_{M\ots_{\cA}N}$ thanks to the properties of the commutation factor $\zl$.
Hence, the category $\Moda$ is a symmetric monoidal category.

\begin{rema}\label{algebratensoralgebra}
If $\cA$ and $\cB$ are two $(\zG,\zl)$-commutative algebras over $\K$, their tensor product $\cA \ots_{\K} \cB$ is $\zG$-graded according to \eqref{grmodtens}. The natural product
$$
(a_{1}\ots b_{1})(a_{2}\ots b_{2})= \zl(\degr{b_{1}},\degr{a_{2}})\, (a_{1}a_{2}) \ots (b_{1}b_{2})
$$
turns $\cA \ots_{\K} \cB$ into a $(\zG,\zl)$-commutative algebra.
\end{rema}

\subsubsection{Internal $\iHom$}
A monoidal category is \emph{closed} if the binary operation $\otimes$ defining the monoidal structure admits an adjoint operation, the so-called \emph{internal $\iHom$ functor}. By definition, this latter satisfies the natural isomorphism
$$
 \Hom(N\ots M, P) \simeq \Hom( N, \iHom(M,P) ),
$$
for all objects $M$, $N$ and $P$. The next result is an easy adaptation of the ungraded case and appears, e.g., in \cite[Prop.\ 2.4.9]{NOy04}.

\begin{prop}
The category  $\Moda$ of graded $\cA$-modules is a closed symmetric monoidal category. The internal $\iHom$ is given by
$$ \iHom_{\cA}(M,N) := \bigoplus\nolimits_{\zg\in \zG} \iHom^{\zg}_{\cA}(M,N)$$
where, for each $\zg\in \zG$,
$$ \iHom^{\zg}_{\cA}(M,N):=\{ f:M \raa N |\, \mbox{ $f$ is $\cA$-linear and } f(M^{\za})\subset N^{\za+\zg}, \text{ for all } \za \in \zG\}\;.$$
The $\cA$-module structure of  $\iHom_{\cA}(M,N)$ reads as
\beas
(af)(m):= a\cdot f(m) &\text{and} & (fa)(m):=\zl(\degr{a}, \degr{m}) f(m)\cdot a\;,
\eeas
where $f$ is a morphism and $m\in M$, $a\in \cA$ are homogeneous.
\end{prop}

\begin{comment}
\begin{proof}
 For every pair $M,N$ of graded $\cA$-modules, the set $\iHom_{\cA}(M,N)$ defined above is also an object of $\Moda$.
Moreover, the following map
\beas
 \Hom_{\cA}(N\ots_{\cA}M, P) & \xrightarrow{\sim} & \Hom_{\cA}\lp N, \iHom_{\cA}(M,P) \rp \\
 f & \mapsto &  \big( g:\, n \mapsto f(n\ots -) \big) \\
 \big( f:\, n \ots m \mapsto g(n)(m) \big) & \mapsfrom %stmaryrd package!
 & g
 \eeas
is an isomorphism, natural in $N$ and $P$. Hence, $\iHom_{\cA}(M,-)$ is the right adjoint functor of $-\ots M$.
\end{proof}
\end{comment}
By abuse of notation, we will refer to elements in $\iHom_{\cA}(M,N)$ as \emph{graded morphisms} and to elements in the subsets  $\iHom^{\zg}_{\cA}(M,N)$ as \emph{homogeneous morphisms} or more precisely as \emph{morphisms of degree $\zg\,$}. Elements in
$\iHom^{0}_{\cA}(M,N)=\Hom_{\cA}(M,N)$ will be called either \emph{morphisms} or \emph{morphisms of degree $0$}.

%%%%%%%%%%%%%%%
\subsection{An extension of the Nekludova-Scheunert Theorem}\label{Sec:NSthmMod}
%%%%%%%%%%%%%%%

Assume $\zvs\in\fS(\zl)$ is a NS-multiplier, so that Nekludova-Scheunert Theorem applies. The isomorphism of categories
$\Ivs$, defined in \eqref{Eq:Izvs},
yields an analogous isomorphism between categories of graded modules,
\bemaps
\widehat{\Ivs}:& \Moda &\xrightarrow{\sim}&   \Modabar \\
& M  &\mapsto & \und{M} \\
& \Hom_{\cA}(M,N) \ni f & \mapsto & f \in \Hom_{\cuA}(\und{M},\und{N})
\eemaps
Here, $\und{M}:=\widehat{\Ivs}(M)$ is the $\zG$-graded vector space $M$ endowed with the $\cuA$-module structure
\be\label{abarmod}
a\star m := \zvs(\degr{a},\degr{m})\, a\cdot m \;, \quad \forall  a\in \cuA, \forall m\in \und{M}\; .
\ee
Moreover, the regrading functor $\zF_{\zls}:\, \glsAlg \to \sAlg\,$ induces analogously a regrading functor on graded modules
\be\label{regrad}
\widehat{\zF}_{\zls}:\, \Modabar \to \sMod_{\cuA}
\ee
for all $(\zG,\zls)$-commutative algebra $\cuA$.
Hence, we get a faithful functor
$$ \Moda \xrightarrow{\widehat{\Ivs}} \Modabar \xrightarrow{\widehat{\zF}_{\zls}} \sMod_{\cuA}\;. $$
To investigate the properties of the above functors with respect to the symmetric monoidal structures
on $\Moda$, $\Modabar$ and $\sMod_{\cuA}$,  we need further notions of category theory.

%----------------------------------
\subsubsection{Closed Monoidal Functors}
%----------------------------------

In this section, we use notation from Appendix.
For example, monoidal categories are written as quintuple: $(\CC, \ots , \I, \za, r, l )$ refers to a monoidal category $C$ with tensor product $\ots$, identity object $\I$ and structural maps $(\za, r, l )$.

\begin{defi}\label{MonoidalFun}
A \emph{Lax monoidal functor} between two monoidal categories $(\CC, \ots , \I, \za, r, l )$ and $(\DD, \bots , \I',\za', r', l')$ is a triple $(F, u, \zt )$
which consists of
\begin{enumerate}[label=(\roman *)]
    \item a functor $F: \CC \to \DD$,
    \item a morphism $u: \I'\to F(\I)$ in $\DD$,
    \item a natural morphism $\zt$, i.e., a family of morphisms  in $\DD$ natural in $X,Y \in \Obj(\CC)\,$,
     $$\zt_{\hspace{-0.1cm}\phantom{.}_{X,Y}}: F(X) \bots  F(Y) \to F(X\ots  Y)\;,$$
\end{enumerate}
 such that the following diagrams commute.
\[
\begin{tikzpicture}
\matrix(m)[matrix of math nodes, row sep=4em, column sep=-5em]
{
 &\Big(F(X)\bots F(Y)\Big)\bots F(Z)  & \qquad \qquad \qquad \qquad \qquad \qquad \qquad \qquad \qquad& F(X)\bots \Big(F(Y)\bots F(Z)\Big) & \\
 F(X \ots Y)\bots F(Z)  & && & F(X) \bots F(Y\ots Z)  \\
 & F\Big((X \ots Y)\ots Z\Big) &&  F\Big(X \ots (Y\ots Z)\Big) & \\
};
\path[->]
(m-1-2) edge node[auto] {$ \za'_{\hspace{-0.1cm}\phantom{.}_{F(X),F(Y),F(Z)}}  $} (m-1-4)
        edge node[left] {$ \zt_{\hspace{-0.1cm}\phantom{.}_{X,Y}} \bots \id_{\hspace{-0.1cm}\phantom{.}_{F(Z)}}\;\;\; $} (m-2-1)
(m-2-1) edge node[below left] {$ \zt_{\hspace{-0.1cm}\phantom{.}_{X \ots Y,Z}} $} (m-3-2)
(m-3-2)edge node[auto] {$ F\za_{\hspace{-0.1cm}\phantom{.}_{X,Y,Z}} $} (m-3-4)
(m-1-4) edge node[right] {$\;\;\; \id_{\hspace{-0.1cm}\phantom{.}_{F(X)}} \bots \zt_{\hspace{-0.1cm}\phantom{.}_{Y,Z}} $} (m-2-5)
(m-2-5) edge node[below right] {$ \zt_{\hspace{-0.1cm}\phantom{.}_{X,Y\ots Z}} $}  (m-3-4)
; \end{tikzpicture}
\]
\beas
\begin{tikzpicture}
\matrix(m)[matrix of math nodes, row sep=4em, column sep=3em]
{
\I'\bots  F(X)  & F(X) \\
 F(\I) \bots  F(X) & F( \I \ots X) \\
};
\path[->]
(m-1-1) edge node[auto] {$ l'_{\hspace{-0.1cm}\phantom{.}_{F(X)}}  $} (m-1-2)
         edge node[left] {$ u \bots  \id_{\hspace{-0.1cm}\phantom{.}_{F(X)}} $} (m-2-1)
(m-2-2) edge node[right] {$ Fl_{\hspace{-0.1cm}\phantom{.}_{X}} $} (m-1-2)
(m-2-1) edge node[below] {$ \zt_{\hspace{-0.1cm}\phantom{.}_{\I,X}}  $}(m-2-2) (m-2-1)
;\end{tikzpicture}
& \quad &
\begin{tikzpicture}
\matrix(m)[matrix of math nodes, row sep=4em, column sep=3em]
{
F(X)\bots \I' & F(X) \\
F(X) \bots F(\I) & F(X\ots \I) \\
};
\path[->]
(m-1-1) edge node[auto] {$ r'_{\hspace{-0.1cm}\phantom{.}_{F(X)}}  $} (m-1-2)
         edge node[left] {$ \id_{\hspace{-0.1cm}\phantom{.}_{F(X)}} \bots u  $} (m-2-1)
(m-2-2) edge node[right] {$Fr_{\hspace{-0.1cm}\phantom{.}_{X}}$} (m-1-2)
 (m-2-1) edge node[below] {$   \zt_{\hspace{-0.1cm}\phantom{.}_{X, \I}}  $}(m-2-2) (m-2-1)
;\end{tikzpicture}
\eeas

A Lax monoidal functor is called a \emph{monoidal functor} if $u$ and all $\zt_{\hspace{-0.1cm}\phantom{.}_{M,N}}$ are isomorphisms.
\end{defi}

%A \emph{strict} monoidal functor is a monoidal functor for which the morphisms $u$ and $\zt_{X,Y}$ (for every $X,Y\in \op{Obj}(\CC)$) are identities.

%We need two more definitions.
\begin{defi}\label{funmono}
A (Lax) monoidal functor between  symmetric monoidal categories
$$
 (F,u,\zt):\; (\CC, \ots , \I,\za ,r , l, \zb) \to (\DD, \bots , \I',\za',r',l', \zb')
$$
 is called  \emph{symmetric} if it commutes with the braidings, i.e. if the following diagram commutes.
\[
\begin{tikzpicture}
\matrix(m)[matrix of math nodes, row sep=4em, column sep=6em]
{
F(X)\bots F(Y)  & F(Y)\bots F(X) \\
F(X \ots Y) & F(Y\ots X) \\
};
\path[->]
(m-1-1) edge node[auto] {$ \zb'_{\hspace{-0.1cm}\phantom{.}_{F(X),F(Y)}}  $} (m-1-2)
         edge node[left] {$ \zt_{\hspace{-0.1cm}\phantom{.}_{X,Y}} $} (m-2-1)
(m-1-2) edge node[auto] {$ \zt_{\hspace{-0.1cm}\phantom{.}_{Y,X}} $} (m-2-2)
(m-2-1) edge node[below] {$  F\zb_{\hspace{-0.1cm}\phantom{.}_{X,Y}} $}(m-2-2)
;\end{tikzpicture}
\]
\end{defi}

%++ closed monoidal -> construction of $G_2$ (between internal hom sets).
Given a (Lax) monoidal functor $(F,u,\zt)$ between two closed monoidal categories,
%(as in Definition \ref{funmono}),
one can construct a natural transformation $\zh$ via the internal $\iHom_{\DD}$ adjunction as follows (we omit indices of natural transformations):
%(with $\zt=\zt_{X,Y}$, $\eta=\eta_{X,Y}$),
\bea\nonumber
\Hom_{\DD}\Big( F\big(\iHom_{\CC}(X,Y)\big)\bots  F(X)\,,\,F(Y) \Big)
& {\simeq} &
 \Hom_{\DD}\Big( F\big(\iHom_{\CC}(X,Y)\big)\,,\, \iHom_{\DD}\big(F(X),F(Y)\big) \Big) \\ \label{etadef}
F(\mathrm{ev})\circ\zt & \leftrightarrow & \zh
\eea
where %$\zt$ stands for $\zt_{\iHom_{\CC}(X,Y),X}$ and $\zh$ for $\zh_{X,Y}$, in addition
$$
\mathrm{ev}:\,\iHom_{\CC}(X,Y)\otimes X \raa Y
$$
is the evaluation map. The triple $(F,u,\zh)$ is a \emph{Lax closed functor} between the closed categories $(\CC, \iHom_{\CC}, \I_{\CC})$ and $(\DD, \iHom_{\DD}, \I_{\DD})$ (see \cite{EK66} for the definition). It is a  \emph{closed functor} if moreover $\zh$ is an isomorphism.
Note that, even if $u$ and $\zt$ are isomorphisms, $\zh$ may not be one.

%----------------------------------
\subsubsection{NS-Functors on Modules}
%----------------------------------

Obviously, from the regrading functor $\widehat{\zF}_{\zls}$ defined in \eqref{regrad}, we obtain a symmetric monoidal functor,
$$(\widehat{\zF}_{\zls}, \id, \id):\, \Modabar \raa \sMod_{\cuA}\;,$$
which induces via \eqref{etadef} a closed functor $(\widehat{\zF}_{\zls},\id,\id)\,$.

%On the other hand, for the NS-functor $\widehat{\Ivs}$ introduced in \eqref{NSfunMod}, we obtain the following result.
\begin{thm}\label{thm:NSmod}
Let $\cA$ be a $(\zG,\zl)$-commutative algebra, $\zvs\in\fS(\zl)$ and $\cuA=\Ivs(\cA)$ the associated $\zG$-graded supercommutative algebra. For every $u\in \Aut_{\cuA}(\cuA)$, there exists a natural transformation $\zt$, such that
\beas (\widehat{\Ivs}, u,  \zt):\; \Moda \raa \Modabar \eeas
is a symmetric monoidal functor, which induces a closed functor $(\widehat{\Ivs},u,\zh)\,$.
By composition with the regrading monoidal functor $(\widehat{\zF}_{\zls}, \id, \id)$, we get a symmetric monoidal functor
$$
\lp\widehat{\zF}_{\zls}\circ \widehat{\Ivs}, \widehat{\zF}_{\zls}(u),\widehat{\zF}_{\zls}(\zt)\rp:\,\Moda \raa \sMod_{\cuA}
$$
which induces a closed functor $\lp\widehat{\zF}_{\zls}\circ \widehat{\Ivs}, \widehat{\zF}_{\zls}(u),\widehat{\zF}_{\zls}(\zh)\rp$.
\end{thm}

\begin{proof}
The map $u: \und{\cA} \to \und{\cA}$ being an $\cuA$-module morphism, it is completely characterized by $u(1)$, its value on the unit element $1\in\und{\cA}$. By construction, $u(1)$ is necessarily of degree $0$ and invertible.
For any pair $M,N$ of $\cA$-modules, we set
\bemaps
\zt:& \und{M} \bots  \und{N} &\to& \und{M\ots  N}\\
    & m \bots n &\mapsto & \zvs(m,n)\, u(1)^{-1} \star (m\ots n)
\eemaps
where $\und{M}=\widehat{\Ivs}(M)$. This family of morphisms of graded $\cA$-modules is natural both in $M$ and $N$ and $(\widehat{\Ivs},u,\zt)$ is easily checked to be a symmetric monoidal functor.
The induced natural transformation $\zh$ (see \eqref{etadef}) is then given by
\bemaps
\zh:& \und{ \iHom_{\cA}(M,N) } &\to & \iHom_{\und{\cA}}\lp \und{M},\und{N} \rp \\
% & & \\
 & f &\mapsto & \lp \zh(f):\,m\mapsto \zvs(f,m)\, u(1)^{-1} \star f(m)  \rp
\eemaps
which is clearly invertible.

The remaining statement follows from the rule of composition of monoidal and closed functors (see \cite{EK66}).
\end{proof}

As a consequence of Theorem \ref{thm:NSmod} and its proof, we have a closed functor $(\widehat{\Ivs},\id,\zh_{\zvs})\,$,
with $\cuA$-module isomorphisms
\bemap\label{eta}
 \zh_{\zvs}:& \widehat{\Ivs} \lp \iEnd_{\cA}(M)\rp &\raa &\iEnd_{\cuA}(\widehat{\Ivs}(M))\\
 & f& \mapsto &\lp \zh_{\zvs}(f):\,m\mapsto \zvs(f,m)\, f(m)\rp
\eemap
where $f$ and $m$ are homogeneous.
Note that we have $\iEnd_{\cA}(M) = \widehat{\Ivs}\lp \iEnd_{\cA}(M)\rp$ as algebras and as $\zG$-graded $\cA^{0}$-modules, but they admit distinct $\cA$- and $\cuA$-module structures (see \eqref{abarmod}).
A direct computation shows that Equation \eqref{etacomp} holds and  {\bf Theorem \ref{Thm:A}} in the Introduction follows.

\begin{comment}
proposition below, important in the following.
%{\cvd x-x-x}
\begin{prop}\label{lemetacomp}
Let $M\in \Obj(\Moda)$ and ``$\circ$'' denotes the composition of endomorphisms.
The map $\zh_{\zvs}$, defined in \eqref{eta}, is an isomorphism of $\cuA$-modules. However, $\zh_{\zvs}$ is not a morphism of algebras, namely
\end{prop}
As a consequence, the map $\zh_{\zvs}:\iEnd_{\cA}(M)\to\iEnd_{\cuA}(\widehat{\Ivs}(M))$ is an isomorphism of $\zG$-graded $\cA^{0}$-modules. %but not of $\cA$-modules.
\end{comment}

%%%%%%%%%%%%%%%%%%%%%%%%%%%%%
%%%%%%%%%%%%%%%%%%%%%%%%%%%%%
\section{Free graded modules and graded matrices}
%%%%%%%%%%%%%%%%%%%%%%%%%%%%%
%%%%%%%%%%%%%%%%%%%%%%%%%%%%%

Throughout this section, $\cA$ is a $(\zG, \zl)$-commutative algebra over $\K$. %, with $\zG$ a finitely generated abelian group.
We focus on free graded modules, investigate the notion of rank and transfer the NS-functors $\widehat{\Ivs}$ to algebras of graded matrices.

%%%%%%%%%%%%%%%%%%%
\subsection{Free graded structures}
%%%%%%%%%%%%%%%%%%%

We define and build free graded modules and free graded algebras. The latter are constructed from graded tensor algebras.

%------------------------------------------------
\subsubsection{Free graded Modules}
%------------------------------------------------

Let $S= \sqcup_{\zg\in \zG} S^{\zg}$ be a $\zG$-graded set. We denote by $\la S\ra_{\K}$ the $\zG$-graded vector space  generated by $S$ over $\K$.

The \emph{free graded $\cA$-module generated by $S$} is the tensor product $\la S\ra_{\cA}=\cA\ots_{\K}\la S\ra_{\K}\,$, with $\cA$-module structure given by
$a\cdot (b\ots m)= (ab) \ots m$, for all $a,b\in \cA$ and $m\in \la S\ra_{\K}\,$. The $\zG$-grading of $\la S\ra_{\cA}$ is defined by
$$
(\la S\ra_{\cA})^{\zg}=\bigoplus_{\za+\zb=\zg}\cA^{\za}\ots_{\K}(\la S \ra_{\K})^{\zb}\;,
$$
for all  $\zg\in \zG$.
A \emph{free graded $\cA$-module}  $M$ is a graded $\cA$-module
freely generated by a subset $S\subset \sqcup_{\zg\in\zG}M^{\zg}$, called a \emph{basis} of $M$ (cf. \cite{NOy04} for the notions of free, injective and projective graded modules over $\zG$-algebras).
An element $m\in M$ decomposes in a given basis $(e_i)_{i\in I}$ as follows,
$$
m= \sum_{i\in I} e_i \cdot m^i = \sum_{i\in I} n^i\cdot e_i \, ,%\quad \mbox{ with } n^{i}= \zl(\degr{e}_{i},\degr{m^{i}})\, m^{i}\;,
$$
where the map $i\mapsto m^i$ has finite support. Right and left components are linked via the relation (\ref{rlmod}) between left and right $\cA$-module structures. By convention, we only consider right components.
Note that $m\in M$ is homogeneous of degree $\zg$ if and only if its components $m^{i}$ are of degree $\degr{m^{i}}=\zg-\degr{e}_{i}$ for all $i\in I$.

\begin{lem}\label{lem:Ivsbasis}
Let $M$ be a free graded $\cA$-module with basis $S$, and $\zvs$ a NS-multiplier. The image of $M$ under the NS-functor $\widehat{\Ivs}$ is a free graded $\Ivs(\cA)$-module with the same basis $S\,$.
\end{lem}

\begin{proof}
By construction, $\Ivs:\,\cA=(A,\cdot) \mapsto \cuA=(A,\star )$ changes the product, ${a\star\, b:=\zvs(\degr{a},\degr{b})\, a\cdot b}\,$,
and $\widehat{\Ivs}:\, M \mapsto \widehat{\Ivs}(M)$ changes the graded  $\cA$-module structure into the graded $\cuA$-module structure $a\star m := \zvs(\degr{a},\degr{m}) \, a\cdot m\,$.
The decomposition of an element $m\in M$ in a basis $(e_{i})_{i\in I}$ varies accordingly:
$$ m=\sum_{i\in I} e_i\cdot m^i = \sum_{i\in I} e_i \star \lp \zvs(\degr{e}_i,\degr{m}^i)^{-1} m^i\rp\;,$$
so that $(e_{i})_{i\in I}$ is also a basis of the module $\widehat{\Ivs}(M)\,$.
\end{proof}

%------------------------------------------------
\subsubsection{Tensor algebras}
%------------------------------------------------

In this paragraph, we follow  \cite{Sch,Sch83}. The tensor algebra of a graded $\cA$-module $M$ is an $\cA$-module and an algebra,
$$  \mathrm{T}(M):= \bigoplus_{k\in \N} M^{\ots k} = \cA \opl M \opl (M \ots_{\cA} M) \opl (M \ots_{\cA} M  \ots_{\cA} M) \opl \ldots\;, $$
with multiplication given by the tensor product over $\cA\,$. The algebra $\mathrm{T}(M)$ is more precisely a $(\N \ts \zG)$-algebra with the following gradings:
\begin{itemize}
    \item the $\N$-grading, called \emph{weight}, given by the number of factors in $M\,$;
    \item the $\zG$-grading, called \emph{degree},  induced by the $\zG$-grading of the module $M$
    (see (\ref{grmodtens})).
\end{itemize}

\medskip
The \emph{graded symmetric algebra} on $M$ is the $\cA$-module and $(\zG,\zl)$-commutative algebra
%\be\label{symtensor}
$$\bigvee M:= \EnsQuot{\mathrm{T} (M)}{\cI}\;,$$%\ee
where $\cI$ is the homogeneous ideal of $\mathrm{T}(M)$ generated by the $\zG$-graded commutators
$$ [v,w]_\zl:=v\ots_{\cA} w - \zl(\degr{v}, \degr{w})\, w\ots_{\cA} v\;,\quad v,w\in M\; .$$
Analogously, the \emph{graded exterior  algebra} on $M$ is the $\cA$-module and $\zG$-algebra
$$\bigwedge\nolimits M := \EnsQuot{\mathrm{T}( M)}{\cJ}\;,$$
where $\cJ$ is the homogeneous ideal of $\mathrm{T}( M)$ generated by the $\zG$-graded anti-commutators
$$ v\ots_{\cA} w + \zl(\degr{v}, \degr{w})\, w\ots_{\cA} v\;,\quad v,w\in M\; .$$
\begin{comment}
As a result, the multiplication $\w$ in $ \bigwedge\nolimits M$ is $(\zG,\zl)$-anti-commutative
on homogeneous elements $x,y\in M$, i.e.
$$ x \w y = -\zl(\degr{x},\degr{y})\, y \w x  \;.$$
\end{comment}
%------------------------------------------------
\subsubsection{Free graded algebras}\label{freeAlg}
%------------------------------------------------
Let $S= \sqcup_{\zg\in \zG} S^{\zg}$ be a $\zG$-graded set and $\zl$ be a commutation factor over the grading group $\zG$.

We view $\K$ as a $(\zG,\zl)$-commutative algebra concentrated in degree $0$ so that definitions of the previous section apply to the $\zG$-graded $\K$-vector space $\la S\ra_{\K}\,$.
The \emph{free} \emph{$\zG$-algebra} generated by $S$ is the tensor algebra $\mathrm{T}( \la S\ra_{\K})\,$.
The \emph{free $(\zG,\zl)$-commutative algebra} generated by $S$ is the graded symmetric algebra
\be\label{poly} \K[S] := \bigvee \la S\ra_{\K}\;,\ee
which is the algebra of polynomials in the graded variables $(X_{i})_{i\in S}$.
\begin{comment}
Assume that $S$ is finite, with cardinality $r\in \N$, and write $S=\{X_{1}, X_{2}, \ldots , X_{r}\}\,$.
The elements of $\bigvee \la S\ra_{\K}$ are then of the form
$$
\sum_{u\in\N^{r}} a(u) X^{u}\;,
$$
where $a:\, \N^{r} \to \K$ has finite support, and $X^{u}= X_{1}^{u_{1}} \vee \ldots \vee  X_{r}^{u_{r}}$ is the usual multi-index notation.
Hence, $\bigvee \la S\ra_{\K}$ is the algebra of polynomials in the graded variables $(X_{i})$, and we write
\be\label{poly} \K[S] := \bigvee \la S\ra_{\K}\;.\ee
\end{comment}

We denote by $\overline{\la S\ra}_{\K}$ the vector space $ \la S\ra_{\K}$ with reversed grading
$$ \lp  \overline{\la S\ra}_{\K} \rp^{\zg} := \lp  \la S\ra_{\K} \rp^{-\zg} \;,\quad \text{for all } \zg\in\zG\;.$$
If $(\zG,\zl)$ is such that $\zG\odp=\emptyset$, we define the algebra of Laurent polynomials in graded variables $X_{i}\in S$ by
\be\label{Blambda}
\K[S,S^{-1}]_{\zl} := \EnsQuot{\lp \bigvee \la S\ra_{\K} \ots_{\K} \bigvee \overline{\la S\ra}_{\K}\rp}{\;\cI}
\ee
where $\cI$ is the ideal generated by the elements $(X_{i}\ots \overline{X}_{i})-1$, for $i=1,\ldots ,r$.
According to Remark \ref{algebratensoralgebra}, the space $\K[S,S^{-1}]_{\zl}$ is a $(\zG,\zl)$-commutative algebra, which is clearly a crossed product.
\begin{comment}
Its elements are of the form
$$ \sum_{u\in \Z^{r}} a(u)X^{u} $$
where $a:\, \Z^{r} \to \K$ has finite support, and $X_{i}^{u_{i}}$ denotes either $X_{i}^{u_{i}}\ots 1$ if $u_{i}\geq 0$ or $1\ots (\overline{X_{i}})^{ -u_{i}}$  if $u_{i}<0\,$.
\end{comment}

%%%%%%%%%%%
\subsection{Basis of free graded modules}\label{Sec:rank}
%%%%%%%%%%%
Let $(e_{i})_{i=1, \ldots , n}$ be a basis of a free and finitely generated graded $\cA$-module $M$.
The \emph{rank} of $(e_{i})$ is $n\in \N$, the \emph{degree} of $(e_{i})$ is the $n$-vector
$$\boldsymbol{\zn}:=(\degr{e}_{1},\ldots,\degr{e}_{n}) \in \zG^{n},$$
which depends on the order of the basis.
The sequence $(r_{\zg})_{\zg \in \zG}\in \N^{|\zG|}$,
where
$$ r_{\zg}= \sharp \{ e_{i} \;|\; \degr{e}_{i} =\zg \,,\; i=1, \ldots  ,n\}\,,$$
is called the \emph{$\zG$-rank} of $(e_{i})$. Obviously, we have $\sum_{\zg\in\zG}r_{\zg}=n$. Two bases have same $\zG$-rank if and only if they have the same degree up to permutation.

The underlying $\Z_{2}$-grading of $\zG$ %given by $\zG= \zG\evp \cup \zG\odp$
(see \eqref{G0G1}), induces a $\Z_{2}$-rank, also called \emph{superrank}, %namely
$$(r\evp, r\odp) = \lp \sum\nolimits_{\zg \in \zG\evp } r_{\zg} \;, \;\sum\nolimits_{\zg \in \zG\odp } r_{\zg} \rp\;.$$

\medskip
%Thanks to Nekludova-Scheunert functor and known results of superalgebra theory,
The rank and superrank are
 invariants of the graded module $M$.
\begin{prop}\label{prop:superrank}
Let $M$ be a free and finitely generated graded $\cA$-module. Any two bases of $M$ have the same rank and superrank, henceforth defining the rank and superrank of $M$.
\end{prop}

\begin{proof}
Let $M$ be a free graded $\cA$-module with two bases $(e_{i})$ and $(e'_{j})$.
Denote by  $\zvs \in \f{S}(\zl)$ a NS-multiplier.
By Lemma \ref{lem:Ivsbasis}, $(e_{i})$ and $(e'_{j})$ are then bases of the graded $\Ivs(A)$-module $\widehat{\Ivs}(M)\,$, with unchanged superranks.
Since $\widehat{\Ivs}(M)$ is a supermodule over the supercommutative algebra  $\Ivs(\cA)$, the superrank of all the bases of $\widehat{\Ivs}(M)$ are equal (see e.g. \cite[p. 114]{Var}). Hence, $(e_{i})$ and $(e'_{j})$ have the same superrank.
\end{proof}

For the $\zG$-rank, the situation is more involved and depends on the algebra $\cA\,$.
A Theorem of Dade \cite{Dad80} yields the following result.
\begin{prop}\label{DadeProp}%\cite{Dad80}
Let $M$ and $N$ be two free and finitely generated graded $\cA$-modules, with $\cA$ a strongly graded algebra. Then,
$M$ and $N$ are isomorphic if and only if they have same rank.
\end{prop}
Equivalently, if $\cA$ is strongly graded, a free graded $\cA$-module of rank $n\in\mathbb{N}$ admits bases of any degree $\boldsymbol{\nu}\in\zG^n$. As the grading group of a strongly graded algebra satisfies $\zG_{\bar{1}}=\emptyset$, rank and superrank are equal for its graded modules.
Set
$$
\zG_{\!\cA^{\ts}}:=\{ \zg \in \zG \;|\; \cA^{\zg} \cap \cA^{\ts} \neq \emptyset \},
$$
where $\cA^\times$ is the subgroup of invertible elements. If the algebra $\cA$ is strongly graded and  $\cA^0$ is a local ring, then $\cA$ is a crossed product (see \cite{Dad80}) and $\zG_{\!\cA^{\ts}}=\zG$. In general, $\zG_{\!\cA^{\ts}}$ is a subgroup of $\zG$ and $\zG_{\!\cA^{\ts}}\cap\zG_{\odp}=\emptyset$, as odd elements of $\cA$ are nilpotent.
Admissible degrees and $\zG$-ranks of bases of $M$ can be determined by studying the translation action of $(\zG_{\!\cA^{\ts}}\!\!)^{n}$ on $\zG^{n}$. To our knowledge, the point (ii) in the following proposition is new.

\begin{prop}\label{prop:rank}
Let $M$ be a free graded $\cA$-module of rank $n$, admitting a basis of degree $\boldsymbol{\zm}\in\;\zG^n$,
\begin{enumerate}[label=(\roman *)]
	\item For all $\boldsymbol{\zn}\in\; (\zG_{\!\cA^{\ts}}\!\!)^{n}+ \boldsymbol{\zm}$, there exists a basis $(f_{i})$ of $M$ of degree $\boldsymbol{\zn}\,$;
	\item %Assume $\cA^{0}$ is a field.
Up to reordering its elements, the degree $\boldsymbol{\zn}$ of any basis of $M$ satisfies $\boldsymbol{\zn} \in\; (\zG_{\!\cA^{\ts}}\!\!)^{n}+ \boldsymbol{\zm}\,$, if $\cA^0$ is a local ring.
\end{enumerate}
\end{prop}

In order to prove this, we need a definition and a lemma.

A \emph{maximal homogeneous} ideal is a proper homogeneous ideal $I$ such that any other homogeneous ideal containing $I$ is equal to $\cA$.

\begin{lem}\label{lem:maxhomideal}
Let $\zvs\in\f{S}(\zl)$, and $\cuA$ be a $(\zG,\zls)$-commutative algebra with $\cuA^{0}$ a local ring.
If $I$ is a maximal homogeneous ideal of $\cuA$ and $I^\zg=I\cap\cA^\zg$,
then we get %non-vanishing homogeneous elements in the algebra $\EnsQuot{\cuA}{I}$ are invertible, more precisely,
 $\cuA^{\zg}\setminus I^{\zg}=\cuA^{\zg}\cap\cuA^\ts$ for all $\zg\in\zG$.
\end{lem}

\begin{proof}[Proof of Lemma \ref{lem:maxhomideal}]
Let $I$ be a maximal homogeneous ideal of $\cuA$.
As $I$ is homogeneous, the algebra $\EnsQuot{\cuA}{I}$ is $\zG$-graded and $\lp \EnsQuot{\cuA}{I}\rp^{\zg}= \EnsQuot{\cuA^{\zg}}{I^{\zg}}\;$.
 Assume that $a\in \cuA^{\zg}\setminus I^{\zg}\,$. The ideal $(a)+I$ is homogeneous, hence $(a)+I=\cuA$ by maximality of $I\,$. This means there exists $b=\sum_{\za\in\zG}b_{\za}\in \cuA$ such that $ba+I=1+I\,$.
Projecting onto the $0$-degree part, we get $1-(b_{-\zg})a\in I^{0}\,$.
Since $\cuA^{0}$ is a local ring, $1-(b_{-\zg})a$ is also contained in the maximal ideal of $\cA^0$ and then $a$ is invertible in $\cA\,$. This means $\cuA^{\zg}\setminus I^{\zg}\subset \cuA^{\zg}\cap\cuA^\ts$ and the converse inclusion is obvious. Moreover, as $a$ is invertible in $\cA$, $a+I$ is invertible in $\cuA/I\,$.
%\end{enumerate}
\end{proof}
%\smallskip
\begin{proof}[Proof of Proposition \ref{prop:rank}]
%\begin{enumerate}[label=(\roman *)]
%	\item
(i) Let $(e_i)$ be a basis of $M$ of degree $\boldsymbol{\zm}\in\zG^n$ and $\boldsymbol{\zm}'=(\zm_{i}')\in (\zG_{\!\cA^{\ts}}\!\!)^{n}$. We can choose elements $a_i \in \cA^{\zm_{i}'}\cap \cA^{\ts}$ and set $ f_i:=e_ia_i $ for every  $i=1, \ldots , n$.
Then, $(f_i)$ form a basis of $M$ of degree $\boldsymbol{\zn}=\boldsymbol{\zm}+\boldsymbol{\zm}'$.
%	\item

\smallskip
(ii) Let $(e_{i})$ and $(f_{i})$ be two bases of $M$ of degrees $\boldsymbol{\zm},\boldsymbol{\zn}\in\zG^n$ respectively.
Let $\zvs\in\f{S}(\zl)\,$.
According to Lemma \ref{lem:Ivsbasis},
 the module $\und{M}:=\widehat{\Ivs}(M)$  is a free graded module over the $(\zG,\zl^\zvs)$-commutative algebra $\cuA:=\Ivs(\cA)$, with same bases $(e_{i})$ and $(f_{i})\,$.

Let $I$ be a maximal homogeneous ideal of $\cuA\,$.
Then, the quotient $\EnsQuot{\und{M}}{I\und{M}}$ is a graded $\lp\EnsQuot{\cuA}{I}\rp$-module. Furthermore, the bases $(e_{i})$ and $(f_{i})$ of $\und{M}$ induce bases $([e_{i}])$ and $([f_{i}])$ of $\EnsQuot{\und{M}}{I\und{M}}\,$, with unchanged degrees.
We have $[f_{i}]=\sum_{j}a_{ij}[e_{i}]\,$, with at least one non-vanishing coefficient. Up to reordering, we can suppose $a_{ii}\in \lp\EnsQuot{\cuA}{I}\rp\setminus \{0\}\,$.
Since $a_{ii}$ is homogeneous, of degree $\degr{a}_{ii}=\degr{f}_{i}-\degr{e}_{i}\,$, there is an invertible element of same degree in $\cuA$, by Lemma \ref{lem:maxhomideal}.
This means $\degr{f}_{i}-\degr{e}_{i} \in  \zG_{\!\cuA^{\ts}}\,$.  As homogeneous invertible elements of $\cA$ and $\cuA$ coincide, we get $\boldsymbol{\zn} \in\; (\zG_{\!\cA^{\ts}}\!\!)^{n}+\boldsymbol{\zm}\,$, up to reordering entries of $\boldsymbol{\zn}$.
%\end{enumerate}
\end{proof}

\medskip
As a direct consequence of (i) in Proposition \ref{prop:rank}, we get the following result.
\begin{cor}\label{cor:GAstar}
Let $\cA$ be a $(\zG,\zl)$-commutative algebra whose even part is a crossed product, i.e. $\zG_{\cA^{\ts}}=\zG\evp\,$. Then, two free graded $\cA$-modules are isomorphic if and only if they have same superrank.
\end{cor}

Note that Clifford algebras (Example \ref{Clifford1}) satisfy the condition $\zG_{\cA^{\ts}}=\zG\evp$, so that their free graded modules are characterized by the rank.
For the algebras introduced in Example \ref{exaWeil}, which generalize the dual numbers, the situation is opposite: $\zG_{\cA^{\ts}}=\{0\}$ and all bases of a free graded module have same $\zG$-rank. Over such an algebra , two free graded modules are isomorphic if and only if they have the same $\zG$-rank.
\begin{comment}
In general, one can always extend scalars to
$$ \cB=\cA\ots_{\K}\K[S,S^{-1}]_{\zl} $$
with $S$ a finite generating set of $\zG$. Such an algebra satisfies $\zG_{\cB^{\ts}}=\zG\evp\;$.
\end{comment}

%%%%%%%%%%%%%%%%%
\subsection{Graded matrix algebras}\label{Sec:grdmtrx}
%%%%%%%%%%%%%%%%%

In this paragraph, we introduce graded matrices over $(\zG,\zl)$-commutative algebras, partly following \cite{KNa84,COP}. See also \cite{NOy04,HWa12} for graded matrices over $\zG$-algebras, without commutativity assumption.
 Let $m,n\in\N$ and  $\boldsymbol{\zm}\in\zG^{m}$, $\boldsymbol{\nu}\in \zG^{n}$.
The space $\Mat(\boldsymbol{\zm}\ts\boldsymbol{\zn};\cA)$ of \emph{$\boldsymbol{\zm}\ts\boldsymbol{\zn}$ graded matrices} is the $\K$-vector space %$\Mat(n\ts m;\cA)$
of $m\ts n$ matrices, over the $(\zG,\zl)$-commutative algebra $\cA$, endowed with the $\zG$-grading:
$$
\Mat(\boldsymbol{\zm}\ts \boldsymbol{\zn};\cA)=\bigoplus_{\zg\in\zG}\Mat^{\zg}(\boldsymbol{\zm}\ts \boldsymbol{\zn};\cA)\;,
$$
where, for every $\zg\in\zG$,
 $$
\Mat^{\zg}(\boldsymbol{\zm}\ts \boldsymbol{\zn};\cA):=\left\{
X=(X^{i}_{\;\,j})_{i,j} \;|\; X^{i}_{\;\,j} \in  \cA^{\zg-\zm_{i}+\zn_{j}}
\right\}\;.
 $$
The elements of $\Mat^{\zg}(\boldsymbol{\zm}\ts \boldsymbol{\zn};\cA)$ are called \emph{homogeneous matrices of degree $\zg$}, or simply \emph{$\zg$-degree matrices}.

Let $p\in\N$ and $\boldsymbol{\pi}\in\zG^{p}$. The product of $m\ts n$ matrices with $n\ts p$ matrices defines a product
$$
\Mat(\boldsymbol{\zm}\ts \boldsymbol{\zn};\cA) \ts \Mat(\boldsymbol{\zn}\ts \boldsymbol{\zp};\cA) \raa \Mat(\boldsymbol{\zm}\ts \boldsymbol{\zp};\cA)
$$
which is compatible with the $\zG$-grading, i.e., $\Mat^{\za}(\boldsymbol{\zm}\ts \boldsymbol{\zn};\cA) \cdot \Mat^{\zb}(\boldsymbol{\zn}\ts \boldsymbol{\zp};\cA) \subset \Mat^{\za+\zb}(\boldsymbol{\zm}\ts \boldsymbol{\zp};\cA)\,$.
It turns the space $\Mat(\boldsymbol{\zn};\cA):=\Mat(\boldsymbol{\zn}\ts \boldsymbol{\zn}; \cA)$ into a $\zG$-algebra.
Moreover, $\Mat(\boldsymbol{\zn};\cA)$ admits graded $\cA$-module structures. We choose the following one, defined for all $\za\in \zG$ by
\begin{eqnarray}\nonumber
\cA^{\za} \ts \Mat(\boldsymbol{\zn};\cA) & \to & \Mat(\boldsymbol{\zn};\cA)\\ \label{A-Mod Matrix}
(a,X) & \mapsto & a\cdot X= \lp \begin{array}{ccc} \zl(\za,\zn_{1})a & & \\ & \ddots & \\ & & \zl(\za,\zn_{n})a \end{array}\rp  X
\end{eqnarray}

From the point of view of category theory, we can see $\Mat(\boldsymbol{\zn};-)$ as the functor
$$\Mat(\boldsymbol{\zn};-):\, \glAlg \raa \zG\mbox{-}\mathtt{Alg} $$
which assigns to each graded algebra $\cA$ the graded algebra of $\boldsymbol{\zn}$-square graded matrices, and to each $\zG$-algebra homomorphism  $f:\cA \to \cB$ the map
\bemap \label{Matfunctor}
 \Mat(f): & \Mat(\boldsymbol{\zn};\cA) &\to& \Mat(\boldsymbol{\zn};\cB) \\
 & \lp X^i_{\;j}\rp_{i,j} &\mapsto& \lp f(X^i_{\;j})\rp_{i,j}
\eemap

The group of invertible $\boldsymbol{\zn}$-square graded matrices is denoted by $\GL(\boldsymbol{\zn};\cA)$, and
analogously, we have the functor
\be\label{GLfunctor} \GL(\boldsymbol{\zn};-):\, \glAlg \raa \mathtt{Grp} \;.\ee
The subset of invertible homogeneous matrices of degree $\zg$ is denoted by $\GL^{\zg}(\boldsymbol{\zn};\cA)$. Similarly to the ungraded case, one gets
\begin{prop}\label{prop:iso}
Let $M$, $N$ be two free graded $\cA$-modules with bases $(e_{i})$ and $(e'_{j})$ of degrees $\boldsymbol{\zm}\in\zG^{m}$ and $\boldsymbol{\zn}\in\zG^{n}$ respectively.
Graded morphisms can be represented by graded matrices via the following isomorphism of $\zG$-graded $\K$-vector spaces,
\bemap\label{iso}
&\iHom_{\cA}(M,N)& \xrightarrow{\sim} & \Mat(\boldsymbol{\zn}\ts\boldsymbol{\zm};\cA)\\
& f &\mapsto & (F^{i}_{\;\;j})_{ij}
\eemap
where $F^{i}_{\;\;j}$ are determined by $f(e_j)=\sum_i e'_i \cdot F^i_{\,\;j}\,$.
If $M=N$ and $(e_{i})=(e'_{i})$, then \eqref{iso} is an isomorphism of $\zG$-algebras and graded $\cA$-modules, namely
$ \iEnd_{\cA}(M) \simeq \Mat(\boldsymbol{\zn};\cA) \;.$
\end{prop}
\begin{comment}
\begin{proof}
By definition of the coefficients $F^{i}_{\,\;j}$, the map \eqref{iso} is an isomorphism of $\K$-vector spaces. If $f$ is homogeneous of degree $\zg$, then $\degr{f(e_{j})}=\zg+\zm_{j}\,$. Hence, we get $\degr{F^{i}_{\,\;j}}=\zg-\zn_{i}+\zm_{j}\,$, so that the map \eqref{iso} preserves the $\zG$-grading.

If $M=N$ and $(e_{i})=(e'_{i})\,$, it is easily checked that the map \eqref{iso} is an algebra morphism. Moreover, for $a\in\cA^{\za}$, we have
$$
(af)(e_{j})= a\cdot f(e_{j}) = \sum_{i} e_{i} \,\lp\zl(\za,\zn_{i}) \, a\cdot F^{i}_{\,\;j}\rp\,,\quad \text{for all } j\;.
$$
In view of \eqref{A-Mod Matrix}, this proves that \eqref{iso} is an $\cA$-module morphism.
\end{proof}
\end{comment}

It is necessary to write matrix coefficients on the right of the basis vectors to get a morphism of algebra.
This leads to an atypical matrix representation for diagonal endomorphisms, but which fits with the $\cA$-module structure
defined in \eqref{A-Mod Matrix}. Namely, the endomorphism $f$, specified by $f(e_i)=~a_ie_i$ with $a_i\in\cA^{\degr{a}_i}$, reads in matrix form as
$$%\be\label{diagMatrx}
F=\lp
\begin{array}{ccc}
\zl(\degr{a}_1,\degr{e}_1)\,a_1 &&\\
&\ddots &\\
&&\zl(\degr{a}_n,\degr{e}_n)\,a_n\\
\end{array}
\rp\;.
$$%\ee

%However, we have seen in a previous paragraph that the bases of a free graded module do not necessarily have the same degree.

%%%%%%%%%%%%%%%%%
\subsection{Change of basis}
%%%%%%%%%%%%%%%%%
Let $M$ be a free graded $\cA$-module of rank $n$, with bases $(e_i)$ to $(e'_i)$ of degree $\boldsymbol{\zm}\in \zG^{n}$ and $\boldsymbol{\zn}\in\zG^{n}$ respectively.
A  basis transformation matrix from $(e_i)$ to $(e'_i)$ is a graded matrix $P\in \Mat^0(\boldsymbol{\zn}\ts\boldsymbol{\zm};\cA)$ with inverse $P^{-1}\in \Mat^0(\boldsymbol{\zm}\ts\boldsymbol{\zn};\cA)\,$. The graded matrices $F\in \Mat(\boldsymbol{\zm};\cA)$ and $F'\in\Mat(\boldsymbol{\zn};\cA)$ representing the same graded endomorphism $f\in \iEnd_{\cA}(M)$ are, as usual, linked through the equality
$$ F= P^{-1}F'P\;. $$
This provides the isomorphism of $\zG$-algebra, $\Mat(\boldsymbol{\zm};\cA) \simeq \Mat(\boldsymbol{\zn};\cA)\,$.

We use now particular change of basis to get graded matrices of specific forms.

%------------------------------------
\subsubsection{Permutation of basis and block graded matrices}\label{section:blockmatrix}
%------------------------------------

Assume $\zG$ is of finite order $N$. We choose an ordering on $\zG$, i.e.,  a bijective map
\beas
\{1, \ldots , N\}& \raa & \zG \\
u & \mapsto & \zg_{u}
\eeas
By permuting the elements of a basis $(e_i)$, of degree $\boldsymbol{\zn}\in \zG^n$ and $\zG$-rank $\vect{r}=(r_{\zg_u})_{u=1,\ldots,n}\,$, we obtain a basis $(e'_i)$ with ordered degree $\boldsymbol{\zm}\in\zG^n$. This means $\mu_i\leq \mu_{i+1}$ if $i=1,\ldots,n-1$.
\begin{comment}
satisfies
\be\label{ordereddegree}
\zm_i=
\begin{cases}
\zg_{1} & \mbox{ if } 1\leq r_{\zg_1}\\
\zg_2 & \mbox{ if }  r_{\zg_1}<i\leq r_{\zg_1}+ r_{\zg_2}\\
\ldots & \\
\zg_{N} & \mbox{ if } n-r_{\zg_{|\zG|} }< i \leq n \\
\end{cases}
\ee
In this new basis, a homogeneous graded matrix $X\in \Mat^x(\boldsymbol{\zm};\cA)$ is a block matrix
\end{comment}
In this new basis, a homogeneous graded matrix $X\in \Mat^x(\boldsymbol{\zm};\cA)$ is a block matrix
\be\label{Matrixblock}
X=\left(
\begin{array}{c|c|c}
\cX_{11} \Top &\; \ldots \; &\cX_{1N}\\[6pt]
\hline
\ldots \Top &\ldots&\ldots\\[3pt]
\hline
&&\\[-6pt]
\cX_{N1} &\ldots&\cX_{NN} \\[3pt]
\end{array}
\right)\,,
\ee
where each block is a matrix of the form $\cX_{uv}\in \Mat(r_{\zg_u}\ts r_{\zg_v};\cA )\,$, with entries in $\cA^{x-\zg_u+\zg_v}$.
This representation of the algebra of graded matrices is the one used in \cite{COP,Cov}, where it is referred to as $\Mat(\vect{r};\cA)\,$.

%------------------------------------
\subsubsection{Rescaling of basis and classical matrices}
%------------------------------------

Assume that the graded algebra $\cA$ contains an invertible element $\f{t}_{\za}\in \cA^{\za}$ of each even degree $\za$. Then, one can change the degree of a matrix through rescaling matrices
$$
P= \lp
\begin{array}{ccc}
\f{t}_{\za_1} & &\\
& \ddots & \\
&& \f{t}_{\za_n}
\end{array}
\rp
$$
with $\za_1,\ldots , \za_n \in \zG\evp$. This leads to the following result.
\begin{prop}\label{cor: specialcase}
Assume $\zG_{\cA^{\ts}}=\zG\evp$ and $(r\evp,r\odp)\in \N^{2}$.
For all $\boldsymbol{\zn}$ and $\boldsymbol{\zm}$ in $\zG\evp^{r\evp}\ts \zG\odp^{r\odp}$, we have an isomorphism of $\zG$-algebras
$$\Mat(\boldsymbol{\zm};\cA) \simeq \Mat(\boldsymbol{\zn};\cA)\;.$$
If more specifically $\boldsymbol{\zm} \in\zG\evp^{r\evp}$, the 0-degree matrices over $\cA$ identify to the usual matrices over the commutative algebra $\cA^0$, namely, we have an isomorphism of $\zG$-algebras
\be\label{Mat0}\Mat^{0}(\boldsymbol{\zm};\cA) \simeq \Mat(r\evp\,;\cA^{0}) \;.\ee
\end{prop}

\begin{proof}
Let $M$ a free graded $\cA$-module with a basis of degree $\boldsymbol{\zm}\in \zG\evp^{r\evp}\ts \zG\odp^{r\odp}$.
By Proposition~\ref{prop:iso}, we have $\Mat(\boldsymbol{\zm};\cA) \simeq \iEnd_{\cA}(M)\,$.
In view of Corollary \ref{cor:GAstar}, for all $\boldsymbol{\zn}\in \zG\evp^{r\evp}\ts \zG\odp^{r\odp}$, there exists a basis of $M$ of degree $\boldsymbol{\zn}$. Applying again Proposition \ref{prop:iso}, we deduce
$$\Mat(\boldsymbol{\zn};\cA) \simeq \iEnd_{\cA}(M) \simeq \Mat(\boldsymbol{\zm};\cA)\;.$$

As a direct consequence, if $\boldsymbol{\zm} \in\zG\evp^{r\evp}$ we have $\Mat(\boldsymbol{\zm};\cA) \simeq \Mat(\boldsymbol{0};\cA)\,$, with $\boldsymbol{0}=(0,\ldots,0)\in\zG^{r\evp}$.  By definition of the degree of a graded matrix, this leads to \eqref{Mat0}.
\end{proof}

The second part of the Proposition \ref{cor: specialcase} is well-known. This is a consequence of a Theorem of Dade \cite{Dad80},
which holds for matrix algebras over any strongly graded $\zG$-algebra.

\subsection{Natural isomorphisms on graded matrices}\label{secJs}

Let $\zvs\in \f{S}(\zl)$ and $M$ be a free graded $\cA$-module of rank $n$, with basis $(e_{i})$ of degree $\boldsymbol{\zn}\in\zG^{n}$.
We denote by $\cuA=\Ivs(\cA)$ the $\zG$-graded supercommutative algebra given by the NS-functor $\Ivs\,$.
Theorem \ref{Thm:A}, together with the isomorphism \eqref{iso}, yields an isomorphism of $\zG$-graded $\cA^{0}$-modules
\bemap \label{Js}
J_{\zvs,\cA}:& \Mat(\boldsymbol{\zn};\cA)& \raa & \Mat(\boldsymbol{\zn};\cuA)\\
    & X & \mapsto & J_{\zvs,\cA}(X)
\eemap
which explicitly reads,
on a homogeneous matrix $X$ of degree $x$, as
\be  \tag{\theequation a} \label{Jsentry}
\lp J_{\zvs,\cA}(X)\rp_{i,j} := \zvs(x, \zn_{j}) \zvs(\zn_{i}, x-\zn_{i}+ \zn_{j})^{-1} X^i_j \;,
\ee
for all  $i,j\in\{1,2,\ldots , n\}\,$.
The inverse of $J_{\zvs,\cA}$ is of the same form, namely $\lp J_{\zvs,\cA}\rp^{-1}=J_{\zd,\cA}\,$, with $\zd(x,y):=\zvs(x,y)^{-1}$ for all $x,y\in\zG$. From Theorem \ref{Thm:A}, we deduce the following
\begin{prop}\label{Cor:B}
%Let $\cA$ be a $(\zG,\zl)$-commutative algebra, $\zvs\in\fS(\zl)$, $\cuA=I_{\zvs(\cA)}$
% and $\boldsymbol{\zn}\in\zG^n$ with $n$ some positive integer.
The map $J_{\zvs,\cA}$ defined in (\ref{Js}) has the following properties:
\begin{enumerate}
\item for any homogeneous matrices $X,Y\in\Mat(\boldsymbol{\zn};\cA)$, of degree respectively $x$ and $y$,
    $$ J_{\zvs}(X Y) =\zvs(x,y)^{-1}\, J_{\zvs}(X)  J_{\zvs}(Y) \;;$$
\item for any graded matrices $X,Y\in\Mat(\boldsymbol{\zn};\cA)$, with either $X$ or $Y$ homogeneous of degree $0$,
    $$ J_{\zvs}(X Y) =\, J_{\zvs}(X) J_{\zvs}(Y) \;;$$
\item for any homogeneous invertible matrix $X\in \GL^{x}(\boldsymbol{\zn};\cA)$, $J_{\zvs}(X)$ is invertible and
    $$ J_{\zvs}(X^{-1})= \zvs(x,-x)\, \lp J_{\zvs}(X)\rp^{-1} \;.$$
\end{enumerate}
\end{prop}
The family of maps $J_{\zvs,\cA}$, parameterized by $\cA\in \Obj(\glAlg)$,
is a natural isomorphism
\[
\begin{tikzpicture}
\matrix(m)[matrix of math nodes, column sep=6em]
{
\glAlg & \zG\mbox{-}\cat{Mod}_{\K}\\
};
\draw[->](m-1-1) to[bend left=40] node[label=above:$\scriptstyle\Mat\big(\boldsymbol{\zn};   -  \big) $] (U) {} (m-1-2);
\draw[->] (m-1-1) to[bend right=40,name=D] node[label=below:$\scriptstyle \Mat\big(\boldsymbol{\zn}; \Ivs(  -  ) \big)$] (V) {} (m-1-2);
\draw[double,double equal sign distance,-implies,shorten >=5pt,shorten <=5pt]
  (U) -- node[label=right:$J_{\zvs}$] {} (V);
\end{tikzpicture}
\]
where the functor $\Mat\big(\boldsymbol{\zn}; \Ivs(  -  ) \big)$ is the composition of the functor $\Mat\big( \boldsymbol{\zn}; - \big)$ (recall (\ref{Matfunctor})) with the NS-functor $\Ivs\,$.
The restriction  of $J_{\zvs,\cA}$ to $0$-degree invertible matrices has further properties.
\begin{prop}\label{Jvsfun}
Let $\zvs\in \f{S}(\zl)$ and $\boldsymbol{\zn}\in \zG^{n}$. The family of maps
\be\label{JzvsGL0}  J_{\zvs, \cA}: \, \GL^0(\boldsymbol{\zn};\cA) \xrightarrow{\sim} \GL^0(\boldsymbol{\zn};\cuA)\;, \ee
 parameterized by $\cA\in \Obj(\glAlg)$, defines a natural isomorphism
\[
\begin{tikzpicture}
\matrix(m)[matrix of math nodes, column sep=6em]
{
\glAlg & \mathtt{Grp}\\
};
\draw[->](m-1-1) to[bend left=40] node[label=above:$\GL^0\big(\boldsymbol{\zn};   -  \big) $] (U) {} (m-1-2);
\draw[->] (m-1-1) to[bend right=40,name=D] node[label=below:$ \GL^0\big(\boldsymbol{\zn}; \Ivs( -  ) \big)$] (V) {} (m-1-2);
\draw[double,double equal sign distance,-implies,shorten >=5pt,shorten <=5pt]
  (U) -- node[label=right:$J_{\zvs }$] {} (V);
\end{tikzpicture}
\]
\end{prop}

\begin{proof}
By Proposition \ref{Cor:B}, the map (\ref{JzvsGL0}) is a group isomorphism.
The naturality of $J_{\zvs,\cA}$ follows from the diagram of groups
\[
\begin{tikzpicture}
\matrix(m)[matrix of math nodes, row sep=3em, column sep =6em]
{
\GL^0(\boldsymbol{\zn}; \cA) & \GL^0(\boldsymbol{\zn}; \cB)\\
\GL^0(\boldsymbol{\zn}; \cuA) & \GL^0(\boldsymbol{\zn}; \cuB) \\
};
\path[->]
(m-1-1) edge node[above]{$ \GL(f) $} (m-1-2)
            edge node[left]{$ J_{\zvs,\cA} $} (m-2-1)
(m-2-1) edge node[below]{$ \GL\big(\Ivs(f)\big) $} (m-2-2)
(m-1-2) edge node[right]{$ J_{\zvs,\cB} $} (m-2-2)
;
\end{tikzpicture}
\]
which is commutative for all $(\zG,\zl)$-commutative algebras $\cA, \cB$ and any $\zG$-algebra morphism $f:\cA\to \cB\,$.
\end{proof}

%%%%%%%%%%%%%%%%%%%%%%%%%%%%%%%%%%%%%%%%%%%%%%%%%%%%
%%%%%%%%%%%%%%%%%%%%%%%%%%%%%%%%%%%%%%%%%%%%%%%%%%%%
\section{Graded Trace} \label{Sec3}
%%%%%%%%%%%%%%%%%%%%%%%%%%%%%%%%%%%%%%%%%%%%%%%%%%%%
%%%%%%%%%%%%%%%%%%%%%%%%%%%%%%%%%%%%%%%%%%%%%%%%%%%%

Throughout this section, $\cA$ is a $(\zG, \zl)$-commutative algebra over $\K$ and $M$ is a free graded $\cA$-module of finite rank. %, with $\zG$ a finitely generated abelian group.
We use the closed monoidal NS-functors to pull-back the supertrace to graded endomorphisms of $M$.
 This yields the \emph{graded trace} introduced in \cite{Sch83} and extends the one introduced in \cite{COP}. We prove its characterization (Theorem \ref{Thm:C}) and provide a matrix formula for it.

%%%%%%%%%%%
\subsection{Definitions}\label{ColorLie}
%%%%%%%%%%%
A \emph{Lie color algebra}, or more precisely of \emph{$(\zG,\zl)$-Lie algebra} (see e.g. \cite{Sch}),
 is a $\zG$-graded $\K$-vector space, $\cL=\bigoplus_{\zg\in \zG}\cL^{\zg}$, endowed with a bracket $[-\, ,-]\,$, which satisfies the following properties, for all homogeneous $a,b,c\in\cL\,$,
\begin{itemize}
\item \emph{$\zl$-skew symmetry :} $\qquad[a,b]= -\zld{a}{b}\, [b,a] $ ;
\item \emph{Graded Jacobi identity : } $\qquad[a,[b,c]]= [[a,b],c] + \zld{a}{b}\, [b,[a,c]] $ .
\end{itemize}
A \emph{morphism of $(\zG,\zl)$-Lie algebras} is a degree-preserving $\K$-linear map $f: \cL \to \cL'$ which satisfies
$ f([a,b])=[f(a),f(b)]\,$, for every $a,b\in\cL\,$.

Examples of $(\zG,\zl)$-Lie algebras are $\zG$-algebras, e.g.\ the algebras $\cA$ or $\iEnd_\cA(M)$, endowed with the $\zG$-graded commutator
$$ [a,b]_{\zl}:= ab - \zld{a}{b}\, ba \;.$$
% for a graded $\cA$-module $M$.
Note that Lie superalgebras are $(\Z_2, \zlsup)$-Lie algebras. By definition, a {\em graded trace} is a degree-preserving $\cA$-linear map  $\Tr: \iEnd_{\cA}(M) \to \cA$, which is also a $(\zG,\zl)$-Lie algebra morphism.

%%%%%%%%%%%
\subsection{Proof of {\bf Theorem \ref{Thm:C}}}
%%%%%%%%%%%
Let $\zvs\in\mathfrak{S}(\zl)$ and $\Ivs$ be the corresponding NS-functor. We use the notation $\cuA=\Ivs(\cA)$, $\zl^\zvs=\zlsup$, and $\und{M}=\widehat{\Ivs}(M)\,$. In view of Theorem~\ref{Thm:A}, the linear isomorphism $\zh_\zvs:\iEnd_{\cA}(M) \to \iEnd_{\cuA}(\und{M})$ satisfies
$$\zh_\zvs([f,g]_{\zl})=\zvs\lp\degr{f},\degr{g}\rp^{-1}\,[\zh_\zvs(f),\zh_\zvs(g)]_{\zlsup}\;,$$
for any pair $f,g$ of homogeneous endomorphisms.
Hence, a graded trace $t:~\iEnd_{\cA}(M) \to \cA$ induces a  Lie superalgebra morphism, $t\circ (\zh_\zvs)^{-1} :\iEnd_{\cuA}(\und{M}) \to \cuA\,$. The latter is  $\cuA$-linear, since $t$ is $\cA$-linear.
As a consequence, the map $t\circ (\zh_\zvs)^{-1}$ is a multiple of the supertrace $\str$ by a non-zero element of $\cuA\evp\,$.
Since $t$ respects the $\zG$-degree, it is then equal to $\Tr=\str\circ \eta_\zvs$ up to multiplication by an element of $\cA^0\,$.

%%%%%%%%%%%%%%%%%%%
\subsection{Graded trace on matrices}
%%%%%%%%%%%%%%%%%%%
Let $n\in \N$ and $M$ be a free graded $\cA$-module of rank $n$ with a basis $(e_{i})$ of degree $\boldsymbol{\zn}\in\zG^{n}$.
Through the isomorphism \eqref{iso}, the graded trace $\Tr$ automatically defines an $\cA$-linear  $(\zG,\zl)$-Lie algebra morphism
$$\Tr:\;\Mat(\boldsymbol{\zn};\cA)\to \cA\;.$$
With the notation of Section \ref{secJs}, we have
$ \Tr=\str \circ J_{\zvs}\,$.
Hence, for a homogeneous matrix $X=(X^{i}_{\,\;j})_{i,j}\in \Mat^x(\boldsymbol{\zn};\cA)$, we find that
$$  \Tr(X) = \sum_i  \zl(\zn_i, x+\zn_i) X^i_{\;\,i} \;. $$
By linearity, the graded trace of an inhomogeneous matrix is the sum of the graded traces of its homogeneous components.
The above formula shows that $\Tr$ does not depend on the NS-multiplier~$\zvs$ chosen for its construction.

In the categorical language, the graded trace $\Tr$ defines a natural transformation
%between the functor $\Mat(\boldsymbol{\zn}; - )$ (see Section \ref{Sec:grdmtrx}) and the functor $\op{Lie}$ \eqref{funLie},
\[
\begin{tikzpicture}
\matrix(m)[matrix of math nodes, column sep=8em, row sep=4em]
{
\glAlg & \glLie \\
};
\draw[->, shorten >=5pt,shorten <=5pt](m-1-1) to[bend left=40] node[label=above:$\Mat(\boldsymbol{\zn};   -  ) $] (U) {} (m-1-2);
\draw[->,shorten >=5pt,shorten <=5pt] (m-1-1) to[bend right=40,name=D] node[label=below:$\op{Lie}( - )$] (V) {} (m-1-2);
\draw[double,double equal sign distance,-implies,shorten >=5pt,shorten <=5pt]
 (U) -- node[label=right:$\Tr$] {} (V);
\end{tikzpicture}
\]
where $\glLie$ is the category of $(\zG,\zl)$-Lie algebras and
the functor $\op{Lie}$ associates to each $(\zG,\zl)$-commutative algebra $\cA=(A,\cdot)$ the corresponding abelian $(\zG,\zl)$-Lie algebra $(A,[\cdot,\cdot]_\zl)$.

%{\cmg Noter le lien entre ``�tre transformation naturelle'' et ``avoir une formule 'universelle' ''} <- on utilisera apr�s pour formula explicite $\gdet ^0$.

%%%%%%%%%%%%%%%%%%%%%%%%%%%%%%%%%%%%%%%%%%%%%%%%%%%%
%%%%%%%%%%%%%%%%%%%%%%%%%%%%%%%%%%%%%%%%%%%%%%%%%%%%
\section{Graded Determinant of $0$-degree matrices} \label{Sec4new1}
%%%%%%%%%%%%%%%%%%%%%%%%%%%%%%%%%%%%%%%%%%%%%%%%%%%%
%%%%%%%%%%%%%%%%%%%%%%%%%%%%%%%%%%%%%%%%%%%%%%%%%%%%

Throughout this section, $\cA$ is a $(\zG, \zl)$-commutative algebra over $\K$. %, with $\zG$ a finitely generated abelian group.
Moreover, we restrict ourselves to the \emph{purely even} case, i.e., $\zG=\zG\evp\;$.
The NS-functors allow us to pull-back the determinant to $0$-degree invertible matrices over $\cA\,$. This yields the \emph{graded determinant} introduced in \cite{KNa84} and extends the one introduced in \cite{COP}. We prove its characterization (Theorem \ref{thm:Gdet0}), provide a new formula for it (Theorem \ref{FormulaGdeto}) and establish some of its properties.

\subsection{Proof of Theorem \ref{thm:Gdet0}}
%%%%%%%%%%%%%%%%%%%%%%%%%%
We first prove existence of the graded determinant $\Gdet^0$ and then its uniqueness.
We end up with a remark.

%------------------------------------
\subsubsection{Existence}
%------------------------------------

%Construction on $0$-degree invertible matrices (as natural transformation)
We define the natural transformation $\Gdet^0$ as the vertical composition (see Appendix \ref{vertical})
\[
\begin{tikzpicture}
\matrix(m)[matrix of nodes, column sep=25em]
{
$(\zG\evp , \zl)$-\texttt{Alg}  &\texttt{Grp}\\
};
\draw[->,shorten >=10pt,shorten <=10pt](m-1-1) to[bend left=60] node[pos=0.51, above] (U') {$\scriptstyle\GL^0\big(\boldsymbol{\zn};   -  \big) $} (m-1-2.north);
\draw[->,shorten >=5pt,shorten <=5pt](m-1-1.north east) to[bend left=15] node[label=above:$\scriptstyle\GL^0\big(\boldsymbol{\zn}; \Ivs(  -  ) \big) $] (U) {} (m-1-2.north west);
\draw[->,shorten >=5pt,shorten <=5pt] (m-1-1.south east) to[bend right=15,name=D] node[label=below:$\scriptstyle\big((\Ivs(  -  ))^0 \big)^{\ts}$] (V) {} (m-1-2.south west);
\draw[->,shorten >=10pt,shorten <=10pt] (m-1-1) to[bend right=50,name=D'] node[pos=0.51,below] (V') {$\scriptstyle\big((  -  )^0 \big)^{\ts}$} (m-1-2.south);
\draw[double,double equal sign distance,-implies, shorten >=18pt,shorten <=5pt]
  (U') -- node[label=above right:$J_{\zvs}$] {} (U);
\draw[double,double equal sign distance,-implies,shorten >=5pt,shorten <=5pt]
  (U) -- node[label=right:$\det_{\Ivs(  -  )}$] {} (V);
  \draw[double,double equal sign distance, shorten >=5pt,shorten <=18pt]
  (V) -- node[] {} (V');
\draw[double,double equal sign distance,-implies, shorten >=10pt,shorten <=10pt]
  (U'.south west) to[bend right=30] node[label=left:$\Gdet^0$] {} (V'.north west);
\end{tikzpicture}
\]
where $ J_{\zvs}$ is the natural isomorphism introduced in Proposition \ref{Jvsfun} and $\det_{\Ivs( - )}$ denotes the natural transformation obtained by whiskering (see Appendix \ref{whisker}) as follows
\[
\begin{tikzpicture}
\matrix[matrix of nodes,column sep=2cm] (cd)
{
 $(\zG\evp\,, \zl)$-\texttt{Alg} & $(\zG\evp,\mathbf{1})\,$-\texttt{Alg} & \texttt{Grp} \\
};
\draw[->]
(cd-1-1) edge node[pos=0.3,below] {$\scriptstyle\Ivs$} (cd-1-2);
%(cd-1-2) edge node[pos=0.65, below]{$\scriptstyle forget $} (cd-1-3);
\draw[->,shorten >=5pt,shorten <=5pt] (cd-1-2) to[bend left=30] node[label=above:$\scriptstyle\GL^0(\boldsymbol{\zn}; - )$] (U) {} (cd-1-3);
\draw[->] (cd-1-2) to[bend right=30,name=D] node[label=below:$\scriptstyle \lp( - )^{0}\rp^{\ts}$] (V) {} (cd-1-3);
\draw[double,double equal sign distance,-implies,shorten >=2pt,shorten <=2pt]
  (U) -- node[label=right:$\det$] {} (V);
\draw[->,shorten >=5pt,shorten <=5pt] (cd-1-1.north) to[bend left=50] node[label=above:$\scriptstyle\GL^0(\boldsymbol{\zn};\Ivs( - ))$] (U') {} (cd-1-3.north);
\draw[->,shorten >=5pt,shorten <=5pt] (cd-1-1.south) to[bend right=50,name=E] node[label=below:$\scriptstyle \lp(\Ivs( - ))^0\rp^{\ts}$] (V') {} (cd-1-3.south);
\draw[double,double equal sign distance,-implies,shorten >=10pt,shorten <=10pt]
  (U') to[bend right=55] node[pos=0.3, left] {$\det_{\Ivs( - )}$} (V');
\end{tikzpicture}
\]
Here, $(\zG\evp\,,1)\,$-\texttt{Alg} denotes the category of commutative algebras which are $\zG\evp\,$-graded.

By construction, $\Gdet^0$ satisfies the first axiom of Theorem \ref{thm:Gdet0}
and reads as \eqref{def:gdet0}.
The second axiom of Theorem \ref{thm:Gdet0} follows then from
 the explicit expression of $J_{\zvs}$ given in \eqref{Jsentry}.

%------------------------------------
\subsubsection{Uniqueness}
%------------------------------------

We need a preliminary result, which generalizes a Theorem of McDonald \cite{McD}.
\begin{lem}\label{generalMcD}
If $(\zD_{\cA})_{\cA \in \glAlg}$ is a family of maps
$$\zD_{\cA}:\, \GL^0(\boldsymbol{\zn};\cA) \to (\cA^0)^{\ts}$$
satisfying axiom \textsc{ai} of Theorem \ref{thm:Gdet0},
then there exists $t\in \Z$ and $\ze\in \{\pm1\}$ such that
%it factors through the above constructed $\Gdet^0$. More precisely, for any $X\in \GL^0(\boldsymbol{\zn};\cA)$, we have that
\be\label{powergdet}
\zD_{\cA}(X)=\ze \cdot \lp \Gdet^0_{\cA}(X)\rp^{t}\;,
\ee
for all $X\in \GL^0(\boldsymbol{\zn};\cA)$, with $\Gdet^0_{\cA}$ as in \eqref{def:gdet0}.
\end{lem}

\begin{proof}
We generalize the proof in \cite{McD}.

According to \eqref{poly}, we denote the $(\zG,\zl)$-commutative algebra of graded polynomials, with homogeneous indeterminates $X^i_{\;\,j} $ of degrees
$$  \degr{X^i_{\;\,j} } := \zn_{j}-\zn_{i}\;,\quad \text{for all } i,j=1,\dots,n \,,$$
 by $\K[X^i_{\;\,j} \,|\, i,j=1, \ldots n ]$. Thus, $X:=(X^i_{\;\,j})_{i,j}$ is a formal $\boldsymbol{\zn}\ts\boldsymbol{\zn}$ graded matrix of degree $0\,$.

From \eqref{def:gdet0} and \eqref{Jsentry}, it follows that
$\Gdet^0(X)$   is an element of $\K[X^i_{\;\,j} \,|\, i,j=1, \ldots n ]$, and by localizing this algebra at the powers of $\Gdet^0(X)$, we get the $(\zG,\zl)$-commutative algebra
$$
T(\boldsymbol{\zn}):=\K\left[X^1_{\,\;1}, X^1_{\,\;2}, \ldots, X^n_{\,\;n}, \frac{1}{\Gdet^0(X)}\right]\;.
$$

For every object $\cA \in \Obj\big(\glAlg\big)\,$, we define the map
\bemaps
\zvy(\boldsymbol{\zn})_{\cA}:& \GL^0(\boldsymbol{\zn};\cA) & \raa & \Hom\left(T(\boldsymbol{\zn}), \cA\right) \\
    & (a^i_{\,\;j})_{ij} & \mapsto & \lp \f{a}:\, X^i_{\,\;j} \mapsto a^i_{\,\;j} \rp
\eemaps
%where  $\Hom$ denotes the categorical Hom set of $\glAlg$ (i.e. $0$-degree algebra homomorphisms).
This is an isomorphism of groups, which is natural in $\cA\,$. Indeed, for every $\zG$-algebra morphism $f\in \Hom(\cA,\cB)$, the following diagram of groups commutes.
\[
\begin{tikzpicture}
\matrix(m)[matrix of math nodes, row sep=3em, column sep =6em]
{
\GL^0(\boldsymbol{\zn}; \cA) & \GL^0(\boldsymbol{\zn}; \cB)\\
 \Hom(T(\boldsymbol{\zn}), \cA) & \Hom\left(T(\boldsymbol{\zn}), \cB\right) \\
};
\path[->]
(m-1-1) edge node[above]{$ \GL(f) $} (m-1-2)
            edge node[left]{$\zvy(\boldsymbol{\zn})_{\cA} $} node[right]{$\wr$} (m-2-1)
(m-2-1) edge node[below]{$ f\circ - $} (m-2-2)
(m-1-2) edge node[right]{$ \zvy(\boldsymbol{\zn})_{\cB} $}node[left]{$\wr$} (m-2-2)
;
\end{tikzpicture}
\]
This means that the covariant functor $\GL^0(\boldsymbol{\zn};-)$ is represented by the $(\zG,\zl)$-commutative algebra $T(\boldsymbol{\zn})$. In particular, the functor $\big((-)^{0}\big)^{\ts}$, which is equal to $\GL^{0}(\boldsymbol{\zm}; - )$ for $\boldsymbol{\zm}=0\in \zG$, is represented by $T(1)=\K\left[Y, \frac{1}{Y}\right]$, with indeterminate $Y$ of degree $0$.
We deduce the following bijection
\[\begin{array}{ccc}
\op{Nat}\Big( \GL^{0}(\boldsymbol{\zn};  - )\,,\, \big((-)^{0}\big)^{\ts}\Big)
&\xrightarrow{\sim}&
\op{Nat}\Big( \Hom(T(\boldsymbol{\zn}),-)\,,\, \Hom(T(1),-) \Big)\\
 \zD & \leftrightarrow & \zk:= \zvy(1)\circ \zD \circ \zvy(\boldsymbol{\zn})^{-1}
\end{array}\]
where $\op{Nat}( - , - )$ denotes the set of natural transformations between two functors.

On the other hand, thanks to Yoneda Lemma, there exists a bijection
 \[\begin{array}{ccc}
\op{Nat}\Big( \Hom(T(\boldsymbol{\zn}),-)\,,\, \Hom(T(1),-) \Big)
&\xrightarrow{\sim}&
 \Hom\lp T(1),T(\boldsymbol{\zn}) \rp \\
 \zk=\{\zk_{\cA} \,|\, \cA\in \glAlg \} & \leftrightarrow & \zvk
\end{array}\]
%between, namely, natural transformations from the functor $\Hom(T(\boldsymbol{\zn}),-)$ to the functor $\Hom(T(1),-)$, and $\zG$-algebra morphisms from $T(1)$ to $T(\boldsymbol{\zn})$.
This correspondence goes as follows: any natural transformation $\zk$ defines a $\zG$-algebra morphism $\zvk:=\zk_{T(\boldsymbol{\zn})}\lp \id_{T(\boldsymbol{\zn})}\rp$ and conversely, given an algebra morphism $\zvk\,$, a natural transformation is completely defined by
$$\zk_{\cA}(f) := f\circ \zvk\;,$$
for all $f\in \Hom(T(\boldsymbol{\zn}),\cA)$ and all object $\cA\in \Obj\big(\glAlg\big)\,$.

Since $ T(1)=\K\left[Y, \frac{1}{Y}\right]$, an algebra morphism  $\zvk \in  \Hom\lp T(1),T(\boldsymbol{\zn}) \rp$
is characterized by $\zvk(Y)$, which is necessarily an invertible element of $T(\boldsymbol{\zn})$. By construction, the only possibilities are
$$ \zvk(Y)=\pm \lp \Gdet^0(X)\rp^{t} \;,$$
with  $t\in \Z\,$.
Hence, via the two above bijections, we finally obtain that any
natural transformation $\zD: \GL^0(\boldsymbol{\zn};-) \Rightarrow \big((-)^0\big)^{\ts}$
satisfies \eqref{powergdet}.
\end{proof}

Let $\zD$ be a natural transformation  satisfying axioms \textsc{ai} and \textsc{aii} of Theorem \ref{thm:Gdet0}.
The preceding lemma forces $\zD$ to satisfy \eqref{powergdet}, and axiom \textsc{aii} yields then $ \zD=\Gdet^{0} $.
This concludes the proof of Theorem \ref{thm:Gdet0}.

\begin{rema}
The graded determinant introduced in \cite{KNa84} satisfies
axioms \textsc{ai} and \textsc{aii} of Theorem \ref{thm:Gdet0}.
Hence it coincides with $\Gdet^0$.
% If $\cA$ is a graded division algebra, which is graded-commutative, the same holds for the graded determinant introduced in \cite{HWa12}.
\end{rema}

%%%%%%%%%%%%%%%%%%%%%%%%%%
\subsection{Explicit formula}\label{sec:formulaGdet0}
%%%%%%%%%%%%%%%%%%%%%%%%%%%

Let $n\in \N$, $\op{S}_n$ be the group of permutations of $\{1, \ldots, n \}$ and $\zs\in \op{S}_n\,$.
The decomposition of the permutation $\zs$ into disjoint cycles can be written in terms of an auxiliary permutation $\widehat{\zs}\in\op{S}_n$ as follows,
\be\label{cycledecomp}
 \zs=\Big(\widehat{\zs}(1) \, \widehat{\zs}(2) \, \ldots
\widehat{\zs}(i_1)\Big)\Big(\widehat{\zs}(i_1+1) \, \widehat{\zs}(i_1+2) \, \ldots \widehat{\zs}(i_2)\Big)\cdots \Big(\widehat{\zs}(i_{m-1}+1) \, \widehat{\zs}(i_{m-1}+2) \, \ldots \widehat{\zs}(i_m)\Big)\;,
\ee
with $i_1<\ldots <i_m=n\,$.
This means
\be\label{ordre}
\begin{cases}
\zs\lp\widehat{\zs}(k)\rp = \widehat{\zs}(k+1)& \mbox{ if } k\notin \{i_1, \ldots,i_m\}\\
\zs\lp\widehat{\zs}(i_j)\rp = \widehat{\zs}(i_{j-1}+1) & \mbox{ for } j=1, \ldots,m
\end{cases}
\ee
with $i_0=0$ by convention.
Such a permutation $\widehat{\zs}$ is called an \emph{ordering} associated to $\zs$.
\begin{exa}
The permutation
$$
\zs=\lp \begin{array}{ccccc}1&2&3&4&5\\4&5&2&1&3 \end{array}\rp \in \op{S}_5
$$
can be written in cycle notation as
\beas
\zs=(14)(253) & \mbox{ or }& \zs=(532)(14).
\eeas
The corresponding orderings are respectively
\beas
\widehat{\zs}=\lp \begin{array}{ccccc}1&2&3&4&5\\1&4&2&5&3 \end{array}\rp
& \mbox{ and }&
\widehat{\zs}=\lp \begin{array}{ccccc}1&2&3&4&5\\5&3&2&1&4\end{array}\rp
\eeas
\end{exa}

Using the above notion of ordering, we derive a formula for the graded determinant $\Gdet^0$, extended to $\Mat^0(\boldsymbol{\zn};\cA)$ via its defining equation  \eqref{def:gdet0}.
\begin{thm}\label{FormulaGdeto}
Let $\boldsymbol{\zn}\in \zG^n$. The graded determinant of a $0$-degree matrix $X=(X^i_{\,\;j})\in \Mat^0(\boldsymbol{\zn};\cA)$  is given by the following formula
\be\label{exprgdet0}
\Gdet^0(X)= \sum_{\zs\in \op{S}_n} \op{sgn}(\zs) \, X^{\widehat{\zs}(1)}_{\;\;\zs\lp \widehat{\zs}(1) \rp} \cdot  X^{\widehat{\zs}(2)}_{\;\;\zs\lp \widehat{\zs}(2) \rp} \cdots  X^{\widehat{\zs}(n)}_{\;\;\zs\lp \widehat{\zs}(n) \rp}
\; ,\ee
where the right-hand side does not depend on the
 chosen orderings $\widehat{\zs}$ associated to each $\zs\in\op{S}_n$.
\end{thm}

\begin{proof}
Let $\zvs\in \f{S}(\zl)\,$, and define $(\und{X}^i_{\,\;j})=J_{\zvs}(X)\,$.
By Theorem \ref{thm:Gdet0}, the graded determinant satisfies $\Gdet^{0}_{\cA}(X)=\det_{\cuA}(\und{X})\,$. Since the product ``$\star\,$'' in $\cuA$ is commutative, we obtain
$$
\Gdet^0_\cA(X)= \sum_{\zs\in \op{S}_r} \op{sgn}(\zs)\, \und{X}^{\widehat{\zs}(1)}_{\;\;\zs\lp \widehat{\zs}(1) \rp} \star   \und{X}^{\widehat{\zs}(2)}_{\;\;\zs\lp \widehat{\zs}(2) \rp} \star \cdots \star \und{X}^{\widehat{\zs}(r)}_{\;\;\zs\lp \widehat{\zs}(r) \rp}\;,
$$
for any ordering $\widehat{\zs}$ associated to $\zs$.

Let us now consider one monomial
$ \und{X}^{\widehat{\zs}(1)}_{\;\;\zs\lp \widehat{\zs}(1) \rp} \star   \und{X}^{\widehat{\zs}(2)}_{\;\;\zs\lp \widehat{\zs}(2) \rp} \star \cdots \star \und{X}^{\widehat{\zs}(r)}_{\;\;\zs\lp \widehat{\zs}(r) \rp} $.
The permutation $\zs$ admits a decomposition into disjoint cycles as in \eqref{cycledecomp}, therefore
\beas
 \prod_{j=1}^{n} \und{X}^{\widehat{\zs}(i)}_{\;\;\zs\lp \widehat{\zs}(i) \rp}&=&
\lp\prod_{j=1}^{i_1} \und{X}^{\widehat{\zs}(j)}_{\;\;\zs\lp \widehat{\zs}(j) \rp}\rp
\star
\lp\prod_{j=i_1+1}^{i_2} \und{X}^{\widehat{\zs}(j)}_{\;\;\zs\lp \widehat{\zs}(j) \rp}\rp
\star
\ldots
\star
\lp\prod_{j=i_{m-1}+1}^{n} \und{X}^{\widehat{\zs}(j)}_{\;\;\zs\lp \widehat{\zs}(j) \rp}\rp
\eeas
with $\prod$ denoting the product with respect to ``$\star\,$''.

Since $J_{\zvs}(X)$ is a matrix of degree $0\,$, the degree of its entries is given by  $\op{deg}\lp \und{X}^{\widehat{\zs}(i)}_{\;\;\zs\lp \widehat{\zs}(i) \rp}\rp=\zn_{\zs\lp \widehat{\zs}(i) \rp}-\zn_{\widehat{\zs}(i)}\,$. In view of \eqref{ordre}, the $m$ products above are then of degree $0$ and we deduce
\beas
 \prod_{j=1}^{n} \und{X}^{\widehat{\zs}(i)}_{\;\;\zs\lp \widehat{\zs}(i) \rp}&=&
\lp\prod_{j=1}^{i_1} \und{X}^{\widehat{\zs}(j)}_{\;\;\zs\lp \widehat{\zs}(j) \rp}\rp
\cdot
\lp\prod_{j=i_1+1}^{i_2} \und{X}^{\widehat{\zs}(j)}_{\;\;\zs\lp \widehat{\zs}(j) \rp}\rp
\cdot
\ldots
\cdot
\lp\prod_{j=i_{m-1}+1}^{n} \und{X}^{\widehat{\zs}(j)}_{\;\;\zs\lp \widehat{\zs}(j) \rp}\rp \;.
\eeas

Using $\und{X}^{i}_{\,\;j}= \zvs(\zn_{i}, \zn_{i}-\zn_{j})\, X^{i}_{\,\;j}$ (see \eqref{Jsentry}) and the definition of ``$\star\,$'', we get
\beas
&& \und{X}^{\widehat{\zs}(j)}_{\;\;\widehat{\zs}(j+1)} \star \lp \und{X}^{\widehat{\zs}(j+1)}_{\;\;\widehat{\zs}(j+2)} \star \cdots \star \und{X}^{\widehat{\zs}(i_1)}_{\;\;\widehat{\zs}(1)} \rp =\\
&=&
 \zvs\lp\zn_{\widehat{\zs}(j+1)}, \zn_{\widehat{\zs}(1)}- \zn_{\widehat{\zs}(j+1)}\rp
 \zvs\lp  \zn_{\widehat{\zs}(j)},  \zn_{\widehat{\zs}(1)}- \zn_{\widehat{\zs}(j)}\rp^{-1}\,
 X^{\widehat{\zs}(j)}_{\;\;\widehat{\zs}(j+1)} \cdot
 \lp \und{X}^{\widehat{\zs}(j+1)}_{\;\;\widehat{\zs}(j+2)} \star \cdots \star \und{X}^{\widehat{\zs}(i_1)}_{\;\;\widehat{\zs}(1)} \rp\;,
\eeas
 for any index $1\leq j < i_{1}$. By induction, we then obtain
$$
\und{X}^{\widehat{\zs}(1)}_{\;\;\zs\lp \widehat{\zs}(1) \rp} \star   \und{X}^{\widehat{\zs}(2)}_{\;\;\zs\lp \widehat{\zs}(2) \rp} \star \ldots \star  \und{X}^{\widehat{\zs}(i_{1})}_{\;\;\zs\lp \widehat{\zs}(i_{1}) \rp}
=
X^{\widehat{\zs}(1)}_{\;\;\zs\lp \widehat{\zs}(1) \rp} \cdot  X^{\widehat{\zs}(2)}_{\;\;\zs\lp \widehat{\zs}(2) \rp} \cdot \ldots \cdot   X^{\widehat{\zs}(i_{1} )}_{\;\;\zs\lp \widehat{\zs}(i_{1} ) \rp}\;.
$$
The same holds for the monomials built from the other cycles of $\zs$.
The claim follows.
\end{proof}

\begin{rema}
The use of specific orderings associated to permutations also appears in the Moore determinant, defined for Hermitian quaternionic matrices  (see e.g. \cite{Asl}).
\end{rema}

%------------------------------------
\subsection{Graded determinant on endomorphisms}
%------------------------------------

Let $M$ be a free graded $\cA$-module of finite rank.
The graded determinant carry over to endomorphisms.
\begin{prop}\label{GdetEnd}
Let $(e_i)$ be a basis of $M$ of degree $\boldsymbol{\zn}$. Then, the composite
\[
\begin{tikzpicture}
    \matrix(m)[matrix of math nodes, column sep=1.5em, row sep=2em]
{
\iEnd_{\cA}^{\,0}(M)& \Mat^0(\boldsymbol{\zn};\cA) && \cA^0\\
};
\path[->]
(m-1-1) edge node[above] {$\sim$} (m-1-2)
(m-1-2) edge node[above] {$\Gdet^0$} (m-1-4);
\end{tikzpicture}
\]
does not depend on the choice of basis of $M$ and thus defines the graded determinant on $0$-degree endomorphisms of $M$. It satisfies
$$
\Gdet^0(f)=\det\lp \widehat{\Ivs}(f) \rp \;,
$$
for all $ \zvs\in\f{S}(\zl)$,
where $\widehat{\Ivs}$ is the NS-functor associated to $\zvs$.
\end{prop}

\begin{proof}
By definition of $J_{\zvs}$ and of the classical determinant, the following diagram commutes.
\[
\begin{tikzpicture}
 \matrix(m)[matrix of math nodes, column sep=4em, row sep=2em]
{
 \Mat^0(\boldsymbol{\zn};\cA) & \Mat^0(\boldsymbol{\zn};\cuA) & \cuA^0=\cA^0\\
 \iEnd_{\cA}^{\,0}(M) & \iEnd_{\cuA}^{\,0}(\und{M})  & \\
};
\path[->]
(m-1-1) edge node[above] {$J_{\zvs}$} node[below]{$\sim$} (m-1-2)
(m-1-2) edge node[above]{$\det$} (m-1-3)
(m-2-1) edge node[above]{$\sim$} node[below]{$\widehat{\Ivs}$}(m-2-2)
(m-2-2) edge node[below right]{$\det$} (m-1-3);
\path[solid, shorten <= 7pt, shorten >= 7pt]
(m-1-1) edge node[auto] {$\wr $} (m-2-1)
(m-1-2) edge node[auto] {$\wr $} (m-2-2);
\end{tikzpicture}
\]
This means that $\Gdet^0(f)=\det\lp \widehat{\Ivs}(f) \rp$, for $f\in \iEnd^{\,0}_{\cA}(M)\,$. Hence, $\Gdet^0(f)$ is independent of the chosen basis.
\end{proof}

%------------------------------------
\subsection{Properties}
%------------------------------------

Most of the classical properties of the determinant over commutative rings still hold true for $\Gdet^0$. We provide some of them, see also \cite{KNa84}.

\begin{prop}\label{Prop:Inv}
Let $M$ be a free graded $\cA$-module of rank $n$, admitting bases of degrees $\boldsymbol{\zm},\boldsymbol{\zn}\in\zG^n$.
\begin{enumerate}
%    \item $u\in\iEnd^0(M)$ is injective if and only if $\Gdet^0(u)$ is not a divisor of zero in $\cA$;
    \item $u\in\iEnd^{\,0}(M)$ is bijective if and only if $\Gdet^0(u)$ is invertible in $\cA^0$;
    \item $\Gdet^0(P^{-1}XP)=\Gdet^0(X)$ for all $X\in\Mat^0(\boldsymbol{\zn};\cA)$ and all invertible $P\in\Mat^0(\boldsymbol{\zn}\ts \boldsymbol{\zm};\cA)$;
    \item $\Gdet^0$ is $\cA^0$-multilinear with respect to rows and columns.
\end{enumerate}
\end{prop}

\begin{proof}
\textrm{(1)} Invertibility of $0$-degree elements is preserved by the arrow functor $\widehat{\Ivs}\,$. Hence, the statement follows from the analogous result over commutative algebras.

\textrm{(2)} This is a direct consequence of Proposition \ref{GdetEnd}.

\textrm{(3)} This follows from the explicit formula \eqref{exprgdet0} of the graded determinant.

\end{proof}

The permutation  of the rows of  a graded matrix $X\in\Mat^0(\boldsymbol{\zn};\cA)$ by $\zs\in\op{S}_n$  produces a matrix $P_{\zs}\cdot X\in \Mat^0(\boldsymbol{\zm}\ts\boldsymbol{\zn};\cA)\,$, where  $P_{\zs}=(\zd^{i}_{\zs(j)})_{i,j}$ and $\zm_{i}=\zn_{\zs(i)}\,$.
In general, the matrix $P_{\zs}\cdot X$  is not homogeneous as a matrix in $\Mat(\boldsymbol{\zn};\cA)\,$. We circumvent this difficulty below, by modifying the permutation matrices.

Let us consider the auxiliary algebra $\cB:=\cA \ots_{\K}\K[S,S^{-1}]_{\zl}$ (see \eqref{Blambda}), with $S$ a finite generating set of $\zG$. Then, we can choose invertible homogeneous elements $\f{t}_{\zg}\in \cB^{\ts}\cap \cB^{\zg}$ for all $\zg\in\zG$.
\begin{lem}\label{permutationMatrx}
The following map is a group morphism,
\beas
\op{S}_n & \to & \GL^0(\boldsymbol{\zn};\cB)\\
\zs & \mapsto &  P(\zs):=\lp \zd^i_{\zs(j)}\, \f{t}_{\zn_i}^{-1}\f{t}_{\zn_j}\rp_{i,j}
\eeas
which satisfies $\Gdet^0(P(\zs))=\op{sgn}(\zs)$.
\end{lem}
\begin{proof}
The statements follow from direct computations.
\end{proof}
From this lemma and multiplicativity of $\Gdet^0$, we deduce the effect of permutation of rows or columns on $\Gdet^0$.
\begin{prop}\label{sigmaX}
Let $\boldsymbol{\zn}\in \zG^n$ and  $X\in \Mat^0(\boldsymbol{\zn};\cA)$. We have for every $\zs\in\op{S}_n\,$,
$$
\Gdet^0_{\cB}(P(\zs)\cdot X)=\Gdet^0_{\cB}( X\cdot P(\zs))=\op{sgn}(\zs)\,  \Gdet^0_{\cA}(X)\;.
$$
\end{prop}

%------------------------------------
\subsection{Graded determinant in particular bases}
%------------------------------------

Assume $\zG$ is of finite order $N$. According to Section \ref{section:blockmatrix},
one can order a degree $\boldsymbol{\zm}\in\zG^n$, so that graded matrices in $\Mat^0(\boldsymbol{\zm};\cA)$ are in block form \eqref{Matrixblock}.
The graded determinant $\Gdet^0$ coincides with the one introduced in \cite{COP} on such matrices.
%in the case of a $\lp(\Z_2)^n\evp, (-1)^{\la .\, ,\,.\ra}\rp$-commutative algebra (e.g. Example \ref{Ex:Clifford}) defined by means of quasideterminants.
\begin{prop}\label{propProperties}
Let $\boldsymbol{\zm}\in\zG^n$ be an ordered degree.% as in \eqref{ordereddegree}.
 The graded determinant $\Gdet^0$ has the following properties:
\begin{enumerate}
    \item for any block-diagonal matrix $D=(D_{uu})_{u=1,\ldots,N} \in \GL^0(\boldsymbol{\zm};\cA)\,$,
    $$ \Gdet^0
\lp \begin{array}{c|c|c}
        D_{11} &  & \\
        \hline
           &\ddots & \\
        \hline
           &    & D_{NN} \end{array} \rp
=\prod_{u=1}^{N} \det(D_{uu}) \;; $$
    \item for any upper or lower block-unitriangular matrix
        \[
        T= \lp \begin{array}{c|c|c}
        \I & \ast &\ast \\
        \hline
           &\ddots & \ast\\
        \hline
           &    & \I \end{array} \rp
        \quad \mbox{ or } \quad
        \lp \begin{array}{c|c|c}
        \I & & \\
        \hline
         \ast  &\ddots & \\
        \hline
         \ast  &  \ast  & \I \end{array} \rp
        \;,\]
        we have
        $$ \Gdet^0(T)=1 \;;$$
        \item if $p\in\N$ and $\cA$ is a $\big((\Z_2)^p, (-1)^{\la -,- \ra}\big)$-commutative algebra, the morphism $\Gdet^0_{\cA}\,$, from $\GL^0(\boldsymbol{\zm};\cA)$ to $(\cA^0)^\ts$, coincides with the graded determinant introduced in \cite{COP} .
            %(which was defined by means of quasideterminants in this case).
\end{enumerate}
\end{prop}

\begin{proof}
(1) If $D=(D_{uu})_{u=1,\ldots,N}$ is a block-diagonal matrix of degree $0$, then we deduce  from  the definition (\ref{Jsentry}) of $J_{\zvs}\,$, that $J_{\zvs}(D)$ is also block diagonal and $ \big(J_{\zvs}(D)\big)_{uu}=D_{uu}\,$, for all $u=1, \ldots , N\,$. Hence, the formula  $\Gdet^0(X)=\det\lp J_{\zvs}(X)\rp$ implies the result.

(2) Similarly, if $T$ is a block unitriangular  matrix of degree $0\,$, then $J_{\zvs}(T)$ is also one and we get $ \Gdet^0(T)=1\,$.

(3) In \cite{COP}, the authors introduce a notion of graded determinant
over $((\Z_2)^p, \ks{.\,}{.})$-commutative algebras $\cA$. This graded determinant is the unique group morphism from $\GL^0(\boldsymbol{\zm};\cA)$ to $(\cA^0)^{\ts}$ which satisfies the properties (1) and (2) above. Hence, it coincides with $\Gdet^0_\cA$ on $\GL^0(\boldsymbol{\zm};\cA)\,$.
\end{proof}

Assume that the graded algebra $\cA$ is a crossed product, i.e., it admits invertible elements $\mathfrak{t}_{\za}\in \cA^{\za}$ for each degree $\za$.
In this case, Proposition \ref{cor: specialcase} provides an algebra isomorphism, $\Mat^0(\boldsymbol{\zn};\cA) \simeq \Mat(n;\cA^0)\,$, via the change of basis $X\mapsto PXP^{-1}$ with
$$
P= \lp
\begin{array}{ccc}
\f{t}_{\zn_1}& &\\
&\ddots &\\
&&\f{t}_{\zn_n}
\end{array}
\rp\;.
$$
Then, Proposition \ref{Prop:Inv} leads to the graded determinant in this particular basis.
\begin{prop}\label{Prop:MatA0}
If $\cA$ is a crossed product, then, for all $X\in \Mat^0(\boldsymbol{\zn};\cA)\,$, we have
$$
\Gdet^0(X)=\det(PXP^{-1})\;.
$$
\end{prop}

%%%%%%%%%%%%%%%%%%%%%%%%%%%%%%%%%%%%%%%%%%%%%%%%%%%%
%%%%%%%%%%%%%%%%%%%%%%%%%%%%%%%%%%%%%%%%%%%%%%%%%%%%
\section{A family of determinant-like functions on graded matrices} \label{Sec4new2}
%%%%%%%%%%%%%%%%%%%%%%%%%%%%%%%%%%%%%%%%%%%%%%%%%%%%
%%%%%%%%%%%%%%%%%%%%%%%%%%%%%%%%%%%%%%%%%%%%%%%%%%%%

As in the previous section, all the considered pairs $(\zG,\zl)$ are such that  $\zG=\zG\evp$. % $\zG$ is finitely generated and

Let $\cA$ be a $(\zG,\zl)$-commutative algebra, $n\in\N$, $\boldsymbol{\zn}\in \zG^n$ and $\zvs$ a NS-multiplier. Using the NS-functor $I_{\zvs}$ and the induced isomorphisms $J_{\zvs}$ of $\zG$-graded $\cA^0$-modules, we extend the graded determinant to all graded matrices via the formula $\Gdet_\zvs:=\det\circ J_\zvs$.
\begin{comment}
\[
\begin{tikzpicture}
\matrix(m)[matrix of math nodes, column sep=3em]
{\Gdet_{\zvs}:
\Mat(\boldsymbol{\zn};\cA)
&
\Mat(\boldsymbol{\zn};I_{\zvs}(\cA))
&
I_{\zvs}(\cA)
&
\cA \\
};
\path[->]
(m-1-1) edge node[below]{$J_{\zvs}$} (m-1-2)
(m-1-2) edge node[below]{$\det$} (m-1-3);
(m-1-3) edge node[below]{$J_{\zvs}^{-1}$} (m-1-4);
\end{tikzpicture}
\]
\end{comment}
Hence, the maps $\Gdet_{\zvs}$ are natural transformations defined as the following vertical composition

\[
\begin{tikzpicture}
\matrix(m)[matrix of math nodes, column sep=25em]
{
\glAlg  &\cat{Set}\\
};
\draw[->,shorten >=10pt,shorten <=10pt](m-1-1) to[bend left=60] node[pos=0.51, above] (U') {$\scriptstyle\Mat\big(\boldsymbol{\zn};   -  \big) $} (m-1-2.north);
\draw[->,shorten >=5pt,shorten <=5pt](m-1-1.north east) to[bend left=15] node[label=above:$\scriptstyle\Mat\big(\boldsymbol{\zn}; \Ivs(  -  ) \big) $] (U) {} (m-1-2.north west);
\draw[->,shorten >=5pt,shorten <=5pt] (m-1-1.south east) to[bend right=15,name=D] node[label=below:$\qquad\scriptstyle{ \Mat(1; \Ivs(-))=\Ivs(  -  )}$] (V) {} (m-1-2.south west);
\draw[->,shorten >=10pt,shorten <=10pt] (m-1-1) to[bend right=50,name=D'] node[pos=0.51,below] (V') {$\qquad\scriptstyle{ \Mat(1;-) ={\rm forget}} $} (m-1-2.south);
\draw[double,double equal sign distance,-implies, shorten >=18pt,shorten <=5pt]
  (U') -- node[label=above right:$J_{\zvs}$] {} (U);
\draw[double,double equal sign distance,-implies,shorten >=5pt,shorten <=5pt]
  (U) -- node[label=right:$\det_{\Ivs(  -  )}$] {} (V);
  \draw[double,double equal sign distance, shorten >=5pt,shorten <=14pt]
  (V) -- node[label=right:] {} (V'); %{\cvd $J_{\zvs}^{-1}$}
\draw[double,double equal sign distance,-implies, shorten >=10pt,shorten <=10pt]
  (U'.south west) to[bend right=30] node[label=left:$\Gdet_{\zvs}$] {} (V'.north west);
\end{tikzpicture}
\]
As $\Gdet^0$, the natural transformations $\Gdet_{\zvs}$ can be equivalently defined on endomorphisms, via the formula $\Gdet_{\zvs}:=\det\circ\zh_\zvs\,$, with $\zh_\zvs$ defined in \eqref{eta}. In the case of a free graded module admitting bases of degrees $\boldsymbol{\zn},\boldsymbol{\zm}\in\zG^n$, this yields the following equality
$$
\Gdet_{\zvs}(P^{-1}XP)=\Gdet_{\zvs}(X)\;,
$$
for all $X\in\Mat(\boldsymbol{\zn};\cA)$ and all invertible $P\in\Mat^0(\boldsymbol{\zn}\ts\boldsymbol{\zm};\cA)\,$.

%%%%%%%%%%%%%%%%%%%%%%
\subsection{Proof of Theorem \ref{Thm:E}}
%%%%%%%%%%%%%%%%%%%%%%
The proof is in three steps.

\textbf{Step 1.}
We prove that $\Gdet_\zvs$ satisfies the properties  \ref{GdetGdet0}--\ref{Axiom3}, using
its defining formula $\Gdet_{\zvs}=\det\circ J_{\zvs}$\,.

Property \ref{GdetGdet0} is obvious.

Property \ref{Axiom1b} follows from Proposition \ref{Cor:B}.

Assume the matrices $X,Y,Z$ satisfy \eqref{XYZ}. Using their homogeneous decomposition and the definition of $J_{\zvs}$ (see \eqref{Jsentry}), we see that the rows of matrices $J_{\zvs}(X),J_{\zvs}(Y),J_{\zvs}(Z)$ satisfy again the relation \eqref{XYZ}. Property \ref{Axiom2} is then deduced from multi-additivity of the determinant.

Property \ref{Axiom3} follows from the definitions of $J_{\zvs}$ and of ``$\star\,$'' (see \eqref{star-cdot}), namely
\[\begin{split}
\Gdet_{\zvs}\lp \begin{array}{c|c} \;\;D\;\;& \\ \hline & \zl(\degr{c},\zn_n)\,c \end{array}\rp
= \det\lp \begin{array}{c|c} \;\;J_{\zvs}(D)\;\;& \\ \hline & c \end{array} \rp
&= c\, \star \,\det\lp J_{\zvs}(D) \rp \\
&= \zvs\Big(\degr{c}, \op{deg}\lp\Gdet_{\zvs}(D)\rp\Big)\, c\,\cdot\, \Gdet_{\zvs}(D)\;.
\end{split}\]

\medskip
%In the proof, we use an auxiliary algebra $\cB= \cA\ots \K[S,S^{-1}]_{\zl}$ (see \eqref{Blambda}), with $S$ a finite generating set of $\zG$.

\textbf{Step 2. }
Let us consider a map $\f{s}:\zG \ts \zG \to \K^{\ts}$. Assume there exists a natural transformation $\zD_{\f{s}}$ as in Theorem \ref{Thm:E}.
We prove that $\f{s}\in\f{S}(\zl)$, or equivalently, that $\f{s}$ satisfies the two following equations
\bea
& \zl(x,y)\f{s}(x,y)\f{s}(y,x)^{-1} = 1\;,&\label{1}\\
& \f{s}(x+y,z)\f{s}(x,y)= \f{s}(x,y+z)\f{s}(y,z) \;,&\label{2}
\eea
for all $x,y,z\in \zG$. To that end, we work over the algebra $\cB= \cA\ots \K[S,S^{-1}]_{\zl}$ (see \eqref{Blambda}), with $S$ a finite generating set of $\zG$, and use the $0$-degree matrix permutations $P(\zs)\in \GL^0(\boldsymbol{\zn};\cB)$, defined in Lemma \ref{permutationMatrx}.

We first prove \eqref{1}.
Let $\boldsymbol{\zn}\in\zG^2$ and $a,b\in\cB$ be two homogeneous invertible elements.
Then, using the transposition $\zs=(1\,2)$, we get
$$
\lp \begin{array}{cc} \zl(\degr{a},\zn_1)\,a & \\ & \zl(\degr{b},\zn_2)\,b\end{array}\rp
=
P(\zs)
 \lp\begin{array}{cc} \zl(\degr{b},\zn_1)\,b & \\ & \zl(\degr{a},\zn_2)\,a\end{array}\rp
 P(\zs)\;.
$$
From Properties  \ref{GdetGdet0}, \ref{Axiom1b} and Lemma \ref{permutationMatrx}, we deduce that
$$
\zD_{\f{s}}\lp \begin{array}{cc} \zl(\degr{a},\zn_1)\,a & \\ & \zl(\degr{b},\zn_2)\,b\end{array}\rp
=
\zD_{\f{s}}\lp \begin{array}{cc} \zl(\degr{b},\zn_1)\,b & \\ & \zl(\degr{a},\zn_2)\,a\end{array}\rp
$$
%\textsf{This can be generalized to higher orders//dimensions.}
Applying \ref{Axiom3} to the left and right hand side of the above equation, we obtain
$$ \f{s}(\degr{b},\degr{a}) ba =\f{s}(\degr{a},\degr{b})\zl(\degr{a},\degr{b})\, ba \;,$$
by $\zl$-commutativity.
Since $a$ and $b$ are invertible, this implies \eqref{1}.

We now prove \eqref{2}. Let $\boldsymbol{\zn}\in\zG^3$ and $D\in\Mat(\boldsymbol{\zn};\cB)$ be a diagonal matrix with homogeneous invertible entries $D^{i}_{\,\;i}=\zl(\degr{a}_i,\zn_i)\, a_i \in\cB^{\degr{a}_i}$, for $i=1,2,3$.
Using the permutation $\zs=(123)$, we define
 $$
 D':=P(\zs^{-1})DP(\zs)
    = \lp\begin{array}{ccc}\zl(\degr{a}_2,\zn_1)\,a_2 & & \\ & \zl(\degr{a}_3,\zn_2)\,a_3 & \\ & & \zl(\degr{a}_1,\zn_3)\,a_1 \end{array}\rp
 \;.$$
Reasoning as above, we obtain that
$
\zD_{\f{s}}(D')
= \zD_{\f{s}}(D)
$.
Furthermore, by means of \ref{Axiom3}, we can compute explicitly both $\zD_{\f{s}}(D')$ and $\zD_{\f{s}}(D)$.
Using moreover the equality \eqref{1}, we get
\[\begin{split}
\zD_{\f{s}}(D')&= \f{s}(\degr{a}_1,\degr{a}_3+\degr{a}_2)\f{s}(\degr{a}_3,\degr{a}_2)\, a_1a_3a_2 \\
&=  \f{s}(\degr{a}_1,\degr{a}_3+\degr{a}_2)\f{s}(\degr{a}_3,\degr{a}_2)\zl(\degr{a}_1,\degr{a}_3+\degr{a}_2)\, a_3a_2a_1 \\
&= \f{s}(\degr{a}_3+\degr{a}_2,\degr{a}_1)\f{s}(\degr{a}_3,\degr{a}_2)\, a_3a_2a_1 \,,
\end{split}\]
and similarly,
$$
\zD_{\f{s}}(D)= \f{s}(\degr{a}_3,\degr{a}_2+\degr{a}_1)\f{s}(\degr{a}_2,\degr{a}_1)\, a_3a_2a_1\;.
$$
Since $a_1,a_2,a_3$ are invertible, we finally obtain
\eqref{2}.
In conclusion, we have $\f{s}\in\f{S}(\zl)\,$.

\medskip

\textbf{Step 3.}
Let us consider a map $\f{s}\in\f{S}(\zl)\,$. Assume that $\zD_{\f{s}}$ is as in Theorem \ref{Thm:E}.
We prove that $\zD_{\f{s}}$ is equal to $\Gdet_{\f{s}}\,$.
For that, we show that $\zD_{\f{s}}$ is uniquely determined by Properties~\ref{GdetGdet0}~-\ref{Axiom3}, first over algebras
 of the form $\cB=\cA\ots_{\K}\K[S,S^{-1}]_{\zl}\,$, with $S$ a finite generating set of $\zG$, and then over all $(\zG,\zl)$-commutative algebras.

{a)}
        Let $n\in\N^*$, $\boldsymbol{\zn}\in \zG^n$ and $X\in\Mat(\boldsymbol{\zn};\cB)\,$. The rows of $X$ can be decomposed into homogeneous parts,
$$
X
= \left[ \begin{array}{c} \mathtt{x}^1 \\ \vdots \\ \mathtt{x}^n \end{array}\right]
= \left[ \begin{array}{c}
                                \sum_{\za_1\in \zG} \mathtt{x}^{1,\za_1}
                                \\ \vdots \\
                                \sum_{\za_n\in \zG} \mathtt{x}^{n,\za_n}
    \end{array} \right]\;.
$$
Applying \ref{Axiom2} inductively to such a decomposition,
we obtain that
\beas
  \zD_{\f{s}}(X)=\sum_{\bza \in \zG^n} \zD_{\f{s}}(X^{\bza}) \;,
& \mbox{ where } &
 X^{\bza}:= \left[\begin{array}{c} \mathtt{x}^{1,\za_1} \\ \vdots \\ \mathtt{x}^{n,\za_n} \end{array} \right]\;.
\eeas
Hence, $X^{\bza}$ is the matrix whose $k$-th row is the $\za_k$-degree component of the $k$-th row of $X$.
Let us now consider an arbitrary map $\zx:\, \{1,2, \ldots , n\} \to \{1,2, \ldots , n\}$ and denote by $X^{\bza}(\zx)$ the matrix whose entries are given by
$$ \lp X^{\bza}(\zx)\rp^i_{\,\; j} = (X^{\bza})^i_{\,\;j}\, \zd_{i,\zx(j)}\;, $$
where $\zd$ is the Kronecker delta.
%(i.e. $X^{\bza}(\zs)$ is the matrix  obtained from $X^{\bza}$ by "keeping only the entries related to $\zs$").
Thanks to \ref{Axiom2}, we then have for any multi-index $\bza\in \zG^n$,
$$
\zD_{\f{s}}(X^{\bza})= \sum_{\zx} \zD_{\f{s}}(X^{\bza}(\zx))\;.
$$
If the map $\zx$ is not bijective, the matrix $X^{\bza}(\zx)$ presents a whole row of zeros,
and then $\zD_{\f{s}}(X^{\bza}(\zx))=0$  by Property \ref{Axiom2}.
As a consequence, we get
\be\label{decompDX}
\zD_{\f{s}}(X)= \sum_{\bza\in \zG^n}\sum_{\zs\in S_n} \zD_{\f{s}}(X^{\bza}(\zs))\;.
\ee
Let $P(\zs)\in\GL^0(\boldsymbol{\zn};\cB)$ be the permutation matrix associated to $\zs\in\op{S}_n$, as introduced in Lemma \ref{permutationMatrx}.
By definition of $X^{\bza}(\zs)$ we have
$$
X^{\za}(\zs) = P(\zs) \cdot \cD(\za,\zs)\;,
$$
where $\cD(\za,\zs)$ is a diagonal matrix with homogeneous entries. %of degree $\za_{\zs(k)$}
%$$\lp\cD(\za,\zs)\rp^k_{\,\;k}= \f{t}_{r_k}^{-1}\f{t}_{r_{\zs(k)}} (X^{\za})^{\zs(k)}_{\,\;k}\;.$$
By Lemma \ref{permutationMatrx} and Properties \ref{GdetGdet0} and \ref{Axiom1b}, we end up with
\be\label{decompthm}
\zD_{\f{s}}(X) = \sum_{\bza\in \zG^n} \sum_{\zs\in S_n} \op{sgn}(\zs) \, \zD_{\f{s}}(\cD(\za,\zs))\;.
\ee
 By induction, Property \ref{Axiom3} fixes the values of $\zD_{\f{s}}(\cD(\za,\zs))$, and the uniqueness of the above map $\zD_{\f{s}}$ follows.

{b)} Let $\cA$ be an arbitrary $(\zG,\zl)$-commutative algebra, $S$ a finite generating set of $\zG$, and  $\zi:\, \cA \to \cB=\cA\ots_{\K}\K[S,S^{-1}]_{\zl}\,$ the canonical embedding.
Since  $\zD_{\f{s}}$ is a natural transformation,  we have
 $$ \zi(\zD_{\f{s}}(X)) = \zD_{\f{s}}(\GL(\zi)(X)) \;,$$
for all $X\in \Mat(\boldsymbol{\zn};\cA)\,$.
The right-hand side of the equation is fixed by point a) above and $\zi$ is injective, hence
 the map $\zD_{\f{s}}:\,~\Mat(\boldsymbol{\zn};\cA)\to \cA$ is unique and $\zD_{\f{s}}=\Gdet_{\f{s}}\,$.
This concludes the proof of Theorem \ref{Thm:E}.

%%%%%%%%%%%%%%%%%%%%%%%%%%
\subsection{$\Gdet_{\zvs}$ on homogeneous graded matrices}
%%%%%%%%%%%%%%%%%%%%%%%%%%

The restriction of $\Gdet_{\zvs}$ to homogeneous matrices has additional properties, which turns it into a proper determinant.

First of all, by construction, $\Gdet_{\zvs}$ preserves homogeneity. Indeed, for any $\zg\in\zG$ and any $\boldsymbol{\zn}\in\zG^{n}$, we have
\bemaps
\Gdet_{\zvs}(\Mat^{\zg}(\boldsymbol{\zn};\cA)) \subset \cA^{n\zg}\;.
\eemaps
To go further, we need the following preliminary results.

\begin{prop}\label{GdetsHomogeneous}
Let $\zvs \in \f{S}(\zl)$ and $\boldsymbol{\zn}\in\zG^n$.
\begin{enumerate}
    \item \label{weakmultipli}
For any couple of homogeneous matrices $X,Y \in \Mat(\boldsymbol{\zn};\cA)\,$, of degrees $x$ and $y$ respectively, we have
$$ \Gdet_{\zvs}(X Y)= \zvs(x,y)^{n(n-1)}\, \Gdet_{\zvs}(X) \cdot \Gdet_{\zvs}(Y) \;.$$
    \item \label{inverse}
For any invertible homogeneous matrix $X\in\GL^x(\boldsymbol{\zn};\cA)\,$, $\Gdet_{\zvs}(X)$ is invertible in $\cA$,
$$
\Gdet_{\zvs}(X^{-1})= \zvs(x,x)^{n(n-1)}\, \lp\Gdet_{\zvs}(X)\rp^{-1}
\;.$$
    \item \label{diagonalHomogeneous}
    For any homogeneous element $a\in\cA$ and homogeneous matrix $X\in\Mat^{x}(\boldsymbol{\zn};\cA)\,$, we have
     $$
  \Gdet_{\zvs}(a\cdot X)= \zvs(\degr{a},\degr{a})^{\frac{n(n-1)}{2}}\zvs(\degr{a},x)^{n(n-1)}\, a^n \cdot \Gdet_{\zvs}(X)\;.
     $$
    In particular, if $X=\I$ is the identity matrix, this reduces to
$ \Gdet_{\zvs}(a\cdot \I )= \zvs(\degr{a},\degr{a})^{\frac{n(n-1)}{2}}\, a^n\;.$
\end{enumerate}
\end{prop}

\begin{proof}
    The first result follows from the formula $\Gdet_\zvs=\det\circ J_{\zvs}$ and from the second point of Proposition~\ref{Cor:B}.

  The second result is a consequence of the first one.

    The third result relies on the first point and on the computation of $\Gdet_{\zvs}(a\cdot\I)\,$. Using displays \eqref{A-Mod Matrix} and \eqref{Jsentry} we get
    \[ \begin{split}
    \Gdet_{\zvs}(a\cdot \I)
    &= \det\lp\begin{array}{ccc} a &&\\ &\ddots&\\&& a \end{array}\rp
    =  a\,\star\, a \,\star\, \ldots \,\star\, a\\
    &=\lp \prod_{i<j}\zvs(\degr{a},\degr{a})\rp\, a\cdot a \cdot \ldots \cdot a
    = \zvs(\degr{a},\degr{a})^{\frac{n(n-1)}{2}}\, a^n\;.
    \end{split}
    \]
\end{proof}
The next two propositions show that $\Gdet_\zvs$ satisfies the two fundamental properties of a determinant on homogeneous matrices.
\begin{prop}
Let $\zvs \in \f{S}(\zl)$, $\boldsymbol{\zn}\in\zG^n$ and $X\in \Mat^{x}(\boldsymbol{\zn};\cA)\,$.
The homogeneous matrix $X$ is invertible if and only if $\Gdet_{\zvs}(X)$ is invertible in $\cA\,$.
\end{prop}

\begin{proof}
By  Proposition \ref{Cor:B}, a homogeneous matrix $X$ is invertible if and only if $J_{\zvs}(X)$ is invertible. Since $\Gdet_{\zvs}(X)=\det(J_{\zvs}(X))$, the proposition follows from the analogous result over commutative algebras.
\end{proof}

We introduce the set of homogeneous invertible matrices
$$
\hGL(\boldsymbol{\zn};\cA):=\bigcup_{\zg\in\zG}\GL^{\zg}(\boldsymbol{\zn};\cA)\;.
$$
Since multiplication of two homogeneous matrices gives a homogeneous matrix, $\hGL(\boldsymbol{\zn};\cA)$ is a group.
In particular, $\cup_{\zg\in\zG}\lp \cA^{\ts}\cap \cA^{\zg}\rp$ is a group.
Recall that $\K$ is the ground field of $\cA$.
\begin{prop}\label{Prop:GdetHomog}
For all $\zvs\in\f{S}(\zl)$, there exists a finitely generated subgroup $\mathbb{U} \leq \K^{\ts}$ such that
\bemaps
\Gdet_{\zvs}: &\hGL(\boldsymbol{\zn};\cA)& \to & \EnsQuot{ \bigcup\limits_{\zg\in\zG}\lp \cA^{\ts}\cap \cA^{\zg}\rp}{\;\mathbb{U}}
\eemaps
is a group morphism for all $\zn\in\bigcup\limits_{n\in\N^*}\zG^n$. Moreover, $\mathbb{U}$ can be chosen as the subgroup of $\K^{\ts}$ generated by $\{ \zvs(x,y)^2\,|\, x,y\in\zG \}$.
\end{prop}
\begin{proof}
Let $\mathbb{U}$ be the subgroup of $\K^{\ts}$ generated by  $\{\zvs(x,y)^2\,|\, x,y\in S \}$ in $\K^{\ts}$, with $S$ a finite generating set of $\zG$. The group $\mathbb{U}$ is finitely generated and contains $\{\zvs(x,y)^2\,|\, x,y\in\zG \}$, since $\zvs$ is biadditive. Hence we have $\zvs(x,y)^{n(n-1)}\in\mathbb{U}$, for all $x,y\in\zG$, $n\in\N^*$, and the result follows from  point (1) in Proposition \ref{GdetsHomogeneous}.
\end{proof}
If $\zG$ is a finite group,  $\mathbb{U}$ can be chosen as a finite group of roots of unity in  $\K^{\ts}$.

If $\zvs$ takes values in $\{\pm1\}$, the preceding results simplify.
In particular, one can take $\mathbb{U}=\{1\}$ and then obtain group morphisms
\bemaps
\Gdet_{\zvs}: &\hGL(\boldsymbol{\zn};\cA)& \to & \bigcup\limits_{\zg\in\zG}\lp \cA^{\ts}\cap \cA^{\zg}\rp\;.
\eemaps
Besides, the statement (3) in Proposition \ref{GdetsHomogeneous} reduces then to
\be\label{Gdetar}
 \Gdet_\zvs(a\cdot X)= \zvs(\degr{a},\degr{a})^{\frac{n(n-1)}{2}}\, a^n \cdot \Gdet_\zvs(X)\;.
 \ee

If $\cA$ admits homogeneous elements of each degree, then $X\in\Mat^{x}(\boldsymbol{\zn};\cA)$ can be written as $X=a\cdot X_0\,$, with $a\in \cA^{x}$ and $X_0\in\Mat^{0}(\boldsymbol{\zn};\cA)\,$.
The above equation can then be used as an ansatz to generalize $\Gdet^0$ to homogeneous matrices. If $n \equiv 0, 1 \!\mod{4}\,$, this ansatz simplifies into the naive one $\Gdet(a\cdot X_0)= a^n \, \Gdet^0(X_0)$, used in \cite{COP}, which does not depend on~$\zvs$. But, if $n \equiv 2,3 \!\mod{4}$, such a naive ansatz is not coherent with multiplication by scalars (see \cite{COP}) and one should use \eqref{Gdetar}.

\begin{rema}
Assume $\cA$ is a graded division ring and a graded commutative algebra, and write $\cA^\times_h:=\bigcup\limits_{\zg\in\zG}\lp \cA^{\ts}\cap \cA^{\zg}\rp$. The construction of a Dieudonn\'e determinant in \cite{HWa12} specifies then as a group morphism $$\hGL(\boldsymbol{\zn};\cA) \to   \cA^\times_h/[\cA^\times_h,\cA^\times_h],$$
where $[\cA^\times_h,\cA^\times_h]$ is generated by the products $aba^{-1}b^{-1}=\zl(\tilde{a},\tilde{b})$ of homogeneous elements. By Proposition \ref{Prop:GdetHomog}, such a morphism is provided by any graded determinant $\Gdet_\zvs$.
\end{rema}

%%%%%%%%%%%%%%%%%%%%%%%%%%
\subsection{$\Gdet_{\zvs}$ on quaternionic matrices: the good, the bad and the ugly}
%%%%%%%%%%%%%%%%%%%%%%%%%%

Recall that the quaternion algebra is the real algebra $\qH=\{ x+\qi y+\qj z+\qk t\,|\, x,y,z,t\in \R \}$ with multiplication law given in table \eqref{quaternions}. According to Example \ref{quaternion}, $\qH$ is a purely even $(\zG,\zl)$-commutative algebra, with $\zG=\left\{0,\degr{\qi},\degr{\qj},\degr{\qk}\right\}\leq (\Z_2)^3$ and $\zl=(-1)^{\la .,.\ra}$, the grading being given by
$$
\degr{\qi}:=(0,1,1)\,, \quad \degr{\qj}:=(1,0,1)\,, \quad\mbox{ and } \quad \degr{\qk}:=(1,1,0)\;.
$$
The graded $\qH$-modules structures on $\qH^n$ are in bijection with subspaces $V\subset \qH^n$ of real dimension $n$, the grading being
$$
\qH^n= V\oplus \qi V\oplus\qj V\oplus\qk V\;,
$$
where, e.g., $(\qH^n)^{\degr{\qi}}=\qi V$.

Let $X\in\Mat(n;\qH)$ be a quaternionic matrix, representing an endomorphism of $\qH^n$ in the basis $(e_i)$.
For every $\boldsymbol{\zn}\in\zG^n$, there exists a grading $V\subset\qH^n$ such that the basis vectors $e_i$ are homogeneous and $(e_i)$ has degree $\boldsymbol{\zn}$. Such a choice of grading turns $X$ into  a graded matrix, $X\in\Mat(\boldsymbol{\zn};\qH)$. This allows us to apply our determinant-like functions to quaternionic matrices.
% We compare them with the classical Dieudonn\'e determinant on homogeneous matrices and show some of their drawbacks on inhomogeneous matrices.

%------------------------------------
\subsubsection{The good: homogeneous matrices}
%------------------------------------
The Dieudonn\'e determinant over $\qH$ is the unique group morphism
\bemaps
\mathrm{Ddet}: &\GL(n;\qH) & \to & \EnsQuot{\mathbb{R}^\ts}{\;\{\pm 1\}}\;,
\eemaps
which satisfies
$$ \mathrm{Ddet}\lp \begin{array}{cccc} 1&&&\\&\ddots&&\\&&1&\\&&&a \end{array}\rp = a\cdot\{\pm 1\} \;,$$
for all $a\in\mathbb{R}^\ts$ (see \cite{Die43}). In particular, this determinant defines a group morphism on $\GL(\boldsymbol{\zn};\qH)$ which is independent of $\boldsymbol{\zn}\in\zG^n$. We use the following quotient map
\bemaps
\pi: &\mathbb{R}^\ts\oplus \qi \mathbb{R}^\ts\oplus \qj\mathbb{R}^\ts\oplus\qk\mathbb{R}^\ts & \to &\EnsQuot{\left(\mathbb{R}^\ts\oplus \qi \mathbb{R}^\ts\oplus \qj\mathbb{R}^\ts\oplus\qk\mathbb{R}^\ts\right)}{\;\{\pm 1,\pm\qi,\pm\qj,\pm\qk\}}\;,
\eemaps
the last group being isomorphic to $\EnsQuot{\mathbb{R}^\ts}{\,\{\pm 1\}}\,$.
\begin{prop}
For all $\zvs\in\f{S}(\zl)$ and all $\boldsymbol{\zn}\in\zG^n$, the following diagram of  groups commutes
\begin{equation}\label{Diag:Ddet-Gdet}
\begin{tikzpicture}
\matrix(m)[matrix of math nodes, row sep=3em, column sep=6em]
{
\hGL(\boldsymbol{\zn}; \qH) &\mathbb{R}^\ts\oplus \qi \mathbb{R}^\ts\oplus \qj\mathbb{R}^\ts\oplus\qk\mathbb{R}^\ts \\
 & \EnsQuot{\mathbb{R}^\times}{\;\{\pm 1\}} \\
};
\path[->]
(m-1-1) edge node[auto] {$ \Gdet_{\zvs} $} (m-1-2)
            edge node[below left] {$ \mathrm{Ddet} $} (m-2-2)
(m-1-2) edge node[auto] {$ \pi $} (m-2-2)
;\end{tikzpicture}
\end{equation}
\end{prop}
\begin{proof}
Let $X_0\in\GL^{0}(\boldsymbol{\zn}; \qH)$. According to Proposition \ref{Prop:MatA0}, there exists $P\in\Mat^0( \boldsymbol{0}\times\boldsymbol{\zn};\mathbb{H})$ such that $PX_0P^{-1}\in\Mat(n;\mathbb{R})$ and $\Gdet^0(X_0)=\det(PX_0P^{-1})$. By multiplicativity of the Dieudonn\'e determinant, we get $\Ddet(X_0)=\Ddet(PX_0P^{-1})$. Since $\Ddet$ is equal to the classical determinant (modulo the sign) on real matrices, the diagram \eqref{Diag:Ddet-Gdet} commutes if restricted to the subgroup $\GL^0(\boldsymbol{\zn}; \qH)\leq\hGL(\boldsymbol{\zn}; \qH)\,$.

If $X\in\GL^{\degr{\qi}}(\boldsymbol{\zn}; \qH)\,$, then $X=\qi\cdot X_0$ with $X_0\in\GL^{0}(\boldsymbol{\zn}; \qH)\,$. By Proposition \ref{GdetsHomogeneous}, we have
$$
  \Gdet_\zvs(X)= \zvs(\degr{\qi},\degr{\qi})^{\frac{n(n-1)}{2}}\, \qi^n \cdot \Gdet_{\zvs}(X_0)\;,
$$
and $\zvs(\degr{\qi},\degr{\qi})^{\frac{n(n-1)}{2}}\, \qi^n\in\{\pm 1,\pm\qi,\pm\qj,\pm\qk\}\,$. Besides, it is known that
$$
\mathrm{Ddet}(X)=\mathrm{Ddet}(\qi\cdot\mathbb{I})\mathrm{Ddet}(X_0)=\mathrm{Ddet}(X_0)\;.
$$
Analogous results hold for $X\in\GL^{\degr{\qj}}(\boldsymbol{\zn}; \qH)$ and $X\in\GL^{\degr{\qk}}(\boldsymbol{\zn}; \qH)\,$. This conclude the proof.
\end{proof}

%------------------------------------
\subsubsection{The bad: non-uniqueness of $\Gdet_{\zvs}$}
%------------------------------------

Let $X\in\Mat(n,\qH)$ be a quaternionic matrix representing an endomorphism of $\qH^n$ in a basis $(e_i)\,$. Once the basis $(e_i)$ receives a degree $\boldsymbol{\zn}\in\zG^n$, the matrix $X$ becomes a graded matrix $X\in\Mat(\boldsymbol{\zn};\qH)\,$.

The value of $\Gdet_{\zvs}(X)$ depends both on $\zvs\in\mathfrak{S}(\lambda)$ and $\boldsymbol{\zn}\in\zG^n$. By Lemma \ref{lem:biadditif}, the multipliers $\zvs\in\f{S}(\zl)$ are characterized by the values of $\zvs(\degr{\qi},\degr{\qi}), \, \zvs(\degr{\qi},\degr{\qj})$ and $\zvs(\degr{\qj},\degr{\qj})\,$, which can be either $1$ or $-1$.
Hence, for each of the $4^n$ possible degrees $\boldsymbol{\zn}\in\zG^n$, there are $8$ determinant-like functions $\Gdet_{\zvs}$ on $\Mat(\boldsymbol{\zn};\qH)\,$.

%------------------------------------
\subsubsection{The ugly: $\Gdet_{\zvs}$ on inhomogeneous matrices}
%------------------------------------
We compute the values of   $\Gdet_{\zvs}(X)$ for a quaternionic matrix $X\in\Mat(2,\qH)\,$. This shows that, indeed,  the value of  $\Gdet_{\zvs}(X)\,$, and even its vanishing, strongly depends on both choices: of multiplier $\zvs\in\f{S}(\lambda)$ and of degree $\boldsymbol{\zn}\in\zG^2$ of the basis.

First, we work with the invertible matrix
$$
X:=\lp \begin{array}{cc}1&\qj\\ \qj&1 \end{array}\rp \in\GL(2;\qH)\;,
$$
with inverse $X^{-1}=\frac{1}{2}\lp \begin{array}{cc}1&-\qj\\ -\qj&1 \end{array}\rp $.
The matrix $X$ is a homogeneous graded matrix of $\GL(\boldsymbol{\zn};\qH)$ if and only if the chosen degree $\boldsymbol{\zn}\in\zG^2$ is of the form $\boldsymbol{\zn}=(\boldsymbol{\zn}_1,\boldsymbol{\zn}_1+\degr{j})\,$, with $\boldsymbol{\zn}_1\in\zG$.
For such a degree, the matrix $X$ is of degree $0$ and, for all $\zvs\in\f{S}(\lambda)$, we get
$$
\Gdet_{\zvs}(X)=\Gdet^0(X)=2\;.
$$
However, for a different $\boldsymbol{\zn}$, the result is completely different. For instance, if $\boldsymbol{\zn}=(0,0)\,$, we then obtain $ J_{\zvs}(X)=X$ and
$$
\Gdet_{\zvs}(X)=1-\qj\star \qj=1+\zvs(\degr{\qj},\degr{\qj})\;,
$$
which means that
$$
\Gdet_{\zvs}(X)=
\begin{cases}
0 \quad\text{if}\; \zvs(\qj,\qj)=-1\;,\\
2 \quad\text{if}\; \zvs(\qj,\qj)=1\;.
\end{cases}
$$
Each case happens for half of the choices of $\zvs\in\mathfrak{S}(\zl)\,$.

Second, we work with a rank one matrix
$$
Y:=\lp \begin{array}{cc}1&\qj\\ -\qj&1 \end{array}\rp \in\Mat(2;\qH)\;,
$$
whose kernel is given by $\qH\cdot (1,\qj)\,$.
Again, the matrix $Y$ is homogeneous of degree $0$ as a graded matrix in $\Mat(\boldsymbol{\zn};\mathbb{H})$ if and only if $\boldsymbol{\zn}=(\boldsymbol{\zn}_1,\boldsymbol{\zn}_1+\degr{j})\,$ for some $\boldsymbol{\zn}_1\in\zG$. We get then
$$
\Gdet_{\zvs}(Y)=\Gdet^0(Y)=0\;,
$$
for all $\zvs\in\f{S}(\lambda)\,$. However, if $\boldsymbol{\zn}=(0,0)\,$, we then obtain
$$
\Gdet_{\zvs}(Y)=1-\zvs(\qj,\qj)=
\begin{cases}
2 \quad\text{if}\; \zvs(\qj,\qj)=-1\;,\\
0 \quad\text{if}\; \zvs(\qj,\qj)=1\;.
\end{cases}
$$
As a conclusion, the non-uniqueness of the functions $\Gdet_\zvs$ prevents them to characterize invertible matrices.

%%%%%%%%%%%%%%%%%%%%%%%%%%%%%%
\section{Graded Berezinian}
%%%%%%%%%%%%%%%%%%%%%%%%%%%%%%

In this section, we go back to the general case of a $(\zG,\zl)$-commutative algebra $\cA\,$, for $\zG$ a finitely generate abelian group with non-zero odd part $\zG\odp\,$.
Applying a Nekludova-Scheunert functor $\Ivs$ to $\cA\,$, we obtain a supercommutative algebra $\cuA$.

For matrices with supercommutative entries, the notion of determinant is replaced by the Berezinian, which is a supergroup morphism
\bemaps
\ber:\GL^{\pari{0}}((n,m);\cuA)\to \cuA^{\times}\;,
\eemaps
with $(n,m)\in\N[\Z_2]\,$. The supergroup $\GL^{\pari{0}}((n,m);\cuA)$ of even invertible supermatrices is often written as $\GL(n|m;\cuA)\,$.
Pulling back the Berezinian to the graded case, we obtain the notion of graded Berezinian.

%------------------------------------
\subsection{Preliminaries}
%------------------------------------
Let $\boldsymbol{\zn}=(\boldsymbol{\zn}\evp\,, \boldsymbol{\zn}\odp)\in \zG\evp^{r\evp} \ts~\zG^{r\odp}\odp$, where $(r\evp,r\odp)\in\N^2$. We introduce the group of even homogeneous invertible matrices
$$
\hGL\evp(\boldsymbol{\zn};\cuA)
:=
\bigcup_{\zg\in \zG\evp}\GL^{\zg}(\boldsymbol{\zn};\cuA).
%=\{X\in \GL^x(\vect{r};\cuA) \, |\, x\in \zG\evp\}
$$
Assume $J_\zvs$ is the map \eqref{Js}, associated to a NS-multiplier $\zvs\in\fS(\zl)$.
The composites
\[
\begin{tikzpicture}
\matrix(m)[matrix of math nodes, column sep=3em]
{
\hGL\evp(\boldsymbol{\zn};\cA)
&
\GL(r\evp|r\odp\,;\cuA)
&
(\cuA\evp)^{\ts}
&
(\cA\evp)^{\ts} \\
};
\path[->]
(m-1-1) edge node[below]{$J_{\zvs}$} (m-1-2)
(m-1-2) edge node[below]{$\ber$} (m-1-3)
(m-1-3) edge node[above]{$\sim$} node[below]{$J_{\zvs}^{-1}$} (m-1-4);
\end{tikzpicture}
\]
define a family of maps parameterized by $\zvs\in\f{S}(\zl)$,
\bemaps
\Gber_{\zvs}:& \hGL\evp(\boldsymbol{\zn};\cA) &\to & (\cA\evp)^{\ts}\;.
\eemaps
Considering the last arrow, labeled with $J_{\zvs}^{-1}$, is a matter of taste as this is the identity map.
This only changes the algebra structure and makes clear that the product and inverse involved in Formula \eqref{Gbers}
are taken in $\cA$ and not in $\cuA$.

Any matrix $X\in\hGL\evp(\boldsymbol{\zn};\cA)$ reads as
$$
X=\lp \begin{array}{cc} \cX_{00} & \cX_{01}\\ \cX_{10}& \cX_{11} \end{array} \rp \;,
$$
where, in particular, $\cX_{00}\in \hGL\evp(\boldsymbol{\zn}\evp;\cA\evp)\,$
and $\cX_{11}\in \hGL\evp( \boldsymbol{\zn}\odp;\cA\evp)\,$. The Formula \eqref{Gbers}, giving the graded Berezinian of $X$,
involves the graded determinant of $\cX_{11}$. We define it just like for matrices of even degree $\boldsymbol{\zn}\evp$, via the formula $\Gdet_\zvs=\det\circ J_\zvs$.

\begin{rema}
Let $\zp\in\zG$ and $\boldsymbol{\zp}:=(\zp,\ldots,\zp)\in\zG^{r\odp}$.
The identity map $T_\zp:\Mat(\boldsymbol{\zn}\odp;\cA)\to\Mat(\boldsymbol{\zn}\odp+\boldsymbol{\zp};\cA)$
is a morphism of $\zG$-algebra and of graded $\cA$-module.
From the Formula \eqref{Jsentry}, defining $J_\zvs$, we deduce that
$\Gdet_{\zvs}(\cX_{11})=\zl(x,\pi)^{r\odp}\Gdet_{\zvs}(T_\zp(\cX_{11}))$ for all $\pi\in\zG$.
%In particular, we deduce that the graded determinant $\Gdet_{\zvs}$
%on $\Mat(\boldsymbol{\zn}\odp+\boldsymbol{\zp};\cA)$ %does not depend on $\pi$, and
%\edz{Which properties needed? to be precised.}
%shares the same properties for $\boldsymbol{\zn}\odp+\boldsymbol{\zp}\in(\zG\odp)^{r\odp}$
%or $\boldsymbol{\zn}\odp+\boldsymbol{\zp}\in(\zG\evp)^{r\odp}$
\end{rema}

%------------------------------------
\subsection{Proof of Theorem \ref{Thm:F}}
%------------------------------------

According to  Proposition \eqref{Jvsfun},
the map $J_\zvs$ restricts to $\GL^0(\boldsymbol{\zn}; \cA)$ as a group morphism.
Hence, the map
\bemaps
\Gber^0:=\ber\circ\,J_{\zvs}: \GL^0(\boldsymbol{\zn}; \cA) \to  (\cA)^{\ts}\;
\eemaps
defines  a group morphism. The independence of $\Gber^0$ in $\zvs\in\fS(\zl)$ follows from the Formula \eqref{Gbers} and the equality $\Gdet_\zvs=\Gdet^0$ on  $0$-degree matrices.

It remains to prove Formula \eqref{Gbers}.
Let us consider a homogeneous matrix $X\in \GL^x((\boldsymbol{\zn}\evp\,, \boldsymbol{\zn}\odp);\cA)$ of arbitrary even degree $x\in\zG\evp\,$.
Decomposing $X$ in block matrices with respect to parity, it reads as a supermatrix,
$$
X=(X^i_{\;\,j})_{i,j=1,\ldots,n}=\lp\begin{array}{c|c} \cX_{00}& \cX_{01}\\ \hline \cX_{10} & \cX_{11} \end{array}\rp\;.
$$
Since moding out by the ideal $\cA_{\odp}$ of odd elements preserves invertibility, we deduce that the block $\cX_{11}$ is invertible. This allows for an UDL decomposition of $X$, with respect to the parity block subdivision,
% it is the product of an upper block unitriangular, a block diagonal and a lower block unitriangular  matrices, denoted respectively by $U$, $D$ and $L$ : \edz{product $\cdot$ deleted}
$$
X=UDL=
\lp\begin{array}{c|c} \I & \cX_{01} \cX_{11}^{-1}\\ \hline  & \I \end{array}\rp
%\cdot
\lp\begin{array}{c|c} \cX_{00}- \cX_{01} \cX_{11}^{-1} \cX_{10} \\ \hline  & \cX_{11} \end{array}\rp
%\cdot
\lp\begin{array}{c|c} \I& \\ \hline \cX_{11}^{-1} \cX_{10} & \I \end{array}\rp
\;.
$$
By construction, we also see that the three graded matrices in the decomposition are homogeneous: the block triangular matrices $U$ and $L$ are of degree $0$, whereas the diagonal matrix $D$ is of the same degree as the original matrix $X$.
Thanks to the properties of $J_{\zvs}$, we then have
$$
J_{\zvs}(X)=J_{\zvs}(U) J_{\zvs}(D) J_{\zvs}(L)
$$
which is again a UDL decomposition. Hence, applying the Berezinian, we finally obtain
\be\label{step1}
\ber(J_{\zvs}(X))=\ber(J_{\zvs}(D)),
\ee
since the Berezinian is multiplicative and equal to $1$ on a block unitriangular matrices.
Let us recall that the Berezinian of a block diagonal invertible supermatrix is equal to
$$
\ber\lp \begin{array}{c|c} \cY_{00} & \\ \hline  & \cY_{11} \end{array}\rp
=
\det(\cY_{00}) \det(\cY_{11})^{-1}\;.
$$
Hence, using \eqref{step1}, the properties of $J_{\zvs}$ and the equality $(J_{\zvs})^{-1}=J_{\zd}$ with $\zd(x,y):=\zvs(x,y)^{-1}$, for all $x,y\in\zG$, we get
\beas
\Gber(X)
&=& J_{\zvs}^{-1} \lp \ber\lp J_{\zvs}(X)\rp\rp
\\
&=&
J_{\zvs}^{-1}
\Big(
\det\lp J_{\zvs}\lp  \cX_{00}- \cX_{01} \cX_{11}^{-1} \cX_{10} \rp \rp
\star
\det \lp J_{\zvs}( \cX_{11})\rp^{- 1}
\Big)
\\
&=&
\zvs(r\evp\,x,-r\odp\,x)\,
J_{\zvs}^{-1}
\Big(
\det\lp J_{\zvs}\lp  \cX_{00}- \cX_{01} \cX_{11}^{-1} \cX_{10} \rp \rp \Big)
\cdot
J_{\zvs}^{-1}
\Big(
\det\lp J_{\zvs}( \cX_{11})\rp^{-1}
\Big)
\\
&=&
\zvs(r\evp\,x,-r\odp\,x)\zvs(-r\odp\,x,-r\odp\,x)\,
\Gdet_{\zvs}\lp  \cX_{00}- \cX_{01} \cX_{11}^{-1} \cX_{10} \rp
\cdot
\Gdet_{\zvs}( \cX_{11})^{-1}\; ,
\eeas
which recovers Formula \eqref{Gbers}.

%-----------------------------------------------------------------------------------------------------------
%-----------------------------------------------------------------------------------------------------------
%-----------------------------------------------------------------------------------------------------------

%%%%%%%%%%%%%%%
\appendix \section{Basic Notions of Category Theory}
\setcounter{thm}{0}
\renewcommand{\thethm}{\textsc{\roman{thm}}}
%%%%%%%%%%%%%%%

In this section we recall basic notions of category theory used in the paper. The main references are \cite{McL}, \cite{Hov} and \cite{Kelly}.

%---------------------------
\subsection{Categories}
%---------------------------

A  \emph{locally small} (respectively \emph{small}) \emph{category} $\CC$ consists of
\begin{itemize}
    \item a class (respectively a set) of \emph{objects} $\Obj(\CC)\,$;
    \item for every pair of objects $X,Y\in \Obj(\CC)\,$, a set of \emph{morphisms} $\Hom_{\CC}(X,Y)$ (if no confusion is possible the subscript is usually dropped);
    \item for any triple $X,Y,Z\in \Obj(\CC)\,$, a \emph{composition of morphisms}
    $$ \circ :\, \Hom(Y,Z) \ts \Hom(X,Y) \;\raa\; \Hom(X,Z) $$
\end{itemize}
which satisfy the following axioms
\begin{description}
    \item[Associativity] For any given morphisms $f\in \Hom(Z,W)\,$, $g\in \Hom(Y,Z)$ and $h\in \Hom(X,Y)\,$, the equality $f\circ (g \circ h) = (f\circ g) \circ h$ holds.
    \item[Identity] For every $X\in \Obj(\CC)\,$, there exists a morphism $\id_{\hspace{-0.1cm}\phantom{.}_{X}} \in \Hom(X,X)$ such that, for any morphisms $f\in \Hom(X,Z)$ and $g\in \Hom(Y,X)\,$, the following equalities hold
        \beas
        f\circ \id_{\hspace{-0.1cm}\phantom{.}_{X}} = f & \mbox{ and } & \id_{\hspace{-0.1cm}\phantom{.}_{X}} \circ \,g=g \;.
        \eeas
\end{description}

All the categories encountered in this paper are \emph{concrete categories}, i.e., the objects are sets with additional structure, and morphisms are functions. This translates into the existence of a forgetful functor to the category of sets, denoted by "{\rm forget}".

%---------------------------
\subsection{Functors}
%---------------------------

A \emph{functor} $F:\, \CC \to \DD$ between categories consist of
%two assignments  (usually denoted by the same letter):
\begin{itemize}
    \item an \emph{object function} $F:\,\Obj(\CC) \raa \Obj(\DD)\,$,
    \item for each pair $X,Y \in \Obj(\CC)$, an \emph{arrow function} $ F:\, \Hom_{\CC}(X,Y) \to \Hom_{\DD}(F(X),F(Y))\,$,
\end{itemize}
such that $F(\id_{\hspace{-0.1cm}\phantom{.}_{X}})=\id_{\hspace{-0.1cm}\phantom{.}_{F(X)}}$ for all $X\in \Obj(\CC)\,$, and, when the composition is meaningful, $F(f\circ \,g)=F(f) \circ \,F(g)\,$.

The \emph{composite of two functors} $F:\, \cat{B} \to \cat{C}$ and $G:\,\cat{A}\to\cat{B}$ is a functor $F\circ G:\, \cat{A}\to \cat{C}$ given by usual composition of the corresponding object functions and arrow functions. Composition of functors is associative.

A particular example of functor is the \emph{identity functor} $\I_{\CC}:\, \CC \to \CC$ which assigns to every object, respectively to every morphism of $\CC\,$, itself. It acts as the identity element for the composition of functors.

A functor $F:\,\cat{A} \to \cat{B}$ is said \emph{invertible} if there exists a second functor $G:\, \cat{B} \to \cat{A}$ such that
\beas
F\circ G=\I_{\cat{B}} &  \mbox{ and }& G\circ F=\I_{\cat{A}}\,.
\eeas

%---------------------------
\subsection{Natural Transformations}
%---------------------------

If a functor is intuitively a morphism in the category of (small/locally small) categories, a natural transformation is a morphism between functors.

More precisely, if $F$ and $G$ are two functors between the same categories $\CC$ and $\DD$,  a \emph{natural transformation} $\zh$ between the functors $F$ and $G$ is represented by the diagram
\be \label{nattrans}
\begin{tikzpicture}
\matrix(m)[matrix of math nodes, column sep=6em]
{
\CC & \DD\\
};
\draw[->](m-1-1) to[bend left=40] node[label=above:$F$] (U) {} (m-1-2);
\draw[->] (m-1-1) to[bend right=40,name=D] node[label=below:$ G$] (V) {} (m-1-2);
\draw[double,double equal sign distance,-implies,shorten >=5pt,shorten <=5pt]
  (U) -- node[label=right:$\zh$] {} (V);
\end{tikzpicture}
\ee
and $\zh$ consists of a family of maps
$$
\zh_{\hspace{-0.1cm}\phantom{.}_{X}}:\, F(X) \to G(X) \;, %\qquad X\in \Obj(\CC)
$$
which is natural in $X\in \Obj(\CC)\,$. This means that, for every morphism $h\in \Hom_{\CC}(X,Y)\,$, the following diagram commutes,
\[
\begin{tikzpicture}
\matrix(m)[matrix of math nodes, row sep=3em, column sep =6em]
{
F(X) & F(Y)\\
G(X) & G(Y) \\
};
\path[->]
(m-1-1) edge node[above]{$ F(h) $} (m-1-2)
            edge node[left]{$\zh_{\hspace{-0.1cm}\phantom{.}_{X}} $} (m-2-1)
(m-2-1) edge node[below]{$ G(h) $} (m-2-2)
(m-1-2) edge node[right]{$ \zh_{\hspace{-0.1cm}\phantom{.}_{Y}} $} (m-2-2)
;
\end{tikzpicture}
\]
There are two ways of composing natural transformation: "vertically" or "horizontally".

\subsubsection{Vertical composition of transformations} \label{vertical}
Given two natural transformations
\[
\begin{tikzpicture}
\matrix(m)[matrix of math nodes, column sep=10em]
{
\CC & \DD\\
};
\draw[->](m-1-1) to[bend left=70] node[label=above:$F$] (U) {} (m-1-2);
\draw[->] (m-1-1) to node[below] (V) {$G$} (m-1-2);
\draw[->] (m-1-1) to[bend right=70] node[label=below:$ H$] (W) {} (m-1-2);
\draw[double,double equal sign distance,-implies,shorten >=5pt,shorten <=5pt]
  (U) -- node[label=right:$\zh$] {} (V);
  \draw[double,double equal sign distance,-implies,shorten <=2pt]
  (V) -- node[label=right:$\zvy$] {} (W);
\end{tikzpicture}
\]
their vertical composition is the natural transformation
\[
\begin{tikzpicture}
\matrix(m)[matrix of math nodes, column sep=8em]
{
\CC & \DD\\
};
\draw[->](m-1-1) to[bend left=40] node[label=above:$F$] (U) {} (m-1-2);
\draw[->] (m-1-1) to[bend right=40,name=D] node[label=below:$ H$] (V) {} (m-1-2);
\draw[double,double equal sign distance,-implies,shorten >=5pt,shorten <=5pt]
  (U) -- node[label=right:$\zvy  \bullet  \zh$] {} (V);
\end{tikzpicture}
\]
defined component-wise by $(\zvy  \bullet  \zh)_{\hspace{-0.1cm}\phantom{.}_{X}}:= \zvy_{\hspace{-0.1cm}\phantom{.}_{X}} \circ \zh_{\hspace{-0.1cm}\phantom{.}_{X}}\,$, $X\in \Obj(\CC)\,$.

%Vertical composition mimic the composition of functors we defined in the previous section.
%With respect to this composition, considering all functors between two fixed categories (say $\CC$ and $\DD$) as objects and corresponding natural transformations as morphisms, we obtain a category called the \emph{functor category} $\DD^{\CC}$ (also denoted by $\op{Fun}(\CC,\DD)$).

Vertical composition is the law of composition of the \emph{functor category} $\DD^{\CC}$ (also denoted by $\op{Fun}(\CC,\DD)$), whose objects are functors between $\CC$ and $\DD$ and whose arrows are the natural transformations between such functors. For every functor $F$, the identity map $\Id_{\hspace{-0.1cm}\phantom{.}_{F}}$ is the natural transformation of components $\Id_{\hspace{-0.1cm}\phantom{.}_{F,X}}=\id_{\hspace{-0.1cm}\phantom{.}_{F(X)}}\,$.

\subsubsection{Horizontal composition of transformations}
 Given two natural transformations
\[
\begin{tikzpicture}
\matrix(m)[matrix of math nodes, column sep=8em]
{
\cat{A}& \cat{B} & \cat{C}  \\
};
\draw[->](m-1-1) to[bend left=40] node[label=above:$F$] (U') {} (m-1-2);
\draw[->] (m-1-1) to[bend right=40,name=D] node[label=below:$G$] (V') {} (m-1-2);
\draw[double,double equal sign distance,-implies,shorten >=5pt,shorten <=5pt]
  (U') -- node[label=right:$\zh$] {} (V');
\draw[->](m-1-2) edge [bend left=40] node[label=above:$F'$] (U) {} (m-1-3.north west);
\draw[->] (m-1-2) to[bend right=40,name=D] node[label=below:$ G'$] (V) {} (m-1-3.south west);
\draw[double,double equal sign distance,-implies,shorten >=5pt,shorten <=5pt]
  (U) -- node[label=right:$\zh'$] {} (V);
\end{tikzpicture}
\]
their horizontal composition is the natural transformation
\[\begin{tikzpicture}
\matrix(m)[matrix of math nodes, column sep=8em]
{
\cat{A} & \cat{C}\\
};
\draw[->](m-1-1) to[bend left=40] node[label=above:$F'\circ F$] (U) {} (m-1-2);
\draw[->] (m-1-1) to[bend right=40,name=D] node[label=below:$G'\circ G$] (V) {} (m-1-2);
\draw[double,double equal sign distance,-implies,shorten >=5pt,shorten <=5pt]
  (U) -- node[label=right:$\zh'\circ \zh $] {} (V);
\end{tikzpicture}
\]
defined component-wise by $(\zh'\circ \zh)_{\hspace{-0.1cm}\phantom{.}_{X}}:=\zh'_{\hspace{-0.1cm}\phantom{.}_{G(X)}}\circ F'(\zh_{\hspace{-0.1cm}\phantom{.}_{X}})=G'(\zh_{\hspace{-0.1cm}\phantom{.}_{X}})\circ \zh'_{\hspace{-0.1cm}\phantom{.}_{F(X)}}\,$, for all $X\in \Obj(\cat{A})\,$.

\subsubsection{Whiskering}\label{whisker}
The horizontal composition allows to define the composition of a natural transformation with a functor, also called \emph{whiskering}. Given a natural transformation as in~\eqref{nattrans} and a functor $H: \, \cat{B} \to \CC\,$, their composite
\[
\begin{tikzpicture}
\matrix(m)[matrix of math nodes, column sep=6em]
{
\cat{B} & \cat{C} & \cat{D} \\
};
\path[->] (m-1-1) edge node[above] {$H$} (m-1-2);
\draw[->] (m-1-2) to [bend left=40] node[above](U){$F$}(m-1-3);
\draw[->] (m-1-2) to [bend right=40] node[below](V){$G$}(m-1-3);
\draw[double,double equal sign distance,-implies,shorten >=5pt,shorten <=5pt]
  (U) -- node[label=right:$\zh$] {} (V);
\end{tikzpicture}
\]
is defined as the horizontal composition of natural transformations $\Id_{\hspace{-0.1cm}\phantom{.}_{H}} \circ \;\zh\,$, i.e.
\[
\begin{tikzpicture}
\matrix(m)[matrix of math nodes, column sep=8em]
{
\cat{B}& \CC & \DD \quad  = \quad \cat{B} & \DD \\
};
\draw[->](m-1-2) edge [bend left=40] node[label=above:$F$] (U) {} (m-1-3.north west);
\draw[->] (m-1-2) to[bend right=40,name=D] node[label=below:$ G$] (V) {} (m-1-3.south west);
\draw[double,double equal sign distance,-implies,shorten >=5pt,shorten <=5pt]
  (U) -- node[label=right:$\zh$] {} (V);
\draw[->](m-1-1) to[bend left=40] node[label=above:$H$] (U') {} (m-1-2);
\draw[->] (m-1-1) to[bend right=40,name=D] node[label=below:$H$] (V') {} (m-1-2);
\draw[double,double equal sign distance,-implies,shorten >=5pt,shorten <=5pt]
  (U') -- node[label=right:$\Id_{\hspace{-0.1cm}\phantom{.}_{H}}$] {} (V');
  \draw[->](m-1-3.north east) to[bend left=40] node[label=above:$H \circ  F$] (U'') {} (m-1-4);
\draw[->] (m-1-3.south east) to[bend right=40,name=D] node[label=below:$ H\circ G$] (V'') {} (m-1-4);
\draw[double,double equal sign distance,-implies,shorten >=5pt,shorten <=5pt]
  (U'') -- node[label=right:$\Id_{\hspace{-0.1cm}\phantom{.}_{H}}\!\circ \zh$] {} (V'');
\end{tikzpicture}
\]
Clearly, one can do the same for right composition of $\zh$ with a functor $H':\, \DD \to \cat{A}\,$,
\[
\begin{tikzpicture}
\matrix(m)[matrix of math nodes, column sep=6em]
{
 \cat{C} & \cat{D} &\cat{A} \\
};
\path[->] (m-1-2) edge node[above] {$H'$} (m-1-3);
\draw[->] (m-1-1) to [bend left=40] node[above](U){$F$}(m-1-2);
\draw[->] (m-1-1) to [bend right=40] node[below](V){$G$}(m-1-2);
\draw[double,double equal sign distance,-implies,shorten >=5pt,shorten <=5pt]
  (U) -- node[label=right:$\zh$] {} (V);
\end{tikzpicture}
\]

%-------------------------------
\subsection{Adjoint Functors}%and Equivalence of Categories
%-------------------------------

%Natural transformations permits to define the most/very useful tools/notions in category theory, such as adjoint functors and equivalence of categories.

A functor $F:\,\CC \to \DD$ is the \emph{left-adjoint} of a functor $G:\,\DD\to \CC$ (or equivalently $G$ is the \emph{right-adjoint} of $F$) if there exists a bijection between the hom-sets,
$$
\Hom_{\DD}(F(X),U) \simeq \Hom_{\CC}(X,G(U))
$$
which is natural in both $X\in \Obj(\CC)$ and $U\in\Obj(\DD)$. This means that
\beas
\Hom_{\DD}(F(X),-) \overset{\sim}{\Rightarrow} \Hom_{\CC}(X,G(-))
& \mbox{ and } &
\Hom_{\DD}(F(-),U) \overset{\sim}{\Rightarrow} \Hom_{\CC}(-,G(U))%\\
 %\mbox{ (for $X$ fixed)}\quad\qquad \qquad & &\qquad \qquad \mbox{ (for $U$ fixed) }
 \eeas
 are natural transformations (for $X$ and $U$ fixed respectively).

%\medskip
%Equivalent categories ...{\cmg need it or not?}

%---------------------------
\subsection{Closed Monoidal Categories}
%---------------------------

A \emph{monoidal category} is a (locally small) category $\CC$ endowed with a bifunctor $\ots\,$,
 a unit object $\I\,$, and
 three natural isomorphisms
\beas
\za=\za_{\hspace{-0.1cm}\phantom{.}_{X,Y,Z}}\!\!&\!\!:\!\!&\!\! (X\ots Y)\ots Z \xrightarrow{\sim}{} X\ots(Y\ots Z) \qquad \mbox{(\emph{associator} or \emph{associativity constraint}) }\\
l =l_{\hspace{-0.1cm}\phantom{.}_{X}}\!\!&\!\!:\!\!&\!\! \I \ots X \xrightarrow{\sim}{} X\\
r =r_{\hspace{-0.1cm}\phantom{.}_{X}}\!\!&\!\!:\!\!&\!\!  X \ots \I \xrightarrow{\sim}{} X
\eeas
such that
$l_{\I} = r_{\I}: \I \to \I\,$,
and they satisfy the following coherence laws:
\[
\begin{tikzpicture}
\matrix(m)[matrix of math nodes, row sep=4em, column sep=-1em]
{
 & & &  (X\ots Y)\ots (Z\ots W) && & \\
((X\ots Y)\ots Z) \ots W & & &&& & X\ots (Y\ots (Z\ots W)) \\
  &(X\ots (Y\ots Z))\ots W &  & &  & X \ots ((Y\ots Z)\ots W) & \\
};
\path[->]
(m-2-1) edge node[auto] {$ \za_{\hspace{-0.1cm}\phantom{.}_{X\ots Y,Z,W}} $} (m-1-4)
(m-1-4) edge node[auto] {$ \za_{\hspace{-0.1cm}\phantom{.}_{X,Y,Z\ots W}} $} (m-2-7)
(m-3-6) edge node[below right] {$\id_{\hspace{-0.1cm}\phantom{.}_{X}} \ots \za_{\hspace{-0.1cm}\phantom{.}_{Y,Z,W}} $} (m-2-7)
(m-2-1) edge node[below left] {$ \za_{\hspace{-0.1cm}\phantom{.}_{X,Y,Z}}\ots \id_{\hspace{-0.1cm}\phantom{.}_{W}} $} (m-3-2)
(m-3-2) edge node[below] {$ \za_{\hspace{-0.1cm}\phantom{.}_{X,Y\ots Z,W}} $} (m-3-6)
; \end{tikzpicture}
\]
\[
\begin{tikzpicture}
\matrix(m)[matrix of math nodes, row sep=3em, column sep=2em]
{
(X \ots\I) \ots Y & & X \ots(\I \ots Y) \\
& X \ots Y & \\
};
\path[->]
(m-1-1) edge node[auto] {$ \za_{\hspace{-0.1cm}\phantom{.}_{X,\I,Y}}$} (m-1-3)
    edge node[below left] {$r_{\hspace{-0.1cm}\phantom{.}_{X}} \ots \id_{\hspace{-0.1cm}\phantom{.}_{Y}} $} (m-2-2)
  (m-1-3) edge node[below right] {$\id_{\hspace{-0.1cm}\phantom{.}_{X}} \ots \, l_{\hspace{-0.1cm}\phantom{.}_{Y}} $}(m-2-2)
; \end{tikzpicture}
\]

\medskip
A monoidal category is called \emph{braided} if it is endowed with a natural isomorphism
$$ \zb=\zb_{\hspace{-0.1cm}\phantom{.}_{X,Y}}:\, X\ots Y \xrightarrow{\sim} Y\ots X  \;,$$
(usually called \emph{commutativity constraint} or simply \emph{braiding}) satisfying the following coherence laws:
\[
\begin{tikzpicture}
\matrix(m)[matrix of math nodes, row sep=4em, column sep=2em]
{
 & X\ots (Y\ots Z)  && (Y\ots Z) \ots X & \\
 (X\ots Y)\ots Z  & && & Y \ots (Z \ots X)  \\
 &(Y\ots X)\ots Z &&  Y \ots (X\ots Z) & \\
};
\path[->]
(m-2-1) edge node[auto] {$ \za_{\hspace{-0.1cm}\phantom{.}_{X,Y,Z}} $} (m-1-2)
        edge node[below left] {$ \zb_{\hspace{-0.1cm}\phantom{.}_{X,Y}}\ots \id_{\hspace{-0.1cm}\phantom{.}_{Z}} $} (m-3-2)
(m-1-2) edge node[auto] {$ \zb_{\hspace{-0.1cm}\phantom{.}_{X,Y\ots Z}}  $} (m-1-4)
(m-1-4) edge node[auto] {$ \za_{\hspace{-0.1cm}\phantom{.}_{Y,Z,X}} $} (m-2-5)
(m-3-2)edge node[below] {$ \za_{\hspace{-0.1cm}\phantom{.}_{Y,X,Z}} $} (m-3-4)
(m-3-4) edge node[below right] {$ \id_{\hspace{-0.1cm}\phantom{.}_{Y}} \ots \zb_{\hspace{-0.1cm}\phantom{.}_{X,Z}} $} (m-2-5)
; \end{tikzpicture}
\]
\[
\begin{tikzpicture}
\matrix(m)[matrix of math nodes, row sep=3em, column sep=5em]
{
 \I \ots X & & X \ots \I  \\
& X & \\
};
\path[->]
(m-1-1) edge node[auto] {$ \zb_{\hspace{-0.1cm}\phantom{.}_{\I,X}}$} (m-1-3)
    edge node[below left] {$l_{\hspace{-0.1cm}\phantom{.}_{X}} $} (m-2-2)
  (m-1-3) edge node[below right] {$r_{\hspace{-0.1cm}\phantom{.}_{X}} $}(m-2-2)
; \end{tikzpicture}
\]

If the braiding also satisfies
$$\zb_{\hspace{-0.1cm}\phantom{.}_{Y,X}}\circ \zb_{\hspace{-0.1cm}\phantom{.}_{X,Y}}=\id_{\hspace{-0.1cm}\phantom{.}_{X\ots Y}} \;, $$
$\text{for all } X,Y\in \Obj(\CC)\,$, then $(\CC, \ots, \I,\za,l,r, \zb)$ is a \emph{symmetric monoidal category}.

\medskip
A monoidal category $(\CC,\ots,\I,\za,l,r)$ is a \emph{right} (resp. \emph{left}) \emph{closed monoidal category} if there exists a bifunctor
\beas
\iHom^{right}_{\CC}(-,-):\, \CC\opp \ts \CC  \raa  \CC
& (\mbox{resp.} &
\iHom^{left}_{\CC}(-,-):\, \CC\opp \ts \CC  \raa  \CC \; )\;,
\eeas
such that for every $X \in \Obj(\CC)\,$, the functor $\iHom^{right}_{\CC}(X,-)$ (resp. $\iHom^{left}_{\CC}(-,X)$)
is the right adjoint of the functor   $- \ots X$ (resp. of the functor $X \ots -$).
A \emph{closed monoidal category} is a left and right closed monoidal category.

%In the literature, this last defined notion is sometimes referred to as biclosed monoidal category, leaving the denomination "closed" to indicate right-closedness.

%We note that in the case the monoidal category $\CC$ is additionally symmetric, right closedness implies left closedness and vice-versa, the right and left $\iHom$ being naturally isomorphic (thanks to the braiding). Hence, in this case we identify them and just write $\iHom(-,-)$ as the \emph{internal hom}.

Note that if  a closed monoidal category is  symmetric, the right and left $\iHom$ are naturally isomorphic thanks to the braiding. Hence, we identify them and designate the \emph{internal $\iHom$} by  $\iHom(-,-)\,$.

\medskip

%%%%%%%%%%%%%%%%%%%%%%%%%%%%%%%%%%%%%%%%%%%%%%%%%%%%%%%%%%%%%%%%%%%%%%
%%%%%%%%%%%%%%%%%%%%%%%%%%%%%%%%%%%%%%%%%%%%%%%%%%%%%%%%%%%%%%%%%%%%%%

%%%%%%%%%%%%%%%%%%%%%%%%%%%%%%%%%%%%%%%%%%%%%%%%%%%%%%%%%%%%%%%%%%%
%%%%%%%%%%%%%%%%%%%%%%%%%%%%%%%%%%%%%%%%%%%%%%%%%%%%%%%%%%%%%%%%%%%

\begin{thebibliography}{XX}%{CaractGdet}
%%%%%%%%%%%%%%%%%%%%%%%%%%%%%%%%%%%%%%%%%%%%%%%%%%%%%%%%%%%%%%%%%%%%%%
%%%%%%%%%%%%%%%%%%%%%%%%%%%%%%%%%%%%%%%%%%%%%%%%%%%%%%%%%%%%%%%%%%%%%%

\bibitem{AM}%[AM99]
H. Albuquerque, and S. Majid,
{\it Quasialgebra structure of the octonions}, J. Algebra
{\bf 220} (1999), 188--224.

\bibitem{AM1}%[AM02]
H. Albuquerque, and S. Majid,
{\it Clifford algebras obtained by twisting of group algebras},
J. Pure Appl. Algebra {\bf 171} (2002), 133--148.

\bibitem{Al03}%[Ale03]
S. Alesker,
{\it Non-commutative linear algebra and plurisubharmonic functions of quaternionic variables},
Bull. Sci. Math. {\bf 127} (2003), no. 1, 1--35.

\bibitem{Asl}%[Asl96]
H. Aslaksen,
{\it Quaternionic determinants},
Math. Intelligencer {\bf 18} (1996), 57--65.

\bibitem{Bah}%[BSZ01]
Yu. A. Bahturin, S. K. Sehgal, and M. V. Zaicev,
{\it Group Gradings on Associative algebras},
J. Algebra {\bf 241} (2001), no. 2, 677--698.

\bibitem{Bourbaki}
N. Bourbaki,
{\it  \'El\'ements de math\'ematique. Alg\`ebre. Chapitres 1 \`a 3.} (Reprint of the 1970
    original),
 Springer, 2007.

\bibitem{Cov} %[Cov13]
T. Covolo,
{\it Cohomological Approach to the Graded Berezinian},
preprint \textsf{arXiv:1207.2962},
to appear in J. Noncommut. Geom.

\bibitem{COP} %[COP12]
T. Covolo, V. Ovsienko, and N. Poncin,
{\it Higher Trace and Berezinian of Matrices over a Clifford Algebra},
{ J. Geom. Phys.} {\bf 62} (2012), no. 11, 2294--2319.


\bibitem{Dad80}%[Dad80]
E.~C. Dade,
\newblock {\em Group-graded rings and modules,}
\newblock Math. Z. {\bf 174} (1980), no. 3, 241--262.

\bibitem{DM99}%[DM99]
 P. Deligne, and J. Morgan,
 {\it Notes on Supersymmetry (following J. Bernstein)},
 Quantum Fields and Strings: A Course for Mathematicians, vol. 1, AMS, 1999, 41--97.

\bibitem{Die43}%[Die43]
J. Dieudonn\'e,
\newblock {\em Les d\'eterminants sur un corps non-commutatif,}
\newblock  Bull. Soc. Math. France {\bf 71} (1943), 27--45.


\bibitem{Dys} %[Dys72]
F. J. Dyson,
{\it Quaternion Determinants},
Helv. Phys. Acta {\bf 45} (1972), 289--302.

\bibitem{EK66}%[EK66]
S. Eilenberg, and G. M. Kelly,
{\it Closed categories}, Proc. Conf. Categorical Algebra (La Jolla, Calif., 1965), Springer, 1966, pp. 421-562.

\bibitem{GMW}%[GMW12]
A. de Goursac, T. Masson, and J.-C. Wallet,
{\it Noncommutative epsilon-graded connections},
 J. Noncommut. Geom. \textbf{6} (2012), no. 2, 343--387.

\bibitem{GRo09}%[GR1]
J. Grabowski, and M. Rotkiewicz,
{\it Higher vector bundles and multi-graded symplectic manifolds},
J. Geom. Phys. {\bf 59} (2009), no. 9, 1285--1305.

\bibitem{HWa12}%[HW12]
R.~Hazrat, and A.~R. Wadsworth,
\newblock {\em Homogeneous {${\rm SK}_1$} of simple graded algebras,}
\newblock New York J. Math. {\bf 18} (2012), 315--336.

\bibitem{Hov}%[Hov99]
M. Hovey,
{\it Model Categories},
Mathematical Surveys and Monographs, vol. 63, AMS, Providence, RI, 1999.

\bibitem{Joy98}%[Joy98]
D. Joyce,
{\it Hypercomplex Algebraic Geometry},
 Quart. J. Math. Oxford Ser. (2)  {\bf 49} (1998),  no. 194, 129--162.

\begin{comment}
\bibitem{KDH}%[KDH09]
K. Kanakoglou, D. Daskaloyannis and A. Herrera-Aguilar, \textit{Mixed Paraparticles, Colors, Braidings and a new class of Realizations for Lie superalgebras}, preprint \textsf{arXiv:0912.1070 [math-ph]}.
\end{comment}

\bibitem{Kelly}%[K05]
G. M. Kelly,
{\it Basic concepts of enriched category theory},
Repr. Theory Appl. Categ., no. 10, 2005.

\bibitem{Kna}%[Kna]
A. Knapp,
{\em Basic Algebra, Vol.1},
Springer, 2006, pp. 717.

\bibitem{KNa84}%[KN84]
Y. Kobayashi, and S. Nagamachi,
\newblock {\em Lie groups and {L}ie algebras with generalized supersymmetric
  parameters,}
\newblock J. Math. Phys. {\bf 25} (1984), no. 12, 3367--3374.

 \bibitem{SoS}%[BLM]
D. Leites (ed.), {\it  Seminar on supersymmetry } (v. $1$. Algebra
and Calculus: Main chapters), (J.~Bernstein, D.~Leites, V.~Molotkov,
V.~Shander), MCCME, Moscow, 2011, 410 pp. (in Russian; a version in English
is in preparation but available for perusal).

\bibitem{Lyc}%[Lyc]
V. Lychagin,
{\it Colour calculus and colour quantizations},
Acta Appl. Math. {\bf 41} (1995), 193--226.

\bibitem{McL}%[MacL]
S. Mac Lane,
{\it Categories for the working mathematician},
second ed., Graduate Texts in Mathematics, vol. 5, Springer-Verlag, New York, 1998.

\bibitem{McD} %[McD82]
B.~R. McDonald,
{\it A characterization of the determinant},
Linear Multilinear Algebra \textbf{12} (1982), no. 1, 31--36.

\bibitem{MO}%[MO09]
S. Morier-Genoud, and V. Ovsienko,
{\it Well, Papa, can you multiply triplets?},
Math. Intelligencer
{\bf 31} (2009), 1--2.

 \bibitem{OM}%[MGO10]
S. Morier-Genoud, and V. Ovsienko,
{\it Simple graded commutative algebras},
{ J. Algebra} {\bf 323} (2010), no. 6, 1649--1664.

\bibitem{NOy04}%[NVO04]
C. N{\u{a}}st{\u{a}}sescu, and F. Van~Oystaeyen,
\newblock {\em Methods of graded rings}, Lecture Notes in
  Mathematics, vol. 1836,
\newblock Springer-Verlag, Berlin, 2004.

\bibitem{RW1}%[RW78a]
V. Rittenberg, and D. Wyler,
{\it Generalized superalgebra},
Nuclear Phys. B {\bf 139} (1978), no. 3, 189--202.

\bibitem{RW2}%[RW78b]
V. Rittenberg, and D. Wyler,
{\it Sequences of $(Z_2 \oplus Z_2)$ graded Lie algebras and superalgebras}
J. Math. Phys. {\bf 19} (1978), no. 10, 2193--2200.

\bibitem{Sch} %[Sch79]
M. Scheunert,
{\it Generalized Lie Algebras},
J. Math. Phys. {\bf 20} (1979), no. 4, 712--720.

\bibitem{Sch83}%[Sch83]
M.~Scheunert,
\newblock {\em Graded tensor calculus,}
\newblock J. Math. Phys. {\bf 24} (1983), no. 11, 2658--2670.

\bibitem{Var}%[Var04]
{V.~S. Varadarajan},
{\it Supersymmetry for Mathematicians: An Introduction},
Courant Lecture Notes, vol. 11,
American Mathematical Society, Courant Institute of Mathematical
Sciences, 2004.

\begin{comment}
\bibitem{YJ}%[YJ11]
W. Yang and S. Jing, {\it A new kind of graded Lie algebra and parastatistical supersymmetry}, Sci. China Math. {\bf 44} (2001), no.9, 1167--1173.
\end{comment}
\end{thebibliography}
\end{document}